\documentclass[final]{amsart}

\usepackage{euscript,nicefrac,fancybox,xcolor}
\usepackage[nomargin]{fixme}
\usepackage{xspace,mathrsfs}
\usepackage{inputenc,amssymb}
\usepackage{amscd}
\usepackage[notref,notcite]{showkeys}
\usepackage{todonotes}
\usepackage{xr,xr-hyper}
\usepackage{epsfig,yfonts,tikz-cd}
\usepackage{ifpdf}
\usepackage[margin=1in]{geometry}

\ifpdf
\usepackage{pdfsync,hyperref}

\newtheorem{theorem}{Theorem}[section]

\newtheorem{lemma}[theorem]{Lemma}
\newtheorem{proposition}[theorem]{Proposition}
\newtheorem{corollary}[theorem]{Corollary}

{
	\theoremstyle{plain}
	\newtheorem{definition}[theorem]{Definition}
	\newtheorem{example}[theorem]{Example}
	
	\newtheorem{remark}[theorem]{Remark}
	
}

\newtheorem{theoremalphabetical}{Theorem}

\renewcommand{\theequation}{\arabic{section}.\arabic{equation}}
\newcounter{subeqn}
\renewcommand{\thesubeqn}{\theequation\alph{subeqn}}
\newcommand{\subeqn}{%
	\refstepcounter{subeqn}
	\tag{\thesubeqn}
}
\makeatletter
\@addtoreset{subeqn}{equation}
\newcommand{\newseq}{%
	\refstepcounter{equation}
}

\newcommand{\nc}{\newcommand}

\nc{\PR}{D^i\!R}
\nc{\rola}{X}
\nc{\wela}{Y}
\nc{\Lcat}{\mathsf{DQ}}
\nc{\lcat}{\mathsf{dq}}
\nc{\micro}{\EuScript{R}}
\nc{\microh}{\EuScript{W}}
\nc{\cat}{\mathcal{V}}
\nc{\func}{\EuScript{T}}
\nc{\res}{\operatorname{res}}
\newcommand{\Spec}{\operatorname{Spec}}
\nc{\sK}{\EuScript{K}}
\nc{\sL}{\EuScript{L}}
\nc{\sM}{\EuScript{M}}
\nc{\sQ}{\EuScript{Q}}
\nc{\sP}{\EuScript{P}}
\nc{\tsL}{\tilde{\sL}}
\nc{\structuresheaf}{\mathscr{O}}
\nc{\structuresheafcal}{\mathscr{O}}
\nc{\Qvectorbundle}{\mathscr{Q}}
\newcommand{\A}{\mathscr{A}}
\nc{\mapomodules}{f}
\nc{\lieS}{\operatorname{Lie}(\mathbb{S}}

\newcommand{\vectorspace}{\mathbb{V}}
\newcommand{\Proj}{\operatorname{Proj}}

\newcommand{\lin}{{\ell}}

\nc{\bQ}{\mathbb{N}}
\nc{\op}{\operatorname{op}}

\nc{\sEnd}{\mathcal{E}nd}

\newcommand{\chamber}{\Delta}

\newcommand{\Ba}{\mathbf{a}}

\newcommand{\Bb}{\mathbf{b}}

\newcommand{\HH}{H}

\newcommand{\Bk}{\mathbf{k}}
\newcommand{\Bl}{\mathbf{l}}

\newcommand{\duall}{*}

\newcommand{\Bx}{\mathbf{x}}
\newcommand{\By}{\mathbf{y}}
\newcommand{\Bz}{\mathbf{z}}
\nc{\sz}{\mathsf{z}}
\nc{\sw}{\mathsf{w}}
\nc{\dF}{\mathsf{F}}
\nc{\dE}{\mathsf{E}}
\nc{\lle}{\Gamma}
\nc{\Dc}{d}
\nc{\pS}{S}
\nc{\ps}{\mathsf{s}}

\nc{\pt}{\mathsf{t}}

\nc{\slehat}{\mathfrak{\widehat{sl}}_e}
\nc{\sllhat}{\mathfrak{\widehat{sl}}_\ell}
\nc{\glehat}{\mathfrak{\widehat{gl}}_e}
\nc{\slnhat}{\mathfrak{\widehat{sl}}_n}
\nc{\glnhat}{\mathfrak{\widehat{gl}}_n}
\nc{\eE}{\EuScript{E}}
\newcommand{\arxiv}[1]{\href{http://arxiv.org/abs/#1}{\tt arXiv:\nolinkurl{#1}}}
\nc{\ep}{\epsilon}
\nc{\RHom}{\mathbb{R}\operatorname{Hom}}
\nc{\sHom}{\mathscr{H}\!\mathit{om}}
\nc{\eF}{\EuScript{F}}
\nc{\fF}{\mathfrak{F}}
\nc{\fE}{\mathfrak{E}}
\nc{\fI}{\mathfrak{I}}
\nc{\fp}{\mathfrak{p}}

\newcommand{\K}{\mathbb{K}}

\newcommand{\Z}{\mathbb{Z}}
\newcommand{\Q}{\mathbb{Q}}
\nc{\Qlb}{\mathbb{\bar \Q}_\ell}
\nc{\Fq}{\mathbb{F}_q}
\nc{\Fqb}{\mathbb{\bar F}_q}

\nc{\walg}{W}

\newcommand{\R}{\mathbb{R}}

\newcommand{\Fp}{{\mathbb{F}_p}}

\newcommand{\C}{\mathbb{C}}

\nc{\KZ}{\mathsf{KZ}}

\newcommand{\DQ}{\mathsf{DQ}}
\newcommand{\dq}{\mathsf{dq}}

\newcommand{\la}{\leftarrow}

\newcommand{\ssM}{\EuScript{M}}

\nc{\Bv}{\mathbf{v}}
\nc{\Bw}{\mathbf{w}}
\nc{\perf}{\operatorname{-perf}}
\nc{\Bp}{\mathbf{p}}
\nc{\tU}{\mathcal{U}}
\nc{\Bu}{\mathbf{u}}
\nc{\Fl}{\mathscr{F}\!\ell}
\nc{\Tr}{\operatorname{Tr}}
\nc{\cata}{\mathfrak{V}}
\nc{\tcat}{\tilde{\mathcal{V}}}
\nc{\tcata}{\tilde{\mathfrak{V}}}
\nc{\sheafK}{\EuScript{K}}
\nc{\bmu}{\boldsymbol{\mu}}
\nc{\bpi}{\boldsymbol{\pi}}
\nc{\dwalg}{\mathbb{W}}
\nc{\dalg}{\mathbb{T}}
\nc{\aalg}{\mathbb{A}}
\nc{\alm}{\mathscr{A}}
\nc{\bra}{\mathscr{B}}
\nc{\bO}{\mathbb{O}}

\nc{\Kos}{\EuScript{K}}
\nc{\tilt}{\mathscr{T}}

\renewcommand{\la}{\lambda}

\newcommand{\al}{\alpha}

\newcommand{\Hom}{\operatorname{Hom}}
\newcommand{\Homs}{\mathscr{H}\!om}
\newcommand{\fd}{\mathfrak {d}}
\newcommand{\ft}{\mathfrak {t}}

\newcommand{\cL}{\mathcal{L}}

\newcommand{\ch}{h}

\newcommand{\Ext}{\operatorname{Ext}}

\newcommand{\bS}{\mathbb{S}}

\newcommand{\excise}[1]{}

\newcommand{\End}{\operatorname{End}}

\newcommand{\fM}{\mathfrak{M}}
\newcommand{\fN}{\mathfrak{N}}

\newcommand{\mmod}{\operatorname{-mod}}

\newcommand{\Coh}{\mathsf{Coh}}

\newcommand{\AZ}{\mathbb{A}_\Z}
\newcommand{\AZn}{\mathbb{A}_\Z^n}
\newcommand{\Lotimes}{\overset{L}{\otimes}}
\nc{\iwedge}[1]{\bigwedge\nolimits^{\! #1}}
\nc{\wedgep}[1]{\iwedge{#1}\C^n}
\nc{\fsl}{\mathfrak{sl}}
\nc{\sln}{\mathfrak{sl}_n}

\newcommand{\DD}{\mathbb{D}}
\newcommand{\uu}{\mathbb{U}_1}
\newcommand{\hdef}{\hbar}

\newcommand{\mhm}{\mu\mathsf{m}}
\newcommand{\MHM}{\mu\mathsf{M}}
\newcommand{\tmhm}{\widetilde{\mhm}}
\newcommand{\tMHM}{\widetilde{\MHM}}
\newcommand{\tdq}{\widetilde{\dq}}

\newcommand{\Bet}{\mathfrak{B}}
\newcommand{\Dol}{\mathfrak{D}}
\newcommand{\Dolcov}{\widetilde{\Dol}}
\newcommand{\ZDol}{\mathfrak{Z}}
\newcommand{\basicW}{\mathfrak{W}}
\newcommand{\algcov}{\Dolcov^{\operatorname{alg}}}
\newcommand{\ZDolcov}{\widetilde{\ZDol}}
\newcommand{\DDD}{\mathbf{d}}
\newcommand{\SGmod}{\EuScript{O}_{\phi}^\hdef \mmod^{\mathbb{S}G}}
\newcommand{\Amodule}{\mathscr{M}}

\newcommand{\gradedHalg}{\tilde{H}}
\newcommand{\degradedHalg}{H}
\newcommand{\gradedchambers}{\tilde{\Lambda}}
\newcommand{\degradedchambers}{\Lambda}
\newcommand{\monomialm}{m}

\FXRegisterAuthor{w}{w}{Ben}
\FXRegisterAuthor{m}{m}{Michael}
\FXRegisterAuthor{r}{r}{The referee}
\makeatletter
\renewcommand*\FXLayoutInline[3]{%
	{\@fxuseface{inline}
		
		\ignorespaces\noindent \ovalbox{\hspace{.03\textwidth} \begin{minipage}{.9\textwidth}
				#3 \fxnotename{#1}: #2
			\end{minipage}\hspace{.03\textwidth}}}}
\makeatother

\begin{document}
	
	\renewcommand{\theitheorem}{\Alph{itheorem}}
	
	\usetikzlibrary{decorations.pathreplacing,backgrounds,decorations.markings,calc,
		shapes.geometric}
	\tikzset{wei/.style={draw=red,double=red!40!white,double distance=1.5pt,thin}}
	\tikzset{bdot/.style={fill,circle,color=blue,inner sep=3pt,outer sep=0}}
	\tikzset{dir/.style={postaction={decorate,decoration={markings,
					mark=at position .8 with {\arrow[scale=1.3]{<}}}}}}
	\tikzset{rdir/.style={postaction={decorate,decoration={markings,
					mark=at position .8 with {\arrow[scale=1.3]{>}}}}}}
	\tikzset{edir/.style={postaction={decorate,decoration={markings,
					mark=at position .2 with {\arrow[scale=1.3]{<}}}}}}
	\author{Michael McBreen}
	
	\address{Department of Mathematics, Chinese University of Hong Kong}
	\author{ Ben Webster}
\address{Department of Pure Mathematics, University of Waterloo \& 
	Perimeter Institute for Theoretical Physics}
		\begin{abstract}
			\noindent {\em Abstract.}
			We consider homological mirror symmetry in the context of hypertoric varieties, showing that an appropriate category of B-branes (that is, coherent sheaves) on an additive hypertoric variety matches a category of A-branes on a Dolbeault hypertoric manifold for the same underlying combinatorial data.  For technical reasons, the A-branes we consider are modules over a deformation quantization (that is, DQ-modules).  We consider objects in this category equipped with an analogue of a Hodge structure, which corresponds to a $\mathbb{G}_m$-action on the dual side of the mirror symmetry.  
			
			This result is based on hands-on calculations in both categories.  We analyze coherent sheaves by constructing a tilting generator, using the characteristic $p$ approach of Kaledin; the result is a sum of line bundles, which can be described using a simple combinatorial rule.  The endomorphism algebra $H$ of this tilting generator has a simple quadratic presentation in the grading induced by $\mathbb{G}_m$-equivariance. In fact, we can confirm it is Koszul, and compute its Koszul dual $H^!$.  
			
			We then show that this same algebra appears as an Ext-algebra of simple A-branes in a Dolbeault hypertoric manifold.  The $\mathbb{G}_m$-equivariant grading on coherent sheaves matches a Hodge grading in this category.  
		\end{abstract}

\title[Homological mirror symmetry for hypertoric varieties I]{Homological mirror symmetry for hypertoric varieties I: conic equivariant sheaves}
	
	\maketitle
	\section{Introduction}
	\label{sec:introduction}
	
	Toric varieties have proven many times in algebraic geometry to be a valuable testing ground.  Their combinatorial flavor and concrete nature has been extremely conducive to calculation.  Certainly this is the case in the domain of homological mirror symmetry (see \cite{Aboutoric,FOOO}).  
	
	Toric varieties have a natural hyperk\"ahler analogue which we call {\bf hypertoric varieties} (some other places in the literature, they are called ``toric hyperk\"ahler varieties'').  Just as toric varieties can be written as K\"ahler quotients of complex vector spaces, hypertoric varieties are hyperk\"ahler quotients by tori (see Definition \ref{def:hypertoric}). 
	
	Despite their name, hypertoric varieties are almost never toric. Rather, they are {\bf conical symplectic resolutions:} the natural map $\pi : \fM \to \Spec H^0(\fM, \mathscr{O}_{\fM})$ is a proper resolution of singularities, and there is an action of $\mathbb{G}_m$ on $\fM$ which dilates the algebraic symplectic form and contracts $\Spec H^0(\fM, \mathscr{O}_{\fM})$ to a point $o$. Among symplectic resolutions, hypertoric varieties are distinguished by the presence of an effective complex hamiltonian action of a half-dimensional complex torus.
	
	In this paper, we study homological mirror symmetry for hypertoric varieties. This is typically understood to mean an equivalence between the derived category of coherent sheaves (or $B$-branes) on an algebraic variety  and the Fukaya category (the $A$-branes) of a related symplectic manifold. We will instead prove a different, but closely related, equivalence.
	
	On the $B$-side, we consider the derived category of coherent sheaves on the hypertoric variety $\fM$. For the statement of our equivalence, it is most natural to impose some finiteness conditions.  
 The simplest version of our equivalence concerns the category $\Coh(\fM)_o$ of sheaves set-theoretically supported on the fiber $\pi^{-1}(o)$ over the cone point of $\Spec H^0(\fM, \mathscr{O}_{\fM})$. 
	
	On the $A$-side, we take our mirror space to be a {\bf Dolbeault hypertoric manifold} $\Dol$, as defined by Hausel and Proudfoot. This is a multiplicative analogue of $\fM$, equipped with a fibration $q : \Dol \to \C^d$ by Lagrangian abelian subvarieties degenerating to a union of toric varieties over $0 \in \C^d$. We prove in the sequel paper \cite{GMW}, joint with Ben Gammage, that $q^{-1}(0)$ is the skeleton of a suitable Liouville structure on $\Dol$. When we need to distinguish, we will call usual hypertoric varieties {\bf additive}. 
	
  We define a certain category $\dq$ of deformation quantization modules on $\Dol$, quantizing the irreducible components of $q^{-1}(0) \subset \Dol$. Let $\DQ$ be the (dg enhanced) derived category of $\dq$. We prove:
	
				\begin{theoremalphabetical} \ref{centralcorollary}
						\label{centralcorollary_intro}
					There is an equivalence of dg categories $D^b(\Coh(\fM_\C)_o)\to \Lcat$.
					\end{theoremalphabetical}

     The simples of $\dq$ may be thought of as certain distinguished objects in the Fukaya category of $\Dol$. We do not attempt to make this precise here; the exact relation of $\DQ$ to $Fuk(\Dol)$ is described in \cite{GMW}. 
	
	The left-hand category has an important extra structure : the conical $\mathbb{G}_m$ action. To understand its mirror, we consider an abelian category $\mhm$ consisting of DQ-modules endowed with a `microlocal mixed Hodge
	structure', along with its derived category $\MHM$. We have the following graded version of Theorem \ref{centralcorollary_intro} :
	
		\begin{theoremalphabetical} \ref{cor:maincorollary} \label{cor:maincorollary_intro}
						There is an equivalence of dg categories $D^b(\Coh_{\mathbb{G}_m}(\fM_\Q)_o)\to \MHM,$ such that tensoring with the weight 1 representation of $\mathbb{G}_m$ corresponds to a $1/2$ Tate twist.
					\end{theoremalphabetical}
					This equivalence may be thought of as homological mirror symmetry
	for two subcategories of the $A$ and $B$ branes, both of which are
	enriched with suitable notions of $\mathbb{G}_m$-equivariance. The reader may compare with \cite{BMO, MaOk} and their sequels, where the same $\mathbb{G}_m$-action plays a key role. 
	
	In fact, we construct a family of equivalences, which are best understood in terms of special t-structures on both sides. On the one hand, the Dolbeault space $\Dol$ depends on a choice of parameter $\zeta \in \mathfrak{t}^{\vee}_\R \cong H^2(\fM, \R)$, in the complement of a periodic hyperplane arrangement. As $\zeta$ crosses these hyperplanes, components of the central fiber $q^{-1}(0)$ may appear or disappear. Thus different chambers yield different abelian categories $\mhm$, which are nevertheless derived equivalent. 
	
    On the other hand, a choice of $\zeta$ in the complement of the arrangement determines a {\bf tilting generator} of $D^b(\Coh_{\mathbb{G}_m}(\fM_\Q)_o)$. This is a vector bundle $\EuScript{T}^{\zeta}$ such that $\Ext(\EuScript{T}^{\zeta},-)$ defines an equivalence of dg categories
	\[ D^b(\Coh(\fM))\cong D^b(\degradedHalg^{\zeta} \mmod^{\operatorname{op}}) \]
	where
	$\degradedHalg^{\zeta}=\End(\EuScript{T}^{\zeta})$. In particular, the natural t-structure on the right-hand side defines an `exotic' t-structure on the left-hand side.
	
	Our construction of $\EuScript{T}^{\zeta}$ follows a recipe of Kaledin \cite{KalDEQ}. The algebra $\degradedHalg^{\zeta}$ is thus an analogue in our context of Bezrukavnikov's noncommutative Springer resolution \cite{BezNon}. Its significance can be understood as follows. Both $\fM$ and $\degradedHalg^{\zeta}$ are naturally defined over $\Z$. Given a field $\K$ of characteristic $p$, let $\fM_{\K}$ and $\degradedHalg^{\zeta}_{\K}$ be the corresponding $\K$-forms. Suppose $p\zeta \in H^2(\mathfrak{M};\Z)$, in which case it defines a class $\la \in\operatorname{Pic}(\fM_{\K})$. There is an associated Frobenius-constant quantization of the variety $\fM_{\K}$ in the sense of \cite{BKpos}. We write $A^\la_{\K}$ for the resulting non-commutative algebra, which deforms $H^0(\fM_\K, \mathscr{O}_{\fM})$. By Theorem \ref{thmprojectiveendomorphisms}, there is equivalence of abelian categories between the category of $A^\la_{\K}$-modules with special central character, and  the category of finite dimensional representations of $\degradedHalg^\la_\K$ satisfying a nilpotence condition. 
	
	While this construction springs from geometry in characteristic $p$, and the tilting property is checked using this approach, the tilting generators we consider are sums of line bundles and have a simple combinatorial construction, as does the endomorphism ring $\degradedHalg^{\zeta}$.  This endomorphism ring inherits a grading from a $\mathbb{G}_m$-equivariant structure on $\EuScript{T}^\zeta$ and is Koszul with respect to it.  Thus, the category of $\mathbb{G}_m$-equivariant coherent sheaves on $\fM$ is controlled by the derived category of graded $\degradedHalg^{\zeta}$-modules, or equivalently by graded modules over $(\degradedHalg^{\zeta})^!$, its Koszul dual. It is this Koszul dual that has a natural counterpart on the mirror side. 
	
	Theorem \ref{mainthm-Hodge} of this paper explains the relevance of these structures to our mirror equivalence. It can be paraphrased as follows.
	
\begin{theoremalphabetical} \label{mainthm-Hodge_intro}
Under the equivalence \ref{cor:maincorollary_intro}, the natural $t$-structure on deformation quantization modules on $\Dol$ corresponds to the exotic $t$-structure on coherent sheaves on $\fM$ arising from the tilting bundle $\tilt^{\zeta}$.
\end{theoremalphabetical}				

    There are many directions one can go from here. For instance, it is natural to expect different $t$-structures should fit together into a real variation of stability in the sense of \cite{anno2015stability}, in particular, as predicted by Conjecture 1 of {\it loc.\ cit}.  
	The second author will show this in the more general context of Coulomb branches in \cite{WebcohII}.  
	
	As a result of our use of DQ-modules as a substitute for the Fukaya category, this paper contains little about Lagrangian branes, pseudo-holomorphic disks and other staples of symplectic geometry. The reader may wish to compare with the interesting recent preprint \cite{LZhypertoric}, which appeared a few days before this paper and treats the problem of non-equivariant mirror symmetry for hypertoric varieties from the perspective of SYZ fibrations.
	
The	variety $\fM$ is the Coulomb branch (in the sense of \cite{BFN}) with
	gauge group given by a torus, and that $\Dol$ is expected to
	be a hyperk\"ahler rotation of the K-theoretic version of this
	construction.  Thus, it is natural to consider how these constructions
	can be generalized to that case.  The analogous calculation of a
	tilting bundle with explicit endomorphism ring can be generalized in
	this case, as the second author will shows in \cite{WebcohI}, but it is very difficult to even conjecture the correct category to consider on the A-side.
	
	One key source of interest in hypertoric varieties is that they provide excellent examples of conic symplectic singularities (see
	\cite{BLPWquant,BLPWgco}), which can be understood in combinatorial
	terms.  Considerations in 3-d mirror symmetry \cite{BLPWgco} and
	calculations in the representation theory of its quantization  led
	Braden, Licata, Proudfoot and the second author to suggest that
	hypertoric varieties should be viewed as coming in dual pairs,
	corresponding to Gale dual combinatorial data.  In particular, the
	categories $\mathcal{O}$ attached to these two varieties are Koszul
	dual \cite{GDKD,BLPWtorico}.  An obvious question in this case is how
	the categories we have considered, such as coherent sheaves, can be
	interpreted in terms of the dual variety (they are certainly not
	equivalent or Koszul dual to the coherent sheaves on the dual variety,
	as some very simple examples show).  Some calculations in quantum
	field theory suggest that they are the representations of a vertex
	algebra constructed by a BRST analogue of the hyperk\"ahler reduction,
	but this is definitely a topic which will need to wait for future research.
	
	\section*{Detailed outline of the argument}
	\subsubsection*{Part 1: Coherent sheaves and characteristic p quantizations of the additive hypertoric variety}
	Section \ref{sec:addit-mult-hypert} defines the additive hypertoric variety $\mathfrak{M}$. In Section \ref{subsecquantizations}  we fix a field $\K$ of characteristic p, and review the relation between the quantization of $\mathfrak{M}_\K$, called $A_{\K}^{\la}$, and coherent sheaves on $\mathfrak{M}_\K$. In Section \ref{sec:weight-functors}, we introduce a category of modules $A^\la_{\K}\mmod_o$, along with its graded counterpart $A^\la_{\K}\mmod_o^D$. All these objects depend on a quantization parameter $\la$. In Sections \ref{secisomQ}, \ref{subsecprojrepweight} and \ref{subsecweightsofsimples} we classify the projective pro-objects $P_\Bx$ of $A^\la_{\K}\mmod_o^D$, which also yields a classification of simple objects $L_\Bx$. 
	
	Both projectives and simples are indexed by the chambers of a periodic
	hyperplane arrangement $\textgoth{A}^{\operatorname{per}}_\la$ defined
	in \ref{defAhyperplane}. We compute the endomorphism algebra
	$\bigoplus_{\Bx,\By\in \gradedchambers}\Hom( P_\Bx,P_\By)$ in Theorem \ref{presentation1}. The latter contains a
	ring of power series $\widehat{S}$ as a central subalgebra, and we
	define a variant $\gradedHalg_\K^\la$ (Definition \ref{definingH}) in which
	$\widehat{S}$ is replaced by the corresponding polynomial ring $S$. We
	find that $A^\la_{\K}\mmod_o^D$ is equivalent to the subcategory of
	$\gradedHalg_\K^\la$-modules on which $S$ acts nilpotently. 
	
	The algebra $\gradedHalg^\la_\K$ has a natural lift to $\Z$, written $\gradedHalg^\la_\Z$, which we will use to compare with characteristic zero objects on the mirror side. Corollary \ref{cor:KoszulityofH} shows that $\gradedHalg_\Z^\la$ is Koszul. We compute the Koszul dual algebra $\gradedHalg_{\la, \K}^! = \bigoplus_{\Bx,\By\in \gradedchambers}\Ext(L_\Bx,L_\By)$ (Definition \ref{defHdual} and Theorem \ref{thm:HandHdualaredual}). 
	
	In Section \ref{sec:degrading} we describe the ungraded category $A^\la_{\K}\mmod_o$ in terms of the graded one. Its simples and projectives are indexed by the toroidal hyperplane arrangement $\textgoth{A}^{\operatorname{tor}}_\la$ obtained as the quotient of $\textgoth{A}^{\operatorname{per}}_\la$ by certain translations. We describe the corresponding algebras $\degradedHalg^\la_{\K} = \bigoplus_{\Bx,\By\in \degradedchambers}\Hom( P_\Bx,P_{\By})$ and $\degradedHalg^!_{\la, \K} = \bigoplus_{\Bx,\By\in \degradedchambers} \Ext( L_\Bx,L_{\By})$, where the sums now range over simples (resp projectives) for $A^\la_{\K}\mmod_o$. 
	
	In Section \ref{sec:tilting-generators} we use the above results to produce a tilting bundle $\tilt^\la$ on $\mathfrak{M}$ with endomorphism ring $\End(\tilt^\la) = \degradedHalg^\la$. Passing to characteristic zero, and replacing $\la$ by a parameter $\zeta \in \mathfrak{t}^*_\R$, we obtain equivalences  (from Corollary \ref{cor:H-equivalence} and Proposition \ref{prop:Hbang-coh}, respectively):  \begin{equation}
		D^b(\Coh(\fM_\Q))\cong
		D^b(\degradedHalg_\Q^{\zeta, \operatorname{op}}\operatorname{-mod})\qquad \degradedHalg_{\zeta, \Q}^!\perf\cong D^b(\Coh(\fM_\Q)_{o})
	\end{equation}
	where $\degradedHalg_{\zeta, \Q}^!\perf$ is the category of perfect dg-modules over this ring.
	\begin{remark}
		Throughout, we will always endow the bounded derived category $D^b$ of an abelian category with its usual dg-enhancement using injective resolutions; thus if we write $D^b(\mathsf{a})\cong \mathsf{C}$ for an abelian category $\mathsf{a}$ and  a dg-category $\mathsf{C}$, we really mean that this dg-enhancement is quasi-equivalent to $\mathsf{C}$.
	\end{remark}
	
	\subsubsection*{Part 2: Deformation quantization and microlocal mixed Hodge modules on the Dolbeault hypertoric manifold}
	
	The second half of our paper begins with a definition of the Dolbeault
	hypertoric manifold $\Dol$ (\ref{def:dolbhyper}), depending on a
	moment map parameter $\zeta$. The complex manifold $\Dol$ is a complex integrable system, with a `central fiber' consisting of a collection of complex Lagrangian submanifolds $\mathfrak{X}_\Bx$ indexed by the chambers of a toroidal hyperplane arrangement $\textgoth{B}^{\operatorname{tor}}_\zeta$ (Definition \ref{def:torBarrange} and Proposition \ref{prop:corecharact}). 
	
	The universal cover $\Dolcov$ of the Dolbeault space is an infinite type complex symplectic manifold, whose geometry is described by a periodic hyperplane arrangement $\textgoth{B}^{\operatorname{per}}_\zeta$. In turn, $\Dolcov$ is an open submanifold of an infinite type algebraic symplectic variety $\algcov$. The latter has a key additional structure: an action of a torus $\mathbb{S} \cong \mathbb{C}^*$ dilating both the complex symplectic form and the base of the integrable system, and preserving the central fiber. 
	
	In Section \ref{sec:defquant}, we define a sheaf  $\EuScript{O}^\hdef_\phi$ of $\C((\hbar))$-algebras on $\Dol$ quantising the structure sheaf, and for each $\mathfrak{X}_\Bx$ we define a module $\sL_\Bx$ over $\EuScript{O}^\hdef_\phi$ supported on $\mathfrak{X}_\Bx$.
	
	Although $\mathbb{S}$ does not preserve $\Dolcov \subset \algcov$, we can nevertheless make sense of $\mathbb{S}$-equivariant DQ modules on $\Dolcov$ and $\Dol$, and we show that $\sL_\Bx$ has a natural $\mathbb{S}$-equivariant structure. 
	
	 We define a subcategory $\lcat$ of $\mathbb{S}$-equivariant $\EuScript{O}^\hdef_\phi$-modules on $\Dol$ generated by the simple DQ-modules $\sL_\Bx$, together with the category dg-category $\Lcat$ of complexes in $\lcat$. The $\mathbb{S}$-equivariance yields a category with $\C$ (rather than $\C((\hbar))$) coefficients. We write $\widetilde{\dq}$ and $\widetilde{\DQ}$ for the corresponding categories on $\widetilde{\Dol}$.
	
	When $\la$ is the reduction of $p\zeta$, the arrangements $\textgoth{A}^{\operatorname{tor}}_\la$ and $\textgoth{B}^{\operatorname{tor}}_\zeta$ are identified. We hence have a bijection of chambers, and a corresponding bijection of isomorphism classes of simple objects for the categories $\lcat$ and $A^\la_{\K}\mmod_o$. Moreover, Theorem \ref{mainthm} shows that the $\Ext$ algebras of the simples in both categories share a common integral form: $\degradedHalg^!_{\la,\C}\cong E^!_{\C}:=\bigoplus_{\Bx,\By\in \degradedchambers}  \Ext(\sL_\Bx,\sL_\Bx)$.  
	
	Unfortunately, some care is needed about concluding that this isomorphism induces an equivalence of categories $\Lcat\to D^b(\mathsf{Coh}(\fM))_o$, since {\it a priori} it is not clear that $ E^!_{\C}$  is formal as a dg-algebra, which we would need to define a fully-faithful functor.  We prove this equivalence by constructing projective objects in $\dq$, and showing that $\degradedHalg_{\la,\C}$ appears as their automorphism algebra.  This shows that we have the desired derived equivalence (Corollary \ref{centralcorollary}).

	We can further account for the grading on $\degradedHalg_{\la,\Q}$ and reduce the structure ring to $\Q$ from $\C$ by considering a new graded abelian category $\mhm$ (Definition \ref{def:mhm}), and a corresponding triangulated category $D^b(\MHM)$. Each object of $\mhm$ is a $\EuScript{O}^\hdef_\phi$-module, such that for each lagrangian $\mathfrak{X}_\Bx$, the restriction to a Weinstein neighborhood of $\mathfrak{X}_\Bx$ is equipped with the structure of a mixed Hodge module. These structures are required to be compatible in a natural sense whenever two components intersect. We define $\mhm$ as the category generated by a special collection of such objects.
	
	Each object $\sL_\Bx$ has a natural lift to $\mhm$, and moreover any simple object of $\mhm$ is isomorphic to such a lift. 
	This allows us to conclude that the equivalence $D^b_{\operatorname{perf}}(\Coh(\fM_\C)_o)\to \Lcat$ can be upgraded to an equivalence of graded categories $D^b_{\operatorname{perf}}(\Coh_{\mathbb{G}_m}(\fM_\Q)_o)\to \MHM$ in the spirit of equivariant mirror symmetry.
	
	\begin{remark}
		In an earlier version of this paper, the proof of the main result depended on the use of this Hodge structure. In revisions responding to a referee's comments, we found a proof that avoids the use of it, so we have moved all discussion of Hodge topics to Section \ref{sec:Hodge}, after the proof of Theorem \ref{centralcorollary}.  We have left the discussion of Hodge structures in the paper, since we believe it is of some interest in understanding how $\C^*$-actions translate through mirror symmetry.
	\end{remark}
	
	\section*{Acknowledgements}
	We would like to thank Andrei Okounkov for suggesting a multiplicative analogue of the hypertoric variety as a mirror, Tam\'as Hausel and Nick Proudfoot for sharing their unpublished results on multiplicative hypertoric varieties, and Roman Bezrukavnikov, Ben Gammage, Sam Gunningham, Paul Seidel, Vivek Shende and Michael Thaddeus for helpful conversations. We would also like to thank the anonymous referee for many helpful comments. The first author performed part of this work at the Massachusetts
	Institute of Technology, the \'Ecole Polytechnique F\'ed\'erale de
	Lausanne, and during the Junior Trimester on Symplectic Geometry and
	Representation Theory at the Hausdorff Research Institute for
	Mathematics. He gratefully acknowledges the hospitality of all these
	institutions. While at the EPFL, he was supported by the Advanced Grant ``Arithmetic and Physics of Higgs moduli spaces'' No. 320593 of the European Research Council. This work was also partly supported by the Simons Foundation, as part of a Simons Investigator award. The second author was supported during the course of
	this work by the NSF under Grant DMS-1151473, the Alfred Sloan
	Foundation and by Perimeter Institute for Theoretical
	Physics. Research at Perimeter Institute is supported by the
	Government of Canada through the Department of Innovation, Science and
	Economic Development Canada and by the Province of Ontario through the
	Ministry of Research, Innovation and Science.
	
	\section{Hypertoric enveloping algebras}
	\label{sec:hypert-envel-algebr}
	
	\subsection{Additive hypertoric varieties}
	\label{sec:addit-mult-hypert}
	
	For a general introduction to hypertoric varieties, see \cite{Pr07}.
	
	Consider a split algebraic torus $T$ over $\Z$ of dimension $k$ (that is, an algebraic
	group isomorphic to $\mathbb{G}_m^k$) and a faithful linear action
	of $T$ on the affine space  $\AZn$, which we may assume is diagonal in the
	usual basis.  We let $D\cong \mathbb{G}_m^n$ be the group of diagonal  matrices in this basis, and write $G :=D/T$. 
	
	We have an induced action of $T$ on the cotangent bundle $T^*\AZn\cong \AZ^{2n}$.  We'll
	use $\sz_i$ for the usual coordinates on $\AZn$, and $\sw_i$ for the dual
	coordinates.  This
	action has an algebraic moment map $\mu\colon T^*\AZn\to \ft^*_\Z$, defined by a map of
	polynomial rings $\Z[\ft_\Z] \to \Z[\sz_1,\dots,\sz_n,\sw_1,\dots,\sw_n]$
	sending a cocharacter $\chi$ to the sum $\sum_{i=1}^n\langle
	\epsilon_i,\chi\rangle \sz_i\sw_i$, where $\epsilon_i$ is the character on
	$D$ defined by the action on the $i$th coordinate line, and $\langle
	-,-\rangle$ is the usual pairing between characters and cocharacters
	of $D$.
	
	For us, the main avatar of this action is the {\bf
		(additive) hypertoric variety}.  This is an algebraic hamiltonian reduction of $T^*\AZn$ by $T$. It comes in affine and smooth
	flavors, these being the categorical and GIT quotients (respectively)
	of the scheme-theoretic fiber $\mu^{-1}(0)$ by the group $T$.  More
	precisely, fix a character $\alpha\colon T\to \mathbb{G}_m$ whose kernel
	does not fix a coordinate line.  
	\begin{definition}\label{def:hypertoric}
		For a commutative ring $\K$, we let \[\fN_\K:=\Spec (\K[\sz_1,\dots,\sz_n,\sw_1,\dots, \sw_n]^T/\langle
		\mu^*(\chi)\mid\chi\in \ft_\Z\rangle)\]
		and 
		\[\fM_\K :=\Proj(\K[\sz_1,\dots,\sz_n,\sw_1,\dots, \sw_n,t]^T/\langle
		\mu^*(\chi)\mid\chi\in \ft_\Z\rangle)\]
		where $t$ is an additional variable of degree 1 with $T$-weight
		$-\alpha$.  
	\end{definition}
	
	Both varieties carry a residual action of the torus $G = D/T$, and an additional commuting action of a rank one torus $\bS := \mathbb{G}_m$ which scales the coordinates $\sw_i$ linearly while fixing $\sz_i$.
	
	We say that the sequence $T \to D \to G$ is {\bf unimodular} if the image of any tuple of coordinate cocharacters in $\fd_\Z := \operatorname{Lie}(D)_\Z$ forming a $\Q$-basis of $\mathfrak{g}_{\Q} := \operatorname{Lie}(G)_\Q$ also forms a $\Z$-basis of $\mathfrak{g}_\Z$. 
	
	Let $\pi\colon \fM_\C \to \fN_\C$ be the natural map. If we assume unimodularity, then $\fM_\C$ is a smooth scheme and $\pi$ defines a proper $T \times \bS$-equivariant resolution of singularities of $\fN_\C$. Together with the algebraic symplectic form on $\fM_\C$ arising from Hamiltonian reduction, this makes $\fM_\C$ a {\bf symplectic resolution}. Many elements of this paper make sense in the broader context of symplectic resolutions, although we will not press this point here. In the non-unimodular case, $\fM_\C$ may have orbifold singularities.

	In the description given above, $\fN_\C$ appears as the Higgs branch
	of the $\mathcal{N}=4$ three-dimensional gauge theory attached to the
	representation of $T_\C$ on $\C^n$.  However, it is more natural from
	the perspective of what is to follow to see $\fN_\C$ as the Coulomb
	branch of the theory attached to the dual action of $(D/T)^\vee$ on
	$\C^n$, in the sense of Braverman--Finkelberg--Nakajima \cite{BFN, NaCoulomb}.  This leads to a different presentation of the hypertoric enveloping algebra, which will be useful for understanding its representation theory.  In particular, the multiplicative hypertoric varieties we'll
	discuss later appear naturally from this perspective as the Coulomb
	branches of related 4 dimensional theories.  
	
	\subsection{Quantizations} \label{subsecquantizations}
	\label{sec:quantizations}

	The ring of functions on the hypertoric variety $\fN_\Z$ has a quantization which we call the {\bf hypertoric enveloping
		algebra}. We construct it by a quantum analogue of the Hamiltonian reduction which defines $\fM_{\Z}$. Consider the
	Weyl algebra $W_n$ generated over $\Z$ by the elements $z_1,\dots,
	z_n,\partial_1,\dots,\partial_n$ modulo the relations:
	\[[z_i,z_j]=0\qquad [\partial_i,\partial_j]=0\qquad
	[\partial_i,z_j]=\delta_{ij}.\]  
	It is a quantization of the ring of functions on $T^*\AZ^n$. The torus $D$ acts on $W_n$, scaling $z_i$ by the character $\epsilon_i$ and $\partial_i$ by $\epsilon_i^{-1}$. It thus determines a decomposition into weight spaces
	$$ W_n = \bigoplus_{\Ba \in \Z^n} W_n[\Ba].$$  Let 
	\[\ch_i^+ :=z_i\partial_i \qquad \ch_i^- :=\partial_iz_i=\ch_i^++1\qquad \ch_i^{\operatorname{mid}}: =\frac{1}2(\ch_i^++\ch_i^-)=\ch_i^++\frac{1}{2}=\ch_i^--\frac{1}{2}.\] Each of the tuples $\ch_i^+$, $\ch_i^-$ and $\ch_i^{\operatorname{mid}}$ generate the same subalgebra, i.e. the $D$-fixed subalgebra $\Z[\ch_i^\pm] = W_n[0]$.  
	
	Via the embedding $T \to D$, $W_n$ carries an action of the torus $T$. To this action one can associate a {\bf non-commutative moment map}, i.e. a map $\mu_q\colon \Z[\ft_\Z] \to W_n$ such that $[\mu_q(\chi),-]$ coincides with the action of the Lie algebra
	$\ft_\Z$. This property uniquely determines $\mu_q$ up to the addition of a character in $\ft_\Z^*$. We make the following choice.
	\[\mu_q(\chi):=\sum_{i=1}^n\langle
	\epsilon_i,\chi\rangle \ch_i^+.\]  It's worth nothing that in the formula above, we have broken the symmetry between $z_i$ and $\partial_i$; it would arguably be more natural to use $\ch_i^{\operatorname{mid}}$, but this requires inserting a lot of annoying factors of $1/2$ into formulas, not to mention being a bit confusing in positive characteristic.  
	
	\begin{definition} \label{defHET}
		The { hypertoric enveloping
			algebra} $A_\Z$ is the subring $W_n^T \subset W_n$ invariant under $T$.
		We'll also consider the central quotients of this algebra
		associated to a character $\la\in \ft^*_\Z$, given by 
		\[A_\Z^\la:=A_\Z/\langle \mu_q(\chi)-\la(\chi)\mid\chi\in \ft_\Z\rangle.\]
	\end{definition}
	We will often abbreviate ``hypertoric enveloping algebra'' to HEA.
	
	Let $A_{\K}:=A_\Z\otimes_{\Z}\K$ be the base change of this algebra
	to a
	commutative ring $\K$.  The algebra $A_\C$ was studied extensively in
	\cite{BLPWtorico,MVdB}.  The algebra $A_{\K}$ when $\K$ has
	characteristic $p$ was studied in work of
	Stadnik \cite{Stadnik}.  Fix a field $\K$ of characteristic $p$ for
	the rest of the paper.
	
	Unlike $W_n$ itself, or its base change to a characteristic 0
	field, the ring $W_n\otimes_\Z \Fp$ has a ``big center'' generated by
	the elements $z_i^p,\partial_i^p$. This central subring can be identified
	with the function ring  $H^0(X^{(1)}, \structuresheaf_{X^{(1)}})$ where $X = T^* \mathbb{A}^n_{\Fp}$.

	\subsection{Coulomb presentation}
	
	The algebra $A_{\K}$ has a different presentation which is more compatible with the subalgebra $\K[\ch_i^{\pm}]$. The action of $D$ on $A_{\K}$ determines a decomposition into weight subspaces. Since $A_{\K} = W_n^T$, its weights lie in $\ft_\Z^\perp = \mathfrak{g}_\Z^*$:
	$$ A_{\K} = \bigoplus_{\Ba \in \ft_\Z^\perp} A_{\K}[\Ba]. $$
	
	For each $\Ba\in \ft_\Z^\perp$, we let \begin{equation} \label{defminimalel} m(\Ba):=\prod_{a_i>0}z_i^{a_i}\prod_{a_i<0}\partial_i^{-a_i}.\end{equation}  Up to scalar multiplication, this is the unique element in $A_{\K}[\Ba]$ in of minimal degree. 
	
	Each weight space $A_{\K}[\Ba]$ is a module over the $D$-invariant subalgebra generated by the $\ch_i^+$. Let:
	\[[h_i]^{(a)} :=\begin{cases}
		1 & a=0\\
		z_i^a\partial_i^a=(h_i^--1)(h_i^--2)\cdots (h_i^--a) & a>0\\
		\partial_i^{-a}z_i^{-a} =(h_i^++1)(h_i^++2)\cdots (h_i^+-a)& a<0
	\end{cases}\]
	
	\begin{theorem}[\mbox{\cite[(6.21b)]{BDGH}}]
		The algebra $A_{\K}$ is generated by $\K[h_1^{\pm}, \dots, h_n^{\pm}]$ and $m(\Ba)$ for $\Ba\in \mathfrak{g}_\Z^*$, subject to the relations:
		\begin{align}
			(h_i^{\pm}-a_i) m(\Ba)&=m(\Ba)h_i^{\pm}\label{eq:Coulomb1}\\
			m(\Ba)m(\Bb)&=\prod_{\substack{a_ib_i<0\\ |a_i|\leq |b_i|}} [h_i]^{(a_i)}\cdot m(\Ba+\Bb)\cdot \prod_{\substack{a_ib_i<0\\ |a_i|> |b_i|}} [h_i]^{(-b_i)}\label{eq:Coulomb2}
		\end{align}
	\end{theorem}
	We call this is the {\bf Coulomb} presentation, since it matches the
	presentation of the abelian Coulomb branch in \cite[(4.7)]{BFN}, and
	shows that the algebra $A_{\K}$ can also be realized using this dual
	approach.  As mentioned in the introduction, the techniques of this
	paper generalize to Coulomb branches with non-abelian gauge group as well, whereas it seems very challenging to generalize them
	to Higgs branches with non-abelian gauge group (that is, hyperk\"ahler
	reductions by non-commutative groups).  
	\subsection{Characteristic \texorpdfstring{$p$}{p} localization}
	\label{sec:char-p-local}
	Following \cite{Stadnik}, in this section we exploit the large center of quantizations in characteristic $p$ to relate modules over $A^{\la}_{\K}$ with coherent sheaves on $\fM_{\K}^{(1)}$. Roughly speaking, upon restriction to fibers of $\pi: \fM_{\K}^{(1)} \to \fN_{\K}^{(1)}$, the quantization becomes the algebra of endomorphisms of a vector bundle, and thus Morita-equivalent to the structure sheaf of the fiber. 
	\begin{theorem} [\mbox{\cite[Thms. 4.3.1 \& 4.3.4]{Stadnik}}]
		For any $\la\in \ft^*_{\mathbb{F}_p}$, there exists a coherent sheaf $\A^\la$ of
		algebras Azumaya over the structure sheaf on
		$\fM_{\K}^{(1)}$ such that $\Gamma(\fM_{\K}^{(1)},\A^\la)\cong
		A^\la_{\K}$.  
	\end{theorem}
	This theorem includes the existence of an injection
	$\HH^0(\fN_{\K}^{(1)}, \structuresheaf_{\fN_{\K}^{(1)}}) \to A^\la_{\K}$; this is induced by the map
	$\HH^0((T^*\mathbb{A}^n_{\K})^{(1)}, \structuresheaf_{(T^*\mathbb{A}^n_{\K})^{(1)}}) \to W_n\otimes{\K}$ sending
	\[\sz_i\mapsto z_i^p\qquad \sw_i\mapsto \partial_i^p.\]

	Consider the moment map $\mu\colon \fM_{\K}^{(1)}\to
	\mathfrak{d}_{\K}^{(1)}$.  Work of Stadnik shows that the
	Azumaya algebra $\A^\la$ splits on fibers of this map after field extension. Fix $\xi\in
	\mathfrak{d}_{\K}^{(1)}$.  Possibly after extending $\K$, we can choose $\nu$ such that
	$\nu^p-\nu=\xi$, and define the splitting bundle as the
	quotient $
	\A^\la /\sum_{i=1}^n\A^\la (\ch_i^+-\nu_i)$; this left module is
	already 
	supported on the fiber $\mu^{-1}(\xi)$, since
	\[(\ch_i^+-\nu)^p-(\ch_i^+-\nu)=z_i^p\partial_i^p-\nu^p+\nu=z_i^p\partial_i^p-\xi.\]
	
	We can thicken this to the formal neighborhood of the fiber $\mu^{-1}(\widehat{\xi})$ by
	taking the inverse limit $\Qvectorbundle_{\nu} :=\varinjlim \A^\la /\sum_{i=1}^n\A^\la (\ch_i^+-\nu_i)^N$.
	
	\begin{theorem}[\mbox{\cite[Thm. 4.3.8]{Stadnik}}]\label{frob-iso}
		The natural map $$\A^\la|_{\mu^{-1}(\widehat{\xi})}\to \End_{\structuresheaf_{\fM^{(1)}}}(\Qvectorbundle_{\nu},\Qvectorbundle_{\nu})$$ is an isomorphism.  
	\end{theorem}
	
	The sheaf $\A^\la$ is not globally split; it has no global
	zero-divisor sections.  It still has a close relationship with
	a tilting vector bundle on $\fM^{(1)}_{\K}$.  We'll fix our attention on the case where $\xi=0$, so $\nu_i\in \mathbb{F}_p$. 
	
	Let $\tilt_{\K}$ be a
	$\bS$-equivariant locally
	free coherent sheaf on $\fM^{(1)}_{\K}$ such that 
	$\tilt_{\K}|_{\mu^{-1}(\hat{0})}\cong \Qvectorbundle_{\nu}$.  Such a sheaf exists by \cite[Theorem 1.8(ii)]{KalDEQ}. 
	As coherent sheaves, we have isomorphisms \[\A^\la_{\K}\cong
	\operatorname{Fr}_*\structuresheaf_{\fM_{\K}}\cong 
	\tilt_{\K}\otimes \tilt^*_{\K}\cong \End(\tilt_{\K}).\]
	
	By  \cite[4.4.2]{Stadnik}, these sheaves have vanishing higher
	cohomology.  Furthermore, combining with results of Kaledin \cite[1.4]{KalDEQ},
	this shows that:
	\begin{proposition}
		For $p$ sufficiently large and $\nu_i$ generic, the sheaf $\tilt_{\K}$
		is a tilting generator on $\fM_\K$ and has a lift $\tilt_\Q$
		which is a tilting generator on $\fM_\Q$; that is,
		$\Ext^i(\tilt_\Q,\tilt_\Q)=0$ for $i>0$, and $\Ext^i(\tilt_\Q,\mathscr{F})=0$ implies
		$\mathscr{F}=0$ for any coherent sheaf on $\fM_\Q$.
	\end{proposition}
	We will later calculate the sheaf $\tilt_{\K}$, once we understand
	$\A^\la_{\K}$ a bit better.
	
	\section{The representation theory of hypertoric enveloping algebras}
	\label{sec:repr-theory-heas}

	\subsection{Module categories and Weight functors}
	\label{sec:weight-functors}
	Recall that we have a short exact sequence of tori $T \to D \to G$. $A^\la_\K$ is a quotient of $W_n^T$, and thus carries a residual action of $G$, which we will now use to study its modules. 
	
	Let $o \in \fN^{(1)}$ be the point defined by $z_i = w_i = 0$, i.e. the unique $\K$-valued $\bS$-fixed point of $\fN^{(1)}$. The following category will play a central role in this paper.
	\begin{definition}
		Let $A^\la_{\K}\mmod_o$ be the category of finitely generated $A^\la_{\K}$-modules which are set-theoretically supported at $o$ when viewed as modules over $\HH^0(\fN_{\K}^{(1)}, \structuresheaf_{\fN_{\K}^{(1)}})$.
	\end{definition}  
	In fact, we will first study the following closely related category.
	\begin{definition} 
		Let $A^\la_{\K}\mmod_o^D$ be the category of modules in $A^\la_{\K}\mmod_o$ which are additionally endowed with a compatible $D$-action, such that $T$ acts via the
		character $\la$, and the action of
		$\Dc_i\in \mathfrak{d}_\Z$ satisfies
		\begin{equation}
			(\ch_i^+-\Dc_i)^Nv=0\qquad N\gg 0. \label{eq:nilpotent}
		\end{equation}
	\end{definition} 
	The difference $\ps_i=\ch_i^+-\Dc_i$ acts centrally on
	such a module, since the adjoint action of $\ch_i^+$ on $A^\la_{\K}$
	agrees with the action of $\Dc_i$.  The operator $\ps_i$ is thus the
	nilpotent part of the Jordan decomposition of $h_i^{\pm}$. The operators $\ps_i$ define an action of the polynomial ring $U_{\K}(\fd)$, which factors through $U_{\K}(\mathfrak{g})$ since elements of $\ft$ act by zero.  This extends to an action of the completion of $U_{\K}(\mathfrak{g})$, since $\ps_i$ acts nilpotently by \eqref{eq:nilpotent}.  
	\begin{definition}
		Let $S := U_{\K}(\mathfrak{g})$, and let $\widehat{S}$ be its completion at zero.
	\end{definition}
	
	Let $\mathfrak{g}^{*,\la}_\Z\subset \fd^*_\Z$ be the $\mathfrak{g}_\Z^*$-coset of characters of $D$
	whose restriction to $T$ coincides with $\la$.  It indexes the $D$-weights which can occur in an object of $A^\la_{\K}\mmod_o^D$.  
	
	We can construct projectives objects in a slight enlargement of $A^\la_{\K}\mmod_o^D$ by working with the
	exact functors picking out weight spaces.  That is, for each $\Ba\in
	\mathfrak{g}^{*,\la}_\Z$, we consider the functor which associates to an object $M \in A^\la_{\K}\mmod_o^D$ the following vector space:
	\[W_{\Ba}(M) :=\{m\in M\mid m\text{ has $D$-weight $\Ba$}\}.\]
	Even though we are working in characteristic $p$, the
	$D$-weights are valued in $\mathfrak{g}^{*,\la}_\Z \subset \Z^n$.  This functor is exact, and we will
	show that it is pro-representable.  
	
	\subsection{Projectives representing the weight functors} \label{subsecprojrepweight}
	To construct the projective object that represents this functor, we
	consider the filtration of it by \[ W_{\Ba}^N(M) :=\{m\in W_{\Ba}(M)\mid
	(\ch_i^+-a_i)^Nm=0\text{ for all }i\}.\]
	
	\begin{proposition}
		We have a canonical isomorphism \[W_{\Ba}^N(M)\cong
		\Hom_{A^\la_{\K}\mmod_o^D}\big(A^\la_{\K}\big/\sum_{i=1}^nA^\la_{\K}(\ch_i^+-a_i)^N,M\big)\]
		where $D$ acts on $A^\la_{\K}\big/\sum_{i=1}^nA^\la_{\K}(\ch_i^+-a_i)^N$ so that the image $1_{\Ba}$ of $1$ has
		weight $\Ba$.   
	\end{proposition}
	Since $W_\Ba(M) =\varinjlim W^N_\Ba(M)$, we have that $W_{\Ba}(M)$ is
	represented by the module 
	\begin{equation} \label{defQa} Q_\Ba :=\varprojlim
		A^\la_{\K}/A^\la_{\K}(\ch_i^+-a_i)^N \end{equation}
	with its induced $D$-action.   Note that
	$Q_\Ba =\Gamma(\mu^{-1}(\hat{0});\Qvectorbundle_\Ba)$.
	
	This is endowed with the usual induced topology, and it is a pro-weight module in the sense that its weight spaces are pro-finite dimensional.  This is a projective object in the category $\widehat{A^\la_{\K}\mmod}^D$ of complete topologically finitely generated $A^\la_{\K}$-modules $M$ with compatible $D$-action in the sense that  \begin{equation*}
		\lim_{N\to \infty}(\ch_i^+-\Dc_i)^Nv=0.
	\end{equation*} That is, $\ps_i$ acts topologically nilpotently on each
	$D$-weight space.  This is  equivalent to \eqref{eq:nilpotent}  if the topology on $M$ is discrete.

	In the arguments below, $\Hom$ and $\End$ will be interpreted to mean continuous homomorphisms compatible with $D$; all objects in $A^\la_{\K}\mmod_o^D$ will be given the discrete topology, so continuity is a trivial condition for homomorphisms between them.

	\begin{lemma} \label{weightspaceofproj}
		If $\Bb$ is a character of $D/T$, then
		$W_{\Ba}(Q_{\Ba+\Bb})\cong \widehat{\pS}$. Otherwise, this weight space is 0.  
	\end{lemma}
	
	\begin{proof}
		For any character $\Bb$ of $D$ which vanishes on $T$, the
		$\Bb$ weight-space $A^\la_{\K}[\Bb]$ is a free rank one module over $\pS$ (acting
		via multiplication by $\ps_i$), generated by $m(\Bb)$. Thus, the $\Ba+\Bb$ weight space of
		$ A^\la_{\K}/A^\la_{\K}(\ch_i^+-a_i)^N$ is generated by $m(\Bb)$,
		subject to the relations
		\[\ps_i^N m(\Bb)\cdot 1_{\Ba}= m(\Bb)\ps_i^N\cdot 1_{\Ba}=m(\Bb)
		(\ch_i^+-a_i)^N\cdot 1_{\Ba}=0\] and is thus free over the
		quotient ring $\pS/\sum \pS\cdot \ps_i^N$.  Taking the inverse limit, we
		see that every weight space of $Q_{\Ba}$ is a free module of rank 1
		over $\widehat{\pS}$.  
	\end{proof}
	\begin{corollary}
		We have an isomorphism of rings
		$\End(Q_\Ba)\cong W_{\Ba}(Q_{\Ba})\cong  \widehat{\pS}$. Since $\widehat{\pS}$
		is local, the module $Q_{\Ba}$ is indecomposable (in the category $\widehat{A^\la_{\K}\mmod}^D$). 
	\end{corollary}
	
	\subsection{Isomorphisms between projectives} \label{secisomQ} 
	In this section, we determine the distinct isomorphism classes of weight functors, i.e. we determine all isomorphisms between the pro-projectives $Q_{\Ba}$. As we will see, there are typically many distinct weights $\Ba \in \mathfrak{g}^{*,\la}_\Z$ that give isomorphic functors. 
	
	By the results of the previous section, the space $W_{\Ba}(Q_{\Ba+\Bb})=\Hom(Q_{\Ba},Q_{\Ba+\Bb})$ is free of rank one over $\widehat{S}$, with generator $m(\Bb)$. Likewise, $\Hom(Q_{\Ba+\Bb}, Q_{\Ba})$ is generated by $m(-\Bb)$. Thus in order to verify whether $Q_\Ba$ and $Q_{\Ba + \Bb}$ are isomorphic, it is enough to check whether the composition $m(-\Bb)m(\Bb)$, viewed as an endomorphism of $Q_\Ba$, is an invertible element of the local ring $End(Q_\Ba) \cong \widehat{S}$.

	By \eqref{eq:Coulomb2}, we have that 
	\begin{equation*}
		m(-\Bb)m(\Bb)=\prod_{i=1}^n [h_i]^{(-b_i)}.
	\end{equation*}
	where the right-hand side is a product of factors of the form
	$h_i^++k$ with $k$ an integer between $\frac{1}{2}$ and $b_i+\frac{1}{2}$. To
	check whether $h_i^+ + k$ defines an invertible element of $\widehat{S}$, it is enough to compute its action on the weight-space of weight $\Ba$, on which $h_i$ acts by $a_i + \ps_i$. The resulting endomorphism $h_i^++k= \ps_i+(a_i+k)$ is invertible if and only if $k+a_i\not\equiv 0\pmod p$. 
	
	The number of non-invertible factors (each equal to $\ps_i)$ in $[h_i]^{(-b_i)}$ is therefore the number of
	integers $k$ divisible by $p$ lying between $a_i + 1/2$ and $a_i + b_i
	+ 1/2$. We denote it by $\delta_i(\Ba,\Ba+\Bb)$.

	We can sum up the above computations as follows. Let \[q(y,k)=
	\begin{cases}
		1 & k=0\\
		\frac{1}{y+k}& k\neq 0\\
	\end{cases}\]  where $y$ is a formal variable and $k\in \K$. Note that $q(\ps_i, a_i +j)(h_i^+ +j)$ acts on a $D$-weight space of weight $\Ba$ by $1$ if $a_i+j$ is not
	divisible by $p$ and by $\ps_i$ if it is. Let \begin{equation} \label{defcmorphism} c^{\Bb}_{\Ba}=m(\Bb) \prod_{i=1}^n
		\prod_{j=1}^{b_i}q(\ps_i,a_i+j) \in
		W_{\Ba}(Q_{\Ba+\Bb})=\Hom(Q_{\Ba},Q_{\Ba+\Bb}).\end{equation} It is a generator of the $\widehat{S}$-module $W_{\Ba}(Q_{\Ba+\Bb})$. Note that this expression
	breaks the symmetry between positive and negative; if $b_i\leq 0$ for
	all $i$, then $c^{\Bb}_{\Ba}=m(\Bb)$, since all the products in the
	definition are over empty sets.  
	
	\begin{lemma}\label{c-reverse}
		\[c^{-\Bb}_{\Ba+\Bb} c^{\Bb}_{\Ba}=\prod_{i=1}^n \ps_i^{\delta_i(\Ba,\Ba+\Bb)}\]
	\end{lemma}
	\begin{proof}
		We have
		\begin{align*}
			c^{-\Bb}_{\Ba+\Bb} c^{\Bb}_{\Ba}&=m(-\Bb) \cdot \prod_{i=1}^n
			\prod_{j=1}^{-b_i}q(\ps_i,a_i+b_i+j) \cdot m(\Bb) \cdot \prod_{i=1}^n
			\prod_{j=1}^{b_i}q(\ps_i,a_i+j)\\
			&=m(-\Bb) m(\Bb) \prod_{i=1}^n
			\left(\prod_{j=1}^{-b_i}q(\ps_i,a_i+b_i+j) \cdot 
			\prod_{j=1}^{b_i}q(\ps_i,a_i+j)\right)\\
			&=\prod_{i=1}^n[h_i]^{(-b_i)}\prod_{j=1}^{-b_i}q(\ps_i,a_i+b_i+j) 
			\prod_{j=1}^{b_i}q(\ps_i,a_i+j)
		\end{align*}
		Note that for each index $i$ only one of the products is non-unital, depending on the sign, and in either case,
		we obtain the product of $q(\ps_i,a_i+j)$ ranging over integers lying between $a_i + 1/2$ and $a_i + b_i
		+ 1/2$.  As we noted earlier, $[h_i]^{(-b_i)} $ is the product of $h_i+j$ with $j$
		ranging over this set.  Thus, we obtain the product over this same set
		of $(h_i+j)q(\ps_i,a_i+j)$, which is precisely $\ps_i^{\delta_i(\Ba,\Ba+\Bb)}$.
	\end{proof}

	It remains for us to describe which pairs $\Ba, \Ba'$ satisfy $\delta_i(\Ba,\Ba') = 0$ for all $i$ and thus index isomorphic projective modules. 
	\begin{definition} \label{defAhyperplane}
		Let  $\textgoth{A}^{\operatorname{per}}_\la$ be the periodic hyperplane arrangement in $\mathfrak{g}^{*,\la}_\Z$ defined by the hyperplanes $\Dc_i=kp-1/2$ for $k\in \Z$ and $i = 1,.., n$.
	\end{definition}
	By definition, $\delta_i(\Ba,\Ba')$ is the minimal number of hyperplanes $\Dc_i = kp-1/2$ crossed when travelling from $\Ba$ to $\Ba'$. 
	Given $\Bx \in \Z^n$, let  \[\chamber_\Bx=\{\Ba\in \mathfrak{g}^{*,\la}_\Z\mid px_i \leq a_i <
	px_i+p\},\qquad \Delta^\R_{\Bx}=\{\Ba\in \mathfrak{g}^{*,\la}_\Z\otimes \R \mid px_i \leq a_i <
	px_i+p\}.\]  
	We have shown
	\begin{theorem}
		We have an isomorphism $Q_\Ba\cong Q_{\Ba'}$ if and only if
		$\Ba,\Ba'\in \chamber_\Bx$ for some $\Bx$.
	\end{theorem}
	Let \[\gradedchambers(\la)= \{\Bx\in \Z^n\mid \chamber_\Bx\neq \emptyset\}\qquad \gradedchambers^\R(\la)= \{\Bx\in \Z^n\mid \chamber_\Bx^\R\neq \emptyset\}.\] 
	
	Thus, $\gradedchambers(\la)$ canonically parametrizes the set of indecomposable
	projective modules in the pro-completion of $A^\la_{\K}\mmod_o^D$.  It follows
	that $\gradedchambers(\la)$ also canonically parametrizes the simple modules in
	this category.  
	
	Let us call the parameter $\la$ {\bf smooth} if there is a
	neighborhood $U$ of $\la$ in $\R\otimes \mathfrak{g}^{*,\la}_\Z$ such that for
	all $\la'\in U$, we have $\gradedchambers(\la)=\gradedchambers^\R(\la')$. In
	particular, if $\la$ is smooth, then the hyperplanes in
	$\textgoth{A}^{\operatorname{per}}_\la$ must intersect
	generically. 
	
	\subsection{A taxicab metric}
	We can endow $\gradedchambers(\la)$ with a metric given by the taxicab
	distance  $|\Bx-\By|_1=\sum_{i}|x_i-y_i|$ for all $\Bx,\By\in
	\gradedchambers(\la)$.  We can add a graph structure to $\gradedchambers(\la)$ by
	adding in a pair of edges between any two chambers satisfying $|\Bx-\By|_1=1$; generically, this is the same as requiring that $\chamber^\R_\Bx$ and $\chamber^\R_\By$ are adjacent across a hyperplane.  
	
	We say that this adjacency is {\bf across
		$i$} if $\Bx,\By$ differ in the $i$th coordinate.  For every $\Bx$, let $\alpha(\Bx)$ be the
	set of neighbors of $\Bx$ in $\gradedchambers(\la)$.  Generically, this is the same as the
	number of facets of $\chamber^\R_\Bx$; we let $\alpha_i(\Bx)\subset
	\alpha(\Bx)$ be those facets adjacent across $i$.  Note that in some
	degenerate cases, we may have that
	$\Bx,\Bx+\epsilon_i,\Bx-\epsilon_i\in \gradedchambers(\la)$ so the size of
	$\alpha_i(\Bx)$ is typically $0$ or $1$, but could be $2$.
	
	\subsection{Weights of simple modules} \label{subsecweightsofsimples}
	
	\begin{definition}
		For any $\Bx\in \Z^n$ such that $\chamber_\Bx\neq
		\emptyset$, we let $P_\Bx :=Q_{\Bb}$ for some $\Bb\in \chamber_\Bx$.
	\end{definition}

	\begin{lemma}
		The module $P_\Bx$ has a unique simple quotient  $L_{\Bx}$, and $L_{\Bx}$ for $\Bx\in \Z^n$ such that $\chamber_\Bx\neq
		\emptyset$ are a complete irredundant list of simple modules in  $A^\la_{\K}\mmod_o^D$.
		
		Furthermore,  the $\Ba$-weight space of $L_{\Bx}$ is 1-dimensional if $\Ba\in
		\chamber_\Bx$ and 0 otherwise.
	\end{lemma}
	\begin{proof}
		We show that $Q_{\Ba}$ has a unique simple quotient by showing the sum of two proper submodules is proper; this then shows that there is a unique maximal proper submodule, and $L_{\Bx}$ is the quotient by it.  A submodule $M\subset Q_{\Ba}$ is proper if and only if $W_{\Ba}(M)\subset W_{\Ba}(Q_{\Ba})\cong \widehat{\pS}$ is a proper submodule, that is, if it lies in the unique maximal ideal $\mathfrak{m} \subset \widehat{\pS}$.  This shows that the sum of two proper submodules is proper, and so $L_{\Bx}$ is well-defined.
		
		Using the isomorphism $Q_{\Ba}\cong Q_{\Bb}$ if $\Ba, \Bb\in \chamber_{\Bx}$, we can extend this to the observation that a submodule $M\subset P_{\Bx}$ is proper if and only if   $W_{\Ba}(M)\subset \mathfrak{m} W_{\Ba}(P_{\Bx})$ for all $\Ba\in \chamber_{\Bx}$.
		
		By Lemma \ref{c-reverse}, we can check that there is a unique submodule $M$ in $P_{\Bx}$ such that
		\[W_{\Ba}(M)=\begin{cases}
			\mathfrak{m} W_{\Ba}(P_{\Bx})& \Ba\in\chamber_\Bx\\
			W_{\Ba}(P_{\Bx}) & \Ba\notin\chamber_\Bx.
		\end{cases}\]
		By the observation above, this must be the maximal proper submodule, so $L_{\Bx}=P_{\Bx}/M.$ This shows that $L_{\Bx}$ has the claimed dimensions of weight spaces.  Furthermore, this shows that we can recover the set $\chamber_{\Bx}$ for $L_{\Bx}$, so we must have $L_{\Bx}\ncong L_{\By}$ if $\Bx\neq \By$.  
		
		For any simple $L$, we must have $W_{\Ba}(L)\neq 0$ for some $\Ba$.  This induces a map $P_{\Bx}\to L$ where $\Ba\in \chamber_\Bx$. Since $L_{\Bx}$ is the unique simple quotient of $P_{\Bx}$, this shows that $L_{\Bx}\cong L$.  This shows that they give a complete list and completes the proof.
	\end{proof}
	
	\begin{example} \label{classicexample}
		An interesting example to keep in mind is the following. Let $T$ be the scalar
		matrices acting on $\mathbb{A}^3$.  In this case, $n=3,k=1$.  The
		space $\mathfrak{g}^{*,\la}_\Z$ is an affine space on which $\Dc_1,\Dc_2$ give a set of
		coordinates, with $\Dc_3$ related by the relation $\Dc_3=-\Dc_1-\Dc_2+\la$
		for some $\la\in \Z$.  Thus, the hyperplane arrangement that
		interests us is given by 
		\[\Dc_1=kp-1/2 \qquad \Dc_2=kp-1/2 \qquad -\Dc_1 -\Dc_2+\la=kp-1/2\]
		In particular, we have that $\chamber_\Bx\neq \emptyset$ if and only if, there
		exist integers $a_1,a_2$ such that 
		\[x_1p\leq a_1 <x_1p+p\quad  x_2p\leq a_2 <x_2p+p \qquad x_3p\leq 
		-a_1-a_2+\la <x_3p+p\] The values of $-a_1-a_2+\la$ for $a_1,a_2$
		satisfying the first two inequalities range from
		$-(x_1+x_2+2)p+2+\la$ to $-(x_1+x_2)p+\la$.  Thus, $x_3$ is a
		possibility if  $-x_1-x_2-2+\lfloor\frac{\la+3}{p}\rfloor x_3\leq
		-x_1-x_2+\lfloor\frac{\la}{p}\rfloor$.  Thus, there are 3 such $x_3$
		if $\lfloor\frac{\la+3}{p}\rfloor=\lfloor\frac{\la}{p}\rfloor$, that is, if $\la\not\equiv -1,-2\pmod
		p$.  If $\la\equiv -1,-2\mod p$, then there are 2, and the parameter
		$\la$ is not smooth.
		
		Of course, the numbers $-1$ and $-2$ have another 
		significance in terms of $\mathbb{P}^2$: the line bundles
		$\structuresheaf(-1)$ and $\structuresheaf(-2)$ on $\mathbb{P}^2$ are the
		unique ones that have trivial pushforward.  This is not coincidence.  Let
		$\lambda_+$ be the unique integer in the range $0\leq \la_+<p$
		congruent to $\la\pmod p$ and $\la_-$ the unique such integer in
		$-p\leq \la_-<0$.   
		The simples $(x_1,x_2, -x_1-x_2+\lfloor\frac{\la}{p}\rfloor)$ and
		$(x_1,x_2,-x_1-x_2-2+\lfloor\frac{\la}{p}\rfloor)$ over $A^\la_{\K}$ can be
		identified with $\HH^0(\mathbb{P}^2;\structuresheaf(\lambda_+))$ and $\HH^1(
		\mathbb{P}^2;\structuresheaf(\lambda_-))$.  If $\la\cong -1,-2\mod p$,
		then the latter group is trivial, so one of the simple representations
		is ``missing.''  
		
		Note that the final simple can be identified with the first cohomology of the kernel of the map $\structuresheaf(\lambda_+)^{\oplus 3}\to \structuresheaf(\lambda_++p)$ defined by $(z_1^p,z_2^p,z_3^p)$ (in characteristic $p$, this is a map of twisted D-modules); this map is surjective on sheaves, but {\em injective} on sections, with the desired simple module its cokernel.    
		
		Let's assume for simplicity that $0\leq \la \leq p-3$.  In this case,
		the ``picture'' of these representations when $p=5$ and $\lambda=1$ is as follows:
		
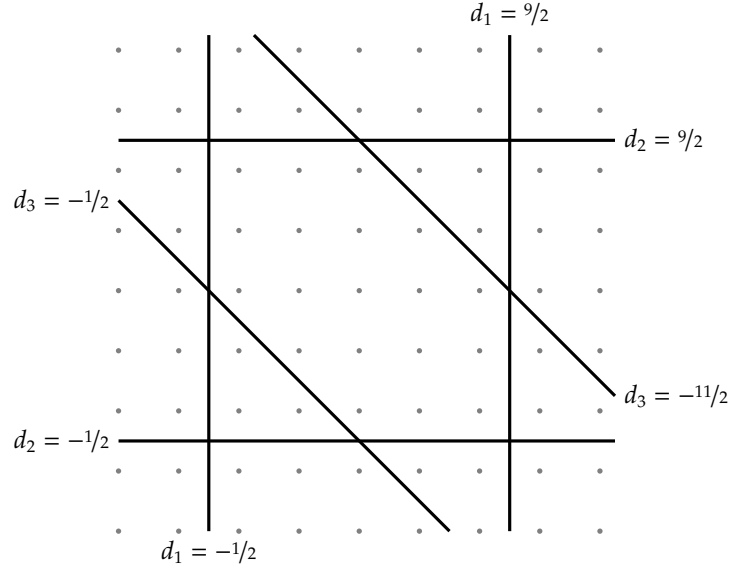
\begin{equation}\label{eq:P2-chambers}
			\begin{tikzpicture}[very thick,scale=2]
				\draw (-.2,-.8) -- node[below, at start,scale=.8]{$\Dc_1=-\nicefrac{1}{2}$} (-.2,2.5);
				\draw (1.8,-.8) -- node[above, at end,scale=.8]{$\Dc_1=\nicefrac{9}{2}$}
				(1.8,2.5);
				\draw (-.8,-.2) -- node[left, at start,scale=.8]{$\Dc_2=-\nicefrac{1}{2}$} (2.5,-.2);
				\draw (-.8,1.8) -- node[right, at end,scale=.8]{$\Dc_2=\nicefrac{9}{2}$}
				(2.5,1.8);
				\draw (-.8,1.4) -- node[left, at start,scale=.8]{$\Dc_3=-\nicefrac{1}{2}$} (1.4,-.8);
				\draw (.1,2.5) -- node[right, at end,scale=.8]{$\Dc_3=-\nicefrac{11}{2}$} (2.5,.1);
				\foreach \x in {-.8,-.4,0,...,2.4}
				\foreach \y in {-.8,-.4,0,...,2.4}
				\fill[color=gray] (\x,\y) circle (.5pt);;
			\end{tikzpicture}
		\end{equation}
		The three chambers shown (read SW to NE) are $\chamber_{(0,0,0)}$,
		$\chamber_{(0,0,-1)}$, $\chamber_{(0,0,-2)}$.  
	\end{example}

	\subsection{The endomorphism algebra of a projective generator}
	Having developed this structure theory, we can easily give a
	presentation of our category. For each $\Bx,\By$ with $\chamber_\Bx\neq
	\emptyset$ and $\chamber_\By\neq \emptyset$, we can define $c_{\Bx,\By}$ to
	be $c^{\Ba'-\Ba}_{\Ba}$ for $\Ba\in \chamber_\By,\Ba'\in \chamber_\Bx$.  For each
	$i$, let $\eta_i(\Bx,\By,\Bu)=\frac12 (|x_i-y_i|+|y_i-u_i|-|x_i-u_i|)$.

	\begin{theorem}\label{presentation1}
		The algebra $\bigoplus_{\Bx, \By\in \gradedchambers}\Hom(P_{\Bx}, P_{\By})$ is generated by the
		idempotents $1_\Bx$ and the elements
		$c_{\Bx,\By}$ over $\widehat{\pS}$ modulo the relation:
		\begin{equation}
			c_{\Bx,\By}c_{\By,\Bu}=\prod_i \ps_i^{\eta_i(\Bx,\By,\Bu)}c_{\Bx,\Bu}\label{general-relation}
		\end{equation}
	\end{theorem}
	
	Note that this relation is homogeneous if $\deg c_{\Bx,\By}=|\Bx-\By|_1$ and $\deg \ps_i=2$. 
	
	\begin{proof} The relation holds by an easy extension of Lemma
		\ref{c-reverse}.  To see that these elements and relations are sufficient, note that in the algebra $\widehat{H}$ with this
		presentation, the Hom-space $1_{\Bx}\widehat{H} 1_{\By}$ is cyclically generated over $\widehat{S}$
		by $c_{\Bx,\By}$.  The image of $c_{\Bx,\By}$ under induced map $1_{\Bx} \widehat{H} 1_{\By}\to
		\Hom(P_\By,P_\Bx)$ generates the target space over $\widehat{S}$. Since the target is free of rank 1 as a $\widehat{\pS}$-module, the map must be an isomorphism.
	\end{proof}
	
	\begin{definition} \label{definingH}
		Let $S_\Z := U_\Z (\mathfrak{g})$. Let $\gradedHalg^{\la}_{\Z}$ be the graded algebra over $S_\Z$ generated by $1_{\Bx}$ and $c_{\Bx,\By}$ with presentation given in Theorem \ref{presentation1}. Let
		$\gradedHalg^{\la}_{\K} := \gradedHalg^{\la}_{\Z}\otimes \K$.
	\end{definition}
	This algebra is isomorphic to its opposite via the anti-isomorphism which acts by the identity on $S_\Z$ and $c_{\Bx,\By}\mapsto c_{\By,\Bx}$.  Since this algebra has a left action on the sum  $\bigoplus_{\Bx} P_{\Bx}$, it naturally has a {\it right} action on $\bigoplus_{\Bx}\Hom(P_{\Bx},M)$ for any $A^\la_{\K}$-module, which we will turn into a left module structure using the anti-automorphism above.  
	
	It may concern the reader that $\gradedHalg^{\la}_{\K}$ is not a unital algebra, but it has a structure which can serve as a replacement.  Note that we follow the terminology and notation of \cite{Brundan-Davidson} in this section.  We call a $\K$-algebra $A$ {\bf locally unital} if there are idempotents $1_{\alpha}$  indexed by some set $\aleph$ such that $A=\oplus_{\alpha,\beta\in \aleph}1_{\alpha}A1_{\beta}$.  
	\begin{definition}\label{def:P-category}
		Given a locally unital algebra $A$, let $\mathscr{P}(A)$ be the category where the objects are the set $\aleph$, and the morphism spaces are given by $\Hom(\alpha,\beta)=1_{\alpha} A 1_{\beta}.$ 
	\end{definition}
	Note that $\mathscr{P}$ is equivalent to the subcategory of left projective modules with objects $A1_{\alpha}$.
	
	We call a module $M$ over $A$  {\bf locally unital} if $M=\oplus_{\alpha\in \aleph}1_{\alpha}M$;  note that this is automatic if $\K$ is a field and $M$ is finite dimensional over $\K$. We can think of $\alpha\mapsto 1_{\alpha}M$ as a functor $\mathscr{P}^{\operatorname{op}}\to \K\mmod$, and conversely, every locally unital left $A$-module arises from a unique such functor.  In particular, the results of  \cite{MOS}, which are formulated in terms of representations of categories, also apply to locally unital algebras.

	The algebra $\gradedHalg^{\la}_{\K} $ is may not be left or right Noetherian as a ring, since it is not finitely generated as a module over itself.  However, it can be {\bf locally left Noetherian}\footnote{This is not identical to the notion of ``locally Noetherian'' found in scheme theory, but is related: the spectrum of a commutative locally Noetherian ring will be a possibly infinite disjoint union of Noetherian schemes, which is thus locally Noetherian as a scheme.}, meaning that left submodules of $A1_{\Bx}$ are finitely generated.
	\begin{proposition}
		If $\K$ is Noetherian, then  the algebra  $\gradedHalg^{\la}_{\K}$ is locally left Noetherian.
	\end{proposition}
	\begin{proof}
		Consider a submodule $U\subset \gradedHalg^{\la}_{\K} 1_{\Bx}$.  The intersection $U\cap 1_{\By}\gradedHalg^{\la}_{\K}1_{\Bx}$ must be of the form $I_{\By}c_{\By,\Bx}$ for some ideal $I_{\By}\subset S_{\K}$.  Note also that since $c_{\Bz,\By}I_{\By}c_{\By,\Bx}\subset I_{\Bz}c_{\Bz,\Bx}$, we have $I_{\By}\subset I_{\Bz}$ if for each $i$, either $x_i\leq y_i\leq z_i$ or $z_i\leq y_i\leq x_i$.   For any subset $B$ of $\Z_{\geq 0}^n$, there is a finite list of points $b^{(1)},\dots, b^{(r)}$ such that for any $b\in B$, there is some $r$ such that $b_i\geq b_i^{(r)}$ for all $i$.  This means that for any finitely generated ideal $I\subset S_{\K}$ and any subset $B\subset \gradedchambers$, the submodule generated by $Ic_{\By,\Bx}$ for $\By\in B $ is finitely generated, since it is generated by $Ic_{\By,\Bx}$ for finitely many choices of $\By$.  Since $\K$ is Noetherian, so is $S_{\K}$, and thus every ideal in $I$ is finitely generated.  
		
		Thus, if $U$ is not finitely generated, then infinitely many different ideals must appear as $I_{\By}$.  By standard methods, we can choose an infinite  sequence $\By^{(1)},\By^{(2)}, \dots$ such that the ideals $I_{\By^{(i)}}$ are all different, and for each $i$, the difference of coordinates $y_i^{(k)}-x_i$ are either all positive and weakly increasing with respect to $k$, or negative and weakly decreasing. In either case, we have $I_{\By^{(1)}}\subset I_{\By^{(2)}}\subset \cdots $.  Since $S_{\K}$ is Noetherian, the existence of such a chain of ideals which are all distinct contradicts the ascending chain condition, proving that $U$ is finitely generated.  
	\end{proof}
	If an algebra $A$ is locally left Noetherian,  then its category of finitely generated, locally unital modules is an abelian category $A\operatorname{-lu-mod}$.  
	The objects $A1_{\alpha_i}$ form a {\bf nearly resolving set of projectives} in the sense of Freyd \cite[\S 1]{Freyd}, i.e. every object in this category is a quotient of a finite sum of these projectives.

	\begin{lemma}\label{lem:Freyd-morita}
		Assume $A$ is locally Noetherian.  If $\mathcal{C}$ is an abelian category and $\alpha\mapsto P_{\alpha}\colon \mathscr{P} \to \mathcal{C}$ is a fully faithful functor such that the set of projectives $\{P_{\alpha}\}_{\alpha\in \aleph}$ is nearly resolving, then the functor \[\mathsf{M}\colon \mathcal{C}\to A\operatorname{-lu-mod}\qquad \mathsf{M}(M)=\oplus_{\alpha\in \aleph}\Hom(P_{\alpha},M) \] is an equivalence.   
	\end{lemma}
	\begin{proof} 
		In the terms of \cite{Freyd}, this functor $\mathsf{M}$ sends an object in $\mathcal{C}$ to the corresponding representation of the category $\mathscr{P}^{\operatorname{op}}$.  By a small modification of \cite[Thm. 1.2]{Freyd} (stated above \cite[Thm. 1.3]{Freyd} with the proof left to the reader), this functor is an equivalence to the subcategory of representations which are the cokernel of a map of the form $\mathsf{M}(\oplus_{i=1}^r  P_{\alpha_i})\to\mathsf{M}(\oplus_{j=1}^s  P_{\beta_j})$, that is of the form $\oplus_{i=1}^r  A1_{\alpha_i}\to \oplus_{j=1}^s  A1_{\beta_j}$. Since $A$ is locally Noetherian,  the modules of this form are exactly the finitely generated, locally unital modules.
	\end{proof}

	Let 
	$\gradedHalg^{\la}_{\K}\mmod_o$ denote the category of finite dimensional representations of $\gradedHalg^{\la}_{\K}$, on which each $\ps_i$ acts nilpotently.  As discussed above, such a module is necessarily locally unital.  
	\begin{theorem} \label{thmprojectiveendomorphisms}
		The functor \[\bigoplus_{\Bx\in
			\gradedchambers(\la)}\Hom\Big( P_\Bx,-\Big) \colon A^\la_{\K}\mmod_o^D\to
		\gradedHalg^{\la}_{\K}\mmod_o\]  
		defines an equivalence of categories between $A^\la_{\K}\mmod_o^D$ and the category of finite dimensional representations of $\gradedHalg^{\la}_{\K}$, on which each $\ps_i$ acts nilpotently.
	\end{theorem}
	\begin{proof}
		First, consider the category $\widehat{A^\la_{\K}\mmod}^D$.  Since any module $M$ in this category is topologically finitely generated over $A^\la_{\K}$, we can assume that the generators are generalized weight vectors for $D$. These weight vectors induce a surjection $\oplus_{i=1}^k Q_{\Ba_i}\to M$.  This shows that the $\{P_{\Bx}\}$ are a nearly resolving set of projectives in this category, and we have an equivalence of the category $\widehat{A^\la_{\K}\mmod}^D$ to the category of modules over the completion $\widehat{H}^{\la}_{\K}\cong \gradedHalg^{\la}_{\K}\otimes_{S}\widehat{S}$ by Lemma \ref{lem:Freyd-morita}; note that we have used the anti-automorphism of $ \gradedHalg^{\la}_{\K}$ to switch between left and right modules. 
		
		Now, we will show that restricting this functor gives the desired equivalence.  A finite dimensional representation of $\gradedHalg^{\la}_{\K}$ on which each $\ps_i$ acts nilpotently can be inflated to a $\widehat{H}^{\la}_{\K}$-module, and thus sent to a $A^\la_{\K}$-module by this equivalence.  The nilpotent condition and finite dimensionality imply that this module is a sum of finitely many generalized weight spaces, so it is supported on a finite union of points in $\fN_{\K}^{(1)}$.  The functions $z_i$ and $w_i$ must act nilpotently for weight reasons, so the only point in the support must be $o$.  
		On the other hand, if the corresponding $A^\la_{\K}$-module is supported on $o$, then by coherence, it must be finite-dimensional, and thus give a finite dimensional  $\widehat{H}^{\la}_{\K}$-module, and $\ps_i$ acts nilpotently on any finite dimensional $\widehat{H}^{\la}_{\K}$-module.  
	\end{proof}
	
	In fact, we will see that when $\la$ is smooth, $\gradedHalg^{\la}_{\K}$ admits a presentation as a quadratic algebra. We begin by producing some generators. 
	
	Let $\epsilon_i=(0,\dots, 0,1,0,\dots,0)$ by the $i$th unit vector.
	Let $c_{\Bx}^{\pm i}=c_{\Bx\pm\epsilon_i,\Bx}$; note that $\deg c_{\Bx}^{\pm i}=1$.  These elements correspond to the adjacencies in the graph structure of $\gradedchambers(\la)$.  Thus, we have a homomorphism from the path algebra of $\gradedchambers(\la)$ sending each length 0 path to the corresponding $1_{\Bx}$ and each edge to the corresponding $c_{\Bx}^{\pm i}$.  
	
	We'll be interested in the particular cases of \eqref{general-relation} which relate these length 1 paths.  
	\newseq
	If $\Bx,\Bx+\epsilon_i\in \gradedchambers(\la)$, then \[\subeqn\label{wall-cross}
	c_{\Bx+\epsilon_i}^{-i}c_{\Bx}^{+i}=\ps_i 1_{\Bx+\epsilon_i}\qquad c_{\Bx}^{+i}c_{\Bx+\epsilon_i}^{-i}=\ps_i 1_{\Bx+\epsilon_i}.\]  
	Note, we can view this as saying that the length 2 paths that cross a hyperplane and return satisfy the same linear relations as the normal vectors to the corresponding hyperplanes.  
	
	If
	$\Bx,\Bx+\epsilon_i,\Bx+\epsilon_j,\Bx+\epsilon_i+\epsilon_j\in
	\gradedchambers(\la)$, then the corresponding chambers fit together as in the picture below:
	
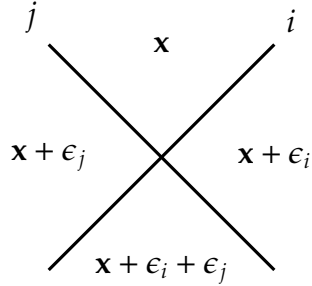
\[\tikz[very thick,scale=1.5]{ 
		\draw (1,1) --  node [at start, above right]{$i$} (-1,-1);
		\draw (-1,1) -- node [at start, above left]{$j$}(1,-1);
		\node at (0,1) {$\Bx$};
		\node at (0,-1) {$\Bx+\epsilon_i+\epsilon_j$};
		\node at (1,0) {$\Bx+\epsilon_i$};
		\node at (-1,0) {$\Bx+\epsilon_j$};}\]
	In this situation, we find that either way of going around the codimension 2 subspace gives the same result, and that more generally any two paths between chambers that never cross the same hyperplane twice give equal elements of the algebra.
	
	\[\subeqn \label{codim1} c_{\Bx+\epsilon_i}^{+ j}c_{\Bx}^{+
		i}=c_{\Bx+\epsilon_j}^{+ i}c_{\Bx}^{+ j}\qquad c_{\Bx+\epsilon_j}^{- j}c_{\Bx+\epsilon_i+\epsilon_j}^{-
		i}=c_{\Bx+\epsilon_i}^{- i}c_{\Bx+\epsilon_i+\epsilon_j}^{- j}.\]
	\[\subeqn \label{codim2} c_{\Bx+\epsilon_i+\epsilon_j}^{- j}c_{\Bx+\epsilon_j}^{+
		i}=c_{\Bx}^{+ i}c_{\Bx+\epsilon_j}^{- j}\qquad
	c_{\Bx+\epsilon_i+\epsilon_j}^{-
		i}c_{\Bx+\epsilon_i}^{+j}=c_{\Bx}^{+j}c_{\Bx+\epsilon_i}^{- i}.\]
	If $\la$ is a smooth parameter, then as the following theorem shows, these are the only relations
	needed. 
	
	\begin{theorem}\label{presentation2}
		If $\la$ is a smooth parameter, then  the algebra $\bigoplus_{\Bx, \By}\Hom(P_{\Bx}, P_{\By})$ is generated by the
		idempotents $1_\Bx$ and the elements
		$c_{\Bx}^{\pm i}$ for all $\Bx\in \gradedchambers(\la)$ over $\widehat{\pS}$ modulo
		the relations (\ref{wall-cross}--\ref{codim2}).
	\end{theorem}
	\begin{proof}
		Since these relations are a consequence of Theorem
		\ref{presentation1}, it suffices to show that the elements
		$c_{\Bx}^{\pm i}$ generate, and that the relations
		\eqref{general-relation} are a consequence of
		(\ref{wall-cross}--\ref{codim2}).
		
		We show that $c_{\Bx}^{\pm i}$ generate $c_{\Bx,\By}$ by induction
		on the $L_1$-norm $|\Bx-\By|$.  If $|\Bx-\By|_1=1$, then
		$c_{\Bx,\By}=c_{\By}^{\pm i}$.  On the other hand, if $|\Bx-\By|_1>1$,
		then there is some $\Bx'\neq \Bx,\By$ such that
		$|\Bx-\Bx'|_1+|\Bx'-\By|_1=|\Bx-\By|_1$. Choosing a generic parameter $\la'$ such that $\gradedchambers^\R(\la')=\gradedchambers(\la)$, we can consider the line segment joining generic points in $\bar{\chamber}_{\Bx}$ and $\bar{\chamber}_{\By}$, and let $\Bx'$ be any chamber this line segment passes through.  The smoothness hypothesis is needed to conclude that there is such a chamber that lies in $\gradedchambers(\la)$.
		Since $c_{\Bx,\By}=c_{\Bx,\Bx'}c_{\Bx',\By}$, this proves generation by
		induction.
		
		We must now check that the relations \eqref{general-relation} are satisfied. First, consider the situation where  $\Bx^{(0)}=\Bx,\dots,\Bx^{(m)}=\By$ is a path with
		$|\Bx^{(i)}-\Bx^{(i+1)}|_1=1$ with $\Bx^{(i)}\in \gradedchambers(\la)$, and
		$\By^{(0)}=\Bx,\dots,\By^{(m)}=\By$ a path with the same
		conditions. These two paths differ by a finite number of applications of
		the relations (\ref{codim1}--\ref{codim2}).  
		
		What it remains to show is that if
		$\Bx^{(0)}=\Bx,\dots,\Bx^{(m)}=\By$ is a path of minimal length
		between these points with
		$|\Bx^{(i)}-\Bx^{(i+1)}|_1=1$, and we have similar paths
		$\By^{(0)}=\By,\dots,\By^{(n)}=\Bu$ and
		$\Bu^{(0)}=\Bx,\dots,\Bu^{(p)}=\Bu$, then 
		\begin{equation}
			c_{\Bx^{(0)},\Bx^{(1)}}\cdots
			c_{\Bx^{(m-1)},\Bx^{(m)}}c_{\By^{(0)},\By^{(1)}}\cdots
			c_{\By^{(n-1)},\By^{(n)}}=c_{\Bu^{(0)},\Bu^{(1)}}\cdots
			c_{\Bu^{(p-1)},\Bu^{(p)}}\prod_{i=1}^n\ps_i^{\eta_i(\Bx,\By,\Bu)}.\label{product}
		\end{equation}

		We'll prove this by induction on $\min(m,n)$.  If $m=0$ or $n=0$, then
		this is tautological.  Assume $m=1$, and $\Bx=\By+\sigma \ep_j$ for
		$\sigma\in \{1,-1\}$.  If $\sigma(y_j-u_j)\geq 0$, then
		$\eta_j(\Bx,\By,\Bu)=0$, so this follows from the statement about
		minimal length paths.  If $\sigma(y_j-u_j)< 0$, then
		$\eta_j(\Bx,\By,\Bu)=1$, and we can assume that
		$\By^{(1)}=\Bx,\dots,\By^{(n)}$ is a minimal length path from $\Bx$ to
		$\Bu$.  Thus \[c_{\Bx,\By}c_{\By,\Bx}\cdots
		c_{\By^{(n-1)},\By^{(n)}}= c_{\Bx,\By^{(2)}}\cdots
		c_{\By^{(n-1)},\By^{(n)}}\ps_j\] as desired.  The argument if $n=1$ is
		analogous.  
		
		Now consider the general case.  Assume for simplicity that $n\geq m$.
		Consider the path $\Bx^{(m-1)},\By^{(0)},\dots,\By^{(n)}=\Bu$.  Either
		this is a minimal path, or by induction, we have that
		\[c_{\Bx^{(m-1)},\By}c_{\By,\By^{(1)}}\cdots
		c_{\By^{(n-1)},\By^{(n)}}=c_{\Bw^{(1)},\Bw^{(2)}}\cdots
		c_{\Bw^{(n-2)},\Bw^{(n-1)}}\ps_j \] for a minimal path
		$\Bw^{(0)}=\Bx^{(m-1)},\Bw^{(2)},\dots, \Bw^{(n-1)}=\By^{(n)}$ with $j$ being the index that changes from $\Bx^{(m-1)}$ to $\By$.  
		
		In the former case, by induction, the relation \eqref{product} for the paths
		$\Bx^{(0)}=\Bx,\dots,\Bx^{(m-1)}$ and
		$\Bx^{(m-1)},\By^{(0)},\dots,\By^{(n)}=\Bu$ holds.  This is just a
		rebracketing of the desired case of \eqref{product}.  In the latter,
		after rebracketing, we have 
		\begin{multline*}
			(c_{\Bx^{(0)},\Bx^{(1)}}\cdots c_{\Bx^{(m-2)},\Bx^{(m-1)}})(
			c_{\Bx^{(m-1)},\Bx^{(m)}}\cdots
			c_{\By^{(n-1)},\By^{(n)}})\\=(c_{\Bx^{(0)},\Bx^{(1)}}\cdots c_{\Bx^{(m-2)},\Bx^{(m-1)}})(
			c_{\Bw^{(0)},\Bw^{(1)}}\cdots
			c_{\Bw^{(n-2)},\Bw^{(n-1)}})\ps_j=c_{\Bu^{(0)},\Bu^{(1)}}\cdots
			c_{\Bu^{(p-1)},\Bu^{(p)}}\prod_{i=1}^n\ps_i^{\eta_i(\Bx,\By,\Bu)}.
		\end{multline*}
		applying \eqref{product} to the shorter paths.
	\end{proof}

	\subsection{Quadratic duality and the Ext-algebra of the sum of all simple modules.}
	\label{sec:quadratic-duality}
	
	The algebra $\gradedHalg^{\lambda}_\Z$ for smooth parameters has already appeared in the literature in \cite{GDKD}; it
	is the ``A-algebra'' of the hyperplane arrangement defined by $\Dc_i=pk-1/2$
	for all $k\in \Z$.  This is slightly outside the scope of that paper,
	since only finite hyperplane arrangements were considered there, but
	the results of that paper are easily extended to the locally finite
	case.  In particular, we have that the algebra $\gradedHalg^{\lambda}_\Z$ is quadratic, and
	its quadratic dual also has a geometric description, given by the
	``B-algebra.'' We will use this to produce a description of the $\Ext$-algebra of the sum of all simple representations of $\gradedHalg^{\lambda}_\Z$.

	If we fix an integer $m$, we may consider the hyperplane arrangement
	given by $\Dc_i=pk-1/2$ for $k\in [-m,m]$.  Let $H^{[m]}$ be the A-algebra associated to this arrangement as in \cite[\S
	8.3]{BLPWtorico} (in that paper, it is denoted by $A(\eta,-)$). We leave the dependence on $\la$ and the ground ring implicit. 
	
	By definition, $H^{[m]}$ is obtained by considering the chambers of the arrangement we have fixed above, putting a quiver structure on this set by connecting chambers adjacent across a hyperplane, and then imposing the same local relations (\ref{wall-cross}--\ref{codim2}). One result which will be extremely important for us
	is:
	\begin{theorem}[\mbox{\cite[8.25]{BLPWtorico}}]
		The algebra $H^{[m]}$ is finite dimensional in each graded degree, with finite global
		dimension $\leq 2n$.  
	\end{theorem}

	There's a natural map of $H^{[m]}$ to $\gradedHalg^{\la}$, sending the
	idempotents for chambers to $1_{\Bx}$ for $x_i\in [-m-1,m]$.
	
	\begin{proposition}  Fix $\Bx,\By$ and an integer $q$.  
		For $m$ sufficiently large, the map $H^{[m]}\to \gradedHalg^{\la}$ induces an
		isomorphism $(1_{\Bx}H^{[m]}1_{\By})_q\cong (1_{\Bx}\gradedHalg^{\lambda}1_{\By})_q$ between homogeneous
		elements of degree $q$ and an isomorphism $\Ext_{H^{[m]}}(L_{\Bx},L_{\By})\cong
		\Ext_{\gradedHalg^{\la}}(L_{\Bx},L_{\By})$.  
	\end{proposition}
	\begin{proof}
		An element of $(1_{\Bx}H^{[m]}1_{\By})_q$ can be written as a sum of length
		$n$ paths from $x$ to $y$.  Thus, it can only pass through $\Bu$ if
		$|\Bx-\Bu|+|\Bu-\By|\leq q$.  Thus, if $m>q+|\Bx|_1+|\By_1|$, then
		no hyperplane crossed by this path is excluded in $H^{[m]}$. The map
		$(1_{\Bx}H^{[m]}1_{\By})_q \to (1_{\Bx}\gradedHalg^\la 1_{\By})_q$ is clearly surjective in this
		case, and injective as well, since any relation used in $H$ is also
		a relation in $H^{[m]}$.  
		
		Thus, if we take a projective resolution of $L_{\Bx}$ over $H^{[m]}$
		and tensor it with $\gradedHalg^\la $, we can choose $m$
		sufficiently large that the result is still exact in degrees below
		$2q$.  Since $H^{[m]}$
		is Koszul, with global dimension $\leq 2n$, every simple over $H^{[m]}$ has a linear resolution of length less than $\leq 2n$.  This establishes that
		the tensor product complex is a projective resolution for $m\gg 0$.
		
		This establishes that we have an isomorphism $\Ext_{H^{[m]}}(L_{\Bx},L_{\By})\to
		\Ext_{\gradedHalg^{\la}}(L_{\Bx},L_{\By})$ for $m\gg 0$.  
	\end{proof}
	\begin{corollary} \label{cor:KoszulityofH}
		The algebra $\gradedHalg^{\la}$ is Koszul with global dimension $\leq 2n$.
	\end{corollary}
	Note that in the language of \cite[\S 5.4]{MOS}, we should say that the category $\mathscr{P}(\gradedHalg^{\la})$ is Koszul.  By \cite[Th. 30]{MOS}, 
	the Koszul dual of $\gradedHalg^{\la}$, is its quadratic
	dual.  Thus, lety us calculate is quadratic dual.
	
	Continue to assume that $\gradedchambers(\la)$ is smooth.  If we dualize the short
	exact sequence \[0\longrightarrow\ft_\Z \longrightarrow\fd_\Z
	\longrightarrow\mathfrak{g}_\Z\longrightarrow 0\] we obtain a dual
	sequence
	\[0\longleftarrow\ft_\Z^* \longleftarrow\fd^*_\Z
	\longleftarrow\mathfrak{g}_\Z^*\longleftarrow 0\]
	Let $\pt_i$ be the image in $\ft_\Z^*$ of the $i$th coordinate weight of
	$\ft_\Z^*$.

	\begin{definition} \label{defHdual}
		\newseq
		Let  $\gradedHalg_{\la,\Z}^!$ (resp $\gradedHalg_{\la,\K}^!$) be the dg-algebra generated over $U_\Z (\ft^*)$ (resp $U_\K (\ft^*)$) by elements $e_{\Bx}$ for $\Bx\in
		\gradedchambers(\la)$, $d_{\Bx}^{\pm i}$ for $\Bx,\Bx\pm \epsilon_i\in \gradedchambers(\la)$ with trivial differential and 
		subject to the quadratic relations:
		\begin{enumerate}
			\item Write $d_{\Bx, \Bu} := d_\Bx^{\pm i}$ where $\Bu = \Bx \pm \epsilon_i$. For each $\Bx$ and each $i$, we have:
			\[\subeqn\label{wall-crossbang}
			\sum_{\Bu\in \al_i(\Bx)}d_{\Bx,\Bu}d_{\Bu,\Bx}=\pt_i e_{\Bx}.\]    
			Note that this implies that if $\al_i(\Bx)=\emptyset$, then $\pt_i e_{\Bx} =0$.  
			\item If
			$\Bx,\Bx+\epsilon_i,\Bx+\epsilon_j,\Bx+\epsilon_i+\epsilon_j\in
			\gradedchambers(\la)$, then 
			\[\subeqn \label{codim1bang} d_{\Bx+\epsilon_i}^{+ j}d_{\Bx}^{+
				i}=-d_{\Bx+\epsilon_j}^{+ i}d_{\Bx}^{+ j}\qquad d_{\Bx+\epsilon_j}^{- j}d_{\Bx+\epsilon_i+\epsilon_j}^{-
				i}=-d_{\Bx+\epsilon_i}^{- i}d_{\Bx+\epsilon_i+\epsilon_j}^{- j}.\]
			\[\subeqn \label{codim2bang} d_{\Bx+\epsilon_i+\epsilon_j}^{- j}d_{\Bx+\epsilon_j}^{+
				i}=-d_{\Bx}^{+ i}d_{\Bx+\epsilon_j}^{- j}\qquad
			d_{\Bx+\epsilon_i+\epsilon_j}^{-
				i}d_{\Bx+\epsilon_i}^{+j}=-d_{\Bx}^{+j}d_{\Bx+\epsilon_i}^{- i}.\]
			\item If $\Bx$ and $\Bu$ are chambers such that $|\Bx-\Bu|=2$
			and there is only one length $2$ path $(\Bx,\By,\Bu)$ in $\gradedchambers(\la)$ from $\Bx$ to
			$\Bu$, then  \[\subeqn \label{doublestep} d_{\Bx,\By}d_{\By,\Bu}=0.\] 
		\end{enumerate}
	\end{definition}
	For example, if $\Bx\notin\gradedchambers$ but $\Bx+\epsilon_i,\Bx+\epsilon_j,\Bx+\epsilon_i+\epsilon_j\in
	\gradedchambers(\la)$, then
	$ d_{\Bx+\epsilon_i+\epsilon_j}^{- j}d_{\Bx+\epsilon_j}^{+
		i}=0.$
	We suppress the dependence of $\gradedHalg$ and $\gradedHalg^!$ on $\la$ and the ground ring, to avoid clutter. The following holds over both $\K$ and $\Z$:
	\begin{theorem} \label{thm:HandHdualaredual}
		The algebras $\gradedHalg$ and $\gradedHalg^!$ are quadratically dual, with the pairing
		$\gradedHalg_1\times \gradedHalg_1^!$ given by \[\langle c_{\Bx}^{\sigma i},
		d_{\By}^{\sigma'
			j}\rangle=\delta_{\Bx,\By}\delta_{i,j}\delta_{\sigma,\sigma'}.\]  
	\end{theorem}
	Again, in the notation of \cite{MOS}, we would say that the categories $\mathscr{P}(\gradedHalg)$ and $\mathscr{P}(\gradedHalg^!)$ are quadratically dual. 
	\begin{proof}
		What we must show is that the quadratic relations of $\gradedHalg$ in
		$\gradedHalg_1\otimes_{\gradedHalg_0} \gradedHalg_1$ are the annihilator of the relations of $\gradedHalg^!$
		in $\gradedHalg_1^!\otimes_{\gradedHalg_0} \gradedHalg_1^!$. It is enough to consider $e_\Bx \gradedHalg_1\otimes_{\gradedHalg_0} \gradedHalg_1 e_{\By}$ for any pair of idempotents $e_\Bx,e_{\By}$.
		This space can only be non-zero if $|\Bx-\By|=2$ or $0$.  Let us first
		assume that $|\Bx-\By|=2$.  If there is one path between $\Bx$ and
		$\By$ in $\gradedchambers$, then $e_\Bx \gradedHalg_1\otimes_{\gradedHalg_0} \gradedHalg_1 e_{\By}\cong
		e_\Bx \gradedHalg_2 e_{\By}$ and there are no relations.  On the other hand,
		in $\gradedHalg^!$, we have that all elements of $e_\Bx \gradedHalg_1^!\otimes_{\gradedHalg_0^!}
		\gradedHalg_1^! e_{\By}$ are relations by (\ref{doublestep}).
		
		If there are two paths, through $\Bu$ and $\Bu'$, then the element
		$c_{\Bx,\Bu}\otimes c_{\Bu,\By}-c_{\Bx,\Bu'}\otimes c_{\Bu',\By}$
		spans the set of relations.  Its annihilator is $d_{\Bx,\Bu}\otimes
		d_{\Bu,\By}+d_{\Bx,\Bu'}\otimes d_{\Bu',\By}$, which spans the relations in
		$\gradedHalg^!$ by (\ref{codim1bang}--\ref{codim2bang}).  This deals with the case where $|\Bx-\By|=2$.  
		
		Now, assume that $\Bx=\By$.  The space $e_\Bx \gradedHalg_1\otimes_{\gradedHalg_0} \gradedHalg_1
		e_{\Bx}$ is spanned by $c_{\Bx,\Bu}c_{\Bu,\Bx}$ for $\Bu\in
		\alpha(\Bx)$.  Thus, $e_\Bx \gradedHalg_1\otimes_{\gradedHalg_0} \gradedHalg_1
		e_{\Bx}\cong \K^{\alpha(\Bx)}$. We can map this to $\fd_{\K}$ by sending the unit vector corresponding to $\Bu$ to $\ps_i$ where
		$\Bx=\Bu\pm \epsilon_i$.  The relations are the preimage of 
		$\ft_{\K}$.  
		
		By standard linear algebra, the annihilator of a preimage is the image
		of the annihilator under the dual map.  Thus, we must consider the
		dual map $\ft_{\K}^\perp\subset \fd_{\K}^*\to \K^{\alpha(\Bx)}$, and identify its image
		with the relations in $\gradedHalg^!$.  These are exactly the relations imposed by taking linear combinations of the relations in (\ref{wall-crossbang}) such that the RHS is 0.
	\end{proof}

		\begin{corollary} \label{corSpellingoutKoszul}
			We have a quasi-isomorphism of dg-algebras $\oplus_{\Bx, \By} \Ext( L_\Bx, L_\By) \cong \gradedHalg^!_{\la, \K}$, with
			individual summands given by $\Ext(L_\By,L_{\Bx})\cong e_\Bx \gradedHalg_\la^! e_\By$.
		\end{corollary}
		\begin{proof}
			Here, we apply Theorem \ref{thmprojectiveendomorphisms}; this equivalence of abelian categories implies that we can replace the computation of $\Ext_{A^\la}(L_{\Bx},L_{\By})$ with that of the corresponding 1-dimensional simple modules over $\gradedHalg^\la$ in the subcategory of modules on which $\ps_i$ acts nilpotently.  
			
			If we instead did the same computation in the bounded derived category of all finitely generated modules, then we would know the result is $e_\Bx \gradedHalg_\la^! e_\By$ by Koszul duality.  The formality of the $\Ext$-algebra follows from the consistency of $A_\infty$-operations with the internal grading, so this is a quasi-isomorphism of dg-algebras.  Thus, we need to know that the inclusion of the category on which $\ps_i$ acts nilpotently induces a fully-faithful functor on derived categories.  
			
			For this, it's enough to show that every pair of objects $A,B$ has an object $C$ (all in the subcategory) and a surjective morphism $\psi\colon C\to A$ such that the induced map $\Ext^n(A,B)\to \Ext^n(C,B)$ is trivial for all $n$.  We can accomplish this with $C$ a sum of quotients of $\gradedHalg^\la 1_\Bz$'s by the ideal generated by $\ps_i^N$ for $N\gg 0$; this is clear for degree reasons if $A$ and $B$ are gradeable, and since gradeable objects dg-generate, this is enough.  
		\end{proof}
		This gives us a combinatorial realization of the $\Ext$-algebra of the
		simple modules in this category. We can restate it in terms of Stanley-Reisner rings as follows.
		
		For every $\Bx,\By$, we have a polytope
		$\bar{\chamber}^{\R}_{\Bx}\cap \bar{\chamber}^{\R}_{\By}$, which has an associated
		Stanley-Reisner ring $SR(\Bx,\By)_\K$.  The latter is the quotient of
		$\K[\pt_1,\dots, \pt_n]$ by the relation that $\pt_{i_1}\cdots
		\pt_{i_k}=0$ if the intersection of $\bar{\chamber}^{\R}_{\Bx}\cap
		\bar{\chamber}^{\R}_{\By}$ with the hyperplanes defined by $a_{i_j}=pn$ for $n\in \Z$ is empty. Let $\overline{SR}(\Bx,\By)_\K$ be
		its quotient modulo the system of parameters defined by the
		image of $\ft_{\K}^\perp$. 
		
		We can define $SR(\Bx,\By)_\Z$ and $\overline{SR}(\Bx,\By)_\Z$ by the same prescription, replacing $\K$ by $\Z$ everywhere. 
		In \cite[4.1]{GDKD}, the authors define a product on the sum
		$\overline{SR}_\Z \cong \oplus_{\Bx,\By\in\gradedchambers}\overline{SR}(\Bx,\By)_\Z$, which they call the
		``B-algebra.'' The same definition works over $\K$.

		The result \cite[4.14]{GDKD} shows that this algebra
		is isomorphic to the ``A-algebra'' (that defined by the relations
		(\ref{wall-cross}--\ref{codim2})) for a Gale dual hyperplane
		arrangement.  Unfortunately, for a periodic arrangement, the Gale dual
		is an arrangement on an infinite dimensional space, which we will not
		consider.  We can easily restate this theorem in a way which will
		generalize for us. 
		Assume that $\la$ is a
		smooth parameter. 
		\begin{proposition}[\mbox{\cite[Th. A]{GDKD}}] \label{extalgebraisSR}
			The algebra $\overline{SR}_\K$ is quadratic dual to $\gradedHalg_\K^\la$.  That is, it
			is isomorphic to $\gradedHalg_{\la, \K}^!$.  In particular, 
			we have the canonical isomorphisms \[\Ext(L_\Bx,L_{\By}) \cong e_\Bx \gradedHalg^!_{\la, \K} e_\By \cong
			\overline{SR}(\Bx,\By)_\K[-|\Bx-\By|_1].\]
		\end{proposition}
		
		\subsection{Interpretation as the cohomology of a toric variety}
		\label{interpretation}
		For our purposes, the key feature of the quadratic dual of $\gradedHalg^{\la}_\Z$ is its topological
		interpretation, which is exactly as in \cite[\S 4.3]{GDKD}. This interpretation will allow us to match the $\Ext$-algebras which appear on the mirror side, in the second half of this paper.
		
		Indeed, the periodic hyperplane arrangement $\textgoth{A}^{\operatorname{per}}_\la$ defines a tiling of $\mathfrak{g}^{*,\la}_\R$ by the polytopes $\bar{\chamber}^{\R}_{\Bx}$. 
		
		To each such polytope we can associate a $G$-toric variety $\mathfrak{X}_{\Bx}$ \cite[Chapter XI]{CdS}. Each facet of the polytope defines a toric subvariety of $\mathfrak{X}_{\Bx}$. In particular, the facet $\chamber^{\R}_\Bx \cap \chamber^{\R}_\By$ defines a toric subvariety $\mathfrak{X}_{\Bx, \By}$ of both $\mathfrak{X}_{\Bx}$ and $\mathfrak{X}_{\By}$.

		Moreover, the Stanley-Reisner ring $SR(\Bx,\By)_\K$ is identified with $\HH^*_G(\mathfrak{X}_{\Bx, \By} ; \K)$, and the quotient $\overline{SR}(\Bx,\By)_\K$ is identified with $\HH^*(\mathfrak{X}_{\Bx, \By} ; \K)$. Composing this identification with Proposition \ref{extalgebraisSR}, we have an identification $$e_\Bx \gradedHalg^!_{\la, \K} e_\By \cong \HH^*(\mathfrak{X}_{\Bx, \By} ; \K) [-|\Bx-\By|_1] .$$ 
		
		In this presentation, multiplication in the $\Ext$-algebra is given by a natural convolution on cohomology groups \cite[\S 4.3]{GDKD}.

		\subsection{Degrading} 
		\label{sec:degrading}
		
		So far, we have only considered $A_\K^\la$-modules which are endowed with a
		$D$-action. Now, we use the results of the preceding sections to describe the category $A^\la_{\K}\mmod_o$ of modules without this extra structure.
		\begin{proposition}
			Assume that $L$ is a simple module in the category $A^\la_{\K}\mmod_o$.  Then we have an isomorphism
			of $ A^\la_{\K}$-modules $L\cong L_{\Bx}$ for some $\Bx$.
		\end{proposition}
		\begin{proof}  
			On the subcategory $ A^\la_{\K}\mmod_o$, the central element 
			\[z_i^p\partial_i^p=\ch_i^+(\ch_i^+-1)(\ch_i^+-2)
			\cdots (\ch_i^+-p+1)\] acts nilpotently, so $\ch_i^+$ has
			spectrum in $\mathbb{F}_p$.  
			In $L$, there thus must exist a simultaneous
			eigenvector $v$ for all $\ch_i^+$'s, and $\Ba$ such
			that $\ch_i^+v=a_iv$.  Thus, $W_\Ba^1(L)\neq 0$, which shows
			that there is a non-zero map $Q_\Ba\cong P_\Bx\to L$, so we must have
			$L\cong L_\Bx$.
		\end{proof}
		This shows that $L_\Bx$ gives a complete list of simples. The module
		$P_\Bx$ represents the $\Ba$ generalized eigenspace of
		$\ch_i^+$, and thus still projective.  In fact, there are
		redundancies in this list, but they are easy to understand.  
		\begin{definition}
			Let $\degradedchambers(\lambda)$ be the quotient of $\gradedchambers(\lambda)$ by the equivalence relation
			that $\Bx\sim \By$ if and only if $\Bx|_{\ft_\Z}=\By|_{\ft_\Z}$. Equivalently, $\Bx\sim \By$ if $\By = \Bx + \gamma$ where $\gamma$ lies in $\ft^{\perp}_\Z = \mathfrak{g}^*_\Z$. 
		\end{definition}
		We write $\bar{\Bx}$ for the image of $\Bx$ in $\degradedchambers(\lambda)$. Recall that $\mathfrak{g}^{*,\la}_\Z$ is a torsor for the lattice $\mathfrak{g}^*_\Z$. The action of the sublattice $p \cdot \mathfrak{g}^*_\Z$ preserves the periodic arrangement $\textgoth{A}^{\operatorname{per}}_\la$. The quotient $\textgoth{A}^{\operatorname{tor}}_\la = \textgoth{A}^{\operatorname{per}}_\la / p \cdot \mathfrak{g}^*_\Z$ is an arrangement on the quotient $\mathfrak{g}^{*,\la}_\Z / p \cdot \mathfrak{g}^*_\Z$, and $\degradedchambers(\la)$ is the set of chambers of $\textgoth{A}^{\operatorname{tor}}_\la$. 
		\begin{example} 
			In the setting of Example \ref{classicexample}, $\textgoth{A}^{\operatorname{tor}}_\la$ has three chambers. A set of representatives is given by those chambers of the periodic arrangement lying within the pictured square.
		\end{example}

		\begin{theorem}
			As $ A^\la_{\K}$-modules $L_\Bx\cong L_\By$ if and only if $\Bx\sim
			\By$. That is, the simple modules in $ A^\la_{\K}\mmod_o$ are in
			bijection with $\degradedchambers(\lambda)$.
		\end{theorem}
		\begin{proof}
		  
			If $\Bx \sim \By$, then $P_\Bx$ and $P_{\By}$ are canonically isomorphic as $A^\la_{\K}$-modules, since \eqref{defQa} is only sensitive to the coset of $\Ba$ under the action of $p \cdot \ft^{\perp}_\Z$. It follows that $L_\Bx \cong L_\By$ as $A^\la_{\K}$-modules.
			On the other hand, if $L_\Bx\cong L_\By$ as $ A^\la_{\K}$-modules,
			their weights modulo $p$ must agree.  This is only possible if $\Bx|_{\ft_\Z}=\By|_{\ft_\Z}$.
		\end{proof}

		When convenient, we will write $L_{\bar{\Bx}}$ for the simple attached to $\bar{\Bx} \in \degradedchambers(\la)$. We can understand the $\Ext$-algebra of simples using the degrading functor 
		$\mathbb{D}\colon A^\la_{\K}\mmod^D_o\to
		A^\la_{\K}\mmod_o$ which forgets the action of $D$.  
		\begin{theorem}
			We have a canonical isomorphism of algebras 
			\[\Ext_{A^\la_{\K}\mmod_o}(L_\Bx,L_{\By})\cong
			\oplus_{\Bx|_{\ft_\Z}=\Bu|_{\ft_\Z}}\Ext_{A^\la_{\K}\mmod_o^D}(L_\Bu,L_{\By})\cong 
			\oplus_{\By|_{\ft_\Z}=\Bu|_{\ft_\Z}}\Ext_{A^\la_{\K}\mmod_o^D}(L_\Bx,L_{\Bu}).\]
		\end{theorem}
		\begin{proof}
			This is immediate from the fact that $P_\Bx$ remains projective in
			$A^\la_{\K}\mmod_o$, so the degrading of a projective resolution of
			$L_\Bx$ remains projective.
		\end{proof}
		One can easily see that this implies that, just like $A^\la_{\K}\mmod^D_o$, the category
		$A^\la_{\K}\mmod_o$  has a Koszul graded lift, since the coincidence of the homological and internal
		gradings is unchanged.

		We can deduce a presentation of
		\[\degradedHalg_{\la, \K}^!=\bigoplus_{\Bx,\By \in \degradedchambers(\la)}\Ext_{A^\la_{\K}\mmod_o}\left( L_{\Bx},  L_{\By} \right)\]
		Indeed, we think of $\gradedHalg_{\la, \K}^!$ as the path algebra of the
		quiver $\gradedchambers(\lambda)$ (over the base ring $U_\K(\ft^*)$) satisfying the relations in definition \ref{defHdual}, and then apply the quotient map to $\degradedchambers(\lambda)$, keeping the arrows and relations in place.  This is well-defined since the relations (\ref{wall-crossbang}--\ref{codim2bang}) are unchanged by adding a character of $G$ to $\Bx$. 
		
		Likewise, we have the following description of the endomorphism algebra of the projectives. Let $\degradedHalg^\la_\K$ be the algebra generated by the idempotents $1_\Bx$ and the elements
			$c_{\Bx}^{\pm i}$ for all $\Bx\in \degradedchambers(\lambda)$ over $\pS$ modulo
			the relations (\ref{wall-cross}--\ref{codim2}). Let $\degradedHalg^\la_\Z$ be the natural lift to $\Z$.
		\begin{proposition} \label{prop:degradedAmodulesalgebra}
	\[  \degradedHalg_{\la, \K} =\bigoplus_{\Bx,\By \in \degradedchambers(\la)}\Hom_{A^\la_{\K}\mmod_o}\left( P_{\Bx},  P_{\By} \right). \]
		\end{proposition} 
		
		\begin{example}
			We continue Example \ref{classicexample}.  The set $\degradedchambers(\lambda)$ has 3 elements corresponding
			to the chambers $A$ where $x_1+x_2+x_3=\lfloor
			\frac{\la}{p}\rfloor$, $B$ where $x_1+x_2+x_3=\lfloor
			\frac{\la}{p}\rfloor-1$, and $C$ where $x_1+x_2+x_3=\lfloor
			\frac{\la}{p}\rfloor-2$.  We have adjacencies between $A$ and $B$
			across 3 hyperplanes, and between $B$ and $C$ across 3 hyperplanes,
			with none between $A$ and $C$.  
			
			Thus, our quiver is 
			\[\tikz[very thick,scale=3]{\node[outer sep=2pt](A) at (0,0) {$A$};\node[outer sep=2pt](B) at (1,0) {$B$};\node[outer sep=2pt](C) at
				(2,0) {$C$};  \draw (A) to[in=180,out=0] (B); \draw (A) to[in=160,out=20]
				(B); \draw (A) to[in=-160,out=-20] (B); \draw (B) to[in=180,out=0]
				(C);\draw (B) to[in=160,out=20] (C); \draw (B) to[in=-160,out=-20]
				(C);}\]
			
			We use $x_i$ to the path from $A$ to $B$ across the $\Dc_i$ hyperplane,
			and $y_i$ the path from $C$ to $B$ across the $\Dc_i$ hyperplane.  Our
			relations thus become:
			\[x_1^*x_1=x_2^*x_2=x_3^*x_3\qquad y_1^*y_1=y_2^*y_2=y_3^*y_3\]
			\[x_1x_1^*+y_1y_1^*=x_2x_2^*+y_2y_2^*=x_3x_3^*+y_3y_3^*\]
			\[x_i^*y_j=-x_j^*y_i \qquad y_i^*x_j=-y_j^*x_i \qquad
			x_ix_j^*=-y_jy_i^* \qquad (i\neq j)\]
			\[y_1^*x_1=y_2^*x_2=y_3^*x_3=x_1^*y_1=x_2^*y_2=x_3^*y_3=0\]
			\[x_i^*x_j=y_i^*y_j=0  \qquad (i\neq j)\]
		\end{example}
		
		Note that there are only finitely many elements of $\degradedchambers(\lambda)$.  In
		fact, the number of such elements has an explicit upper bound.
		A {\bf basis} of the inclusion $T\subset D$ is a set of coordinates
		such that the corresponding coweights form a basis of
		$\fd_{\mathbb{Q}}/\ft_{\mathbb{Q}}$. For generic parameters, taking the intersection of the corresponding coordinate subtori defines a bijection of the bases with the vertices of $\textgoth{A}^{\operatorname{tor}}_\la$. 
		
		\begin{lemma}
			The number of elements of $\degradedchambers^\R(\lambda)$ is less than or equal to the number
			of bases for the inclusion $T\subset D$.
		\end{lemma}
		\begin{proof}
			Choose a generic cocharacter $\xi\in \ft_{\mathbb{Q}}^\perp\subset
			\fd^*_{\mathbb{Q}}$.  
			Note that a real number $c$ satisfies the equations $x_ip\leq c
			<x_ip+p$ if and only if it satisfies $x_ip-\epsilon< c
			<x_ip+p-\epsilon$ for $\epsilon$ sufficiently small.  Thus, we will
			have no fewer nonempty regions if we consider the chambers   
			\[ E^\R_{\Bx}=\{\Ba\in \mathfrak{g}^{*,\la}_\Z\otimes \R \mid px_i-\epsilon_i <  a_i <
			px_i+p-\epsilon_i\}\] for some sufficiently small $\epsilon_i>0$
			chosen generically.  Note
			that $E^\R_{\Bx}$ is open.  
			For any $\bar E^{\R}_\Bx$, there is a maximal
			point for this cocharacter, that is, a point $\Ba$ such that for all
			$\Bb\neq \Ba\in \bar \chamber^{\R}_\Bx$, we have $\xi(\Bb-\Ba)<0$.  By
			standard convex geometry, this is only possible if there are
			hyperplanes in our arrangement passing through $\Ba$ defined by
			coordinates that are a basis.  In fact, by the genericity of the
			elements $\epsilon_i$, we can assume that the point $\Ba$ is hit by
			exactly a basis of hyperplanes.  This gives a map from $\degradedchambers(\lambda)^\R$ to the
			set of bases and this map is injective, since all but one of the
			chambers that contain $\Ba$ in its closure will contain points
			higher than $\Ba$.  
		\end{proof}
		
		Since the number of elements of $\degradedchambers^\R(\la)$ is lower-semicontinuous in $\la$, we see immediately that $\la$ is smooth if the size of $\degradedchambers(\lambda)$ is the number
		of bases.

		\subsection{Tilting generators for coherent sheaves}
		\label{sec:tilting-generators}

		We can also interpret these results in terms of coherent sheaves.   In
		particular, we can consider the coherent sheaf $\Qvectorbundle_\Ba=\varprojlim
		\A^\la_{\K}/\A^\la_{\K}(\ch_i^+-a_i)^N$
		on the formal completion of the
		fiber $\mu^{-1}(0)$. Here, as before, we assume that $a_i\in \mathbb{F}_p$, so $a_i^p-a_i=0$.  
		On this formal
		subscheme, this
		is an equivariant splitting bundle for the Azumaya algebra
		$\A^\la_{\K}$ by \cite[4.3.4]{Stadnik}.

		If we think of
		$\A^\la_{\K}|_{\mu^{-1}(\hat 0)}$ as a left module over
		itself, it decomposes according the eigenvalues of $\ch_i^+$
		acting on the right.  By construction, each generalized eigenspace defines a copy of $\Qvectorbundle_\Ba$ for some weight $\Ba$. If we let $\mathfrak{g}^{*,\la}_\Fp$ be the set of
		characters of $\fd_\Fp$ which agree with $\la\pmod p$ on $\ft_\Fp$,
		then these are precisely the simultaneous eigenvalues of the Euler
		operators $\ch_i^+$ that occur. Thus, we have
		\[\A^\la_{\K}|_{\mu^{-1}(\hat 0)}\cong\bigoplus_{\Bb\in \mathfrak{g}^{*, \la}_\Fp} \Qvectorbundle_\Bb\]
		In particular, given an $\A^\la_{\K}$-module $\Amodule$ over the formal neighborhood of $\mu^{-1}(\hat 0)$, we have an isomorphism of coherent sheaves
		\begin{equation} \label{cohiscoh}
			\Amodule \cong\bigoplus_{\Bb\in \mathfrak{g}^{*, \la}_\Fp}
			\Homs_{\A^\la_{\K}}(\Qvectorbundle_\Bb,\Amodule).
		\end{equation}
		
		The elements of $\A^\la_{\K}$ act on $\Qvectorbundle_\Ba$ on the left as endomorphisms of the underlying coherent sheaf; in particular,
		$\Qvectorbundle_\Ba$ naturally decomposes as the sum of the generalized
		eigenspaces for the Euler operators $\ch_i^+$. 
		
		In fact, each eigenspace for the action of $\ch_i^+$ defines a line bundle, so that the sheaf $\Qvectorbundle_\Ba$ is the sum of these line bundles. The next few results will provide a description of these line bundles. We begin with some preliminaries. Recall that $\fM_\K$ is defined as a free quotient of a $D$-stable subset of $T^*\mathbb{A}^n_\K$ by $T$. Given any character of $\Bx \in D$, the associated bundle construction defines a $D$-line bundle on $\fM_\K$. If we forget the $D$-equivariance, then the underlying line bundle depends only on the image $\bar{\Bx}$ of $\Bx$ in $\mathfrak{d}^*_\Z / \ft^{\perp}_\Z$.
		\begin{definition}
			Given $\Bx \in \mathfrak{d}^*_\Z$, let $\lin(\Bx)$ be the associated $D$-equivariant line bundle line bundle on $\fM_\K$.  We sometimes write $\lin(\bar{\Bx})$ for the underlying non-equivariant line-bundle. 
		\end{definition}

		Recall that the Weyl algebra $W_\K$ defines a coherent sheaf over the spectrum of its center, namely $(T^*\mathbb{A}_\K^{(1)})^n$. As a coherent sheaf, it is simply a direct sum of copies of the structure sheaf. Consider a monomial $m(\Bk,\Bl) := \prod_{i=1}^n\partial_i^{\Bk_i} z_i^{\Bl_i}$, viewed as a section of the structure sheaf. We have the following description of its $D$-weight $\Bx \in \mathfrak{d}^*_\Z$.

		Write $\epsilon_i$ for the generators of $\mathfrak{d}^*_\Z$, so that $\Bx = \sum_{i=1}^n \delta_i \epsilon_i$. Let $\delta_i^+$ be the maximal power of $z_i^p$ dividing $m(\Bl,\Bk)$ and $\delta_i^-$ be the maximal power of $\partial_i^p$ dividing $m(\Bl,\Bk)$. Then $\delta_i = \delta_i^+ - \delta_i^-$. In the notation of Section \ref{secisomQ}, we can write this as $\Bx = \sum_{i=1}^n \delta_i(0,\Bl - \Bk) \epsilon_i$. We conclude the following.

		\begin{lemma} \label{linebundlelemma}
			The monomial $m(\Bl,\Bk)$ descends to a section of the line bundle $\lin \Big( \sum_{i=1}^n \delta_i(0,\Bl - \Bk) \epsilon_i \Big)$ on $\mathfrak{M}^{(1)}_\K$.
		\end{lemma}

		The following proposition holds over the formal neighborhood $\pi^{-1}(\hat{0})$.
		\begin{proposition} \label{characQ}
			\begin{align}
				\displaystyle \Homs_{\A^\la_{\K}}(\Qvectorbundle_\Bb,\Qvectorbundle_\Ba) & \cong
				\lin \Big(\sum_{i=1}^n \delta_i(\Bb,\Ba)\ep_i \Big)\label{eq:2} \\
				\Qvectorbundle_\Ba & \cong \bigoplus_{\Bb\in \mathfrak{g}^{*, \la}_\Fp} \lin \Big(\sum_{i=1}^n \delta_i(\Bb,\Ba)\ep_i\Big)
			\end{align}
		\end{proposition}
		Note that the image of $\sum_{i=1}^n \delta_i(\Bb,\Ba)\ep_i$ in $\mathfrak{d}^*_\Z / \ft^{\perp}_\Z$ depends only on the class of $\Bb$ in $\mathfrak{d}^*_\Z / p \cdot \ft^{\perp}_\Z$, so that the sum is well defined. The different isomorphism classes of line bundles that appear are in bijection with the
		chambers of $\degradedchambers(\lambda)$, but not canonically so, since we must
		choose $\Ba$.
		\begin{proof}[Proof of Proposition \ref{characQ}] 
			The second isomorphism follows from the first by \eqref{cohiscoh}. To construct the first isomorphism, we recall that $\A^\la_{\K} \cong \End(\Qvectorbundle_\Ba) \cong \End(\Qvectorbundle_\Bb)$. Thus $\Homs_{\A^{\la}_\K}(\Qvectorbundle_\Ba, \Qvectorbundle_\Bb)$ is a line bundle. It has a section given by the element $m(\Bb - \Ba) \in A^{\la}_{\K}$. By Lemma \ref{linebundlelemma}, it is the line bundle defined via the associated bundle construction by the character $\Big(\sum_{i=1}^n \delta_i(\Bb,\Ba)\ep_i \Big)$ of $T$. The proposition follows.
		\end{proof}
		
		We now pass from characteristic $p$ to characteristic zero. The first step is to replace the parameter $\la \in \mathfrak{t}^*_{\mathbb{F}_p}$ by a parameter $\zeta \in \mathfrak{t}^*_\R$.
		\begin{definition} \label{defAhyperplaneperiod1}
		Let $\textgoth{A}^{\operatorname{per}}_\zeta$ be the periodic hyperplane arrangement in $\mathfrak{g}^{*,\zeta}_\R$ defined by the hyperplanes $\Dc_i=k$ for $k\in \Z$ and $i = 1,.., n$.
	\end{definition}
	  This is the arrangement obtained from \ref{defAhyperplane} by sending $p \to \infty$ and rescaling by $\frac{1}{p}$. We can define $\gradedchambers(\zeta), \degradedchambers(\zeta)$ as before. If the element $p \zeta$ lies in $\mathfrak{t}^*_\Z$, then  its image $\la$ in $\mathfrak{t}^*_{\mathbb{F}_p}$ satisfies $\gradedchambers(\la) = \gradedchambers(\zeta),  \degradedchambers(\la) = \degradedchambers(\zeta)$. The parameter $\zeta$ is smooth if and only if $\la$ is smooth.
	  
	  Let \[ \tilt^{\zeta}_\Z\cong\bigoplus_{\bar{\Bx} \in \degradedchambers(\zeta)}\lin(\bar{\Bx}).\]
		For another commutative ring $R$, let $\tilt^{\zeta}_R$ be the corresponding bundle on $\fM_{R}$, the base change to $\Spec (R)$.
		Every line bundle which appears has a canonical $\bS$-equivariant structure (induced from the trivial $\bS$-equivariant structure on $\structuresheaf_{T^*\mathbb{A}_\Z^n}$), and we endow $\tilt^{\zeta}_\Z$ with the induced $\bS$-equivariant structure.  Note that any lift of $\degradedchambers(\zeta)$ to $\gradedchambers(\zeta)$ determines a $D\times \bS$-equivariant structure, although we do not need it here. The $\bS$-weights make $\End(\tilt^{\la}_\Z)$ into a $\Z_{\geq 0}$ graded algebra.
		Let $\Bx \in \mathfrak{d}^*_\Z$. 
		
		Consider the monomial 
		\[\monomialm(\Bx) :=\prod_{x_i>0}\sz_i^{x_i}\prod_{x_i<0}\sw_i^{-x_i}.
		\]
		Note the similarity with \eqref{defminimalel}, with the key difference that we do not require $\Bx \in \ft^{\perp}_\Z$. After Hamiltonian reduction, this defines a section of $\lin(\Bx)$ with $\bS$-weight equal to $|\Bx|_1$. By the same token, it defines an element of $\Hom(\lin(\By), \lin(\By'))$ whenever $\By' = \By + \Bx$. 

		\begin{proposition}\label{prop:equiv-coh}
			For all $\la$, we have an isomorphism of graded algebras 
			$\degradedHalg^{\lambda}_\Z\cong \End_{\mathsf{Coh}(\fM)}(\tilt^\la_\Z)$
			sending $c_{\Bx,\By}\mapsto \monomialm(\By-\Bx)$ and $\mathsf{s}_i \mapsto \sz_i \sw_i$.
		\end{proposition}
		\begin{proof}
			We first check that the map is well-defined. The map $\mathsf{s}_i \mapsto \sz_i \sw_i$ is well-defined since the linear relations satisfied by $\mathsf{s}_i$ exactly match the relations on $\sz_i \sw_i$ coming from restriction to the zero fiber of the $T$-moment map. The map $c_{\Bx,\By}\mapsto \monomialm(\By-\Bx)$ is well-defined if the elements $\monomialm(\By-\Bx)$ satisfy relations (\ref{wall-cross}) and (\ref{codim1}--\ref{codim2}). Relation (\ref{wall-cross}) is satisfied if $\monomialm(\epsilon_i)\monomialm(-\epsilon_i) = \sz_i \sw_i$. This is immediate from the definition. 
			
			The relations (\ref{codim1}--\ref{codim2}) are clear from the commutativity of multiplication.  
			Thus, we have defined an algebra map $\degradedHalg^{\zeta}_\Z\to \End(\tilt^\zeta_\Z)$. This is a map of graded algebras, since both $c_{\Bx,\By}$ and $\monomialm(\By-\Bx)$ have degree $|\By-\Bx|_1$.
			
			This map is a
			surjection, since homomorphisms from one line bundle to another are
			spanned over $\Z[\sz_1\sw_1,\dots, \sz_n\sw_n]$ by $\monomialm(\Bx)$.  Since $\degradedHalg^{\zeta}_\Z$ is torsion free over $\Z$, it's enough to
			check that it is injective modulo sufficiently large primes, which
			follows from Theorems \ref{presentation1} and \ref{presentation2}.
		\end{proof}
		This allows us to understand more fully the structure of the bundle
		$\tilt^\zeta_\Z$.  Note that the bundle $\tilt^\zeta_\Z$ depends on $\zeta$, but only
		through the structure of the set $\degradedchambers(\zeta)$.
		
		\begin{proposition}\label{prop:smooth-tilting}
			The bundle $\tilt^\zeta_\Q$ is a tilting generator on $\fM_\Q$ if and only if $\zeta$ is smooth.
		\end{proposition}
		\begin{proof}
		  The bundle $\tilt^\zeta_\Q$ is tilting by Theorem \ref{frob-iso}, so we need only check if it is a generator. In order to check this over
			$\Q$, it is enough to check it modulo a large prime $p$.  Fix an affine line $Z$ in $\mathfrak{g}^{*, \zeta}_\Z$. By
			\cite[4.2]{KalDEQ}, there is an
			integer $N$, independent of $p$, such that the set of $\la\in Z_{\Fp}$ such that $\tilt^\la_{\Fp}$ is {\it not} a
			generator has size $\leq N$.  
			
			If $\zeta$ is smooth, then for all sufficiently large $p$ we can find smooth $\zeta'$ satisfying $p\zeta' \in \mathfrak{t}^*_\Z$ and such that $\degradedchambers(\zeta') = \degradedchambers(\zeta)$. It follows that moreover $\tilt^{\zeta'}_\Q = \tilt^{\zeta}_\Q$. 
			
			Since $\zeta'$ is smooth, $\la' = p \zeta'$ is also smooth. The number of $\la \in
			Z_{\Fp}$ such that $\degradedchambers(\la) = \degradedchambers(\la')$ is asymptotic
			to $Ap$ where $A$ is the volume in $Z_{\R/\Z}$ of the real points such
			that $\degradedchambers(\la/p)^\R = \degradedchambers(\zeta)^\R$.  Thus, whenever $p\geq N/A$, there must
			be some choice of $\la$ such that $\tilt^{\la/p}_\Q = \tilt^{\zeta}_\Q$ is a tilting generator.
			
			If $\zeta$ is not smooth, then $\degradedHalg^{\zeta}_{\Q}$ has fewer simple modules than at a smooth parameter, so 
			$\tilt^\zeta_\Q$ cannot be a generator.
		\end{proof}
		
		Combining the above results yields the following equivalence of categories. In the following, we view $\tilt^\zeta_\Q$ as a coherent sheaf of $\degradedHalg^{\zeta}_\Q$-modules.
		\begin{corollary}\label{cor:H-equivalence}
			For smooth $\zeta$, the adjoint functors 
			\begin{align*}
				-\Lotimes_{\degradedHalg^{\zeta}_\Q}\tilt^\zeta_\Q&\colon
				D^b(\degradedHalg_\Q^{\zeta, \operatorname{op}}\operatorname{-mod})\to
				D^b(\Coh(\fM_\Q))\\
				\RHom(\tilt^\zeta_\Q,-)&\colon D^b(\Coh(\fM_\Q))\to
				D^b(\degradedHalg_\Q^{\zeta, \operatorname{op}}\operatorname{-mod})
			\end{align*}
			define equivalences between the derived categories of coherent sheaves over
			$\fM_\Q$ and finitely generated right $\degradedHalg^{\zeta}_\Q$-modules. 
			
			The same functors define an equivalence between the derived categories of graded modules and equivariant sheaves:
			\begin{align*}
				-\Lotimes_{\degradedHalg^{\zeta}_\Q}\tilt^\zeta_\Q&\colon
				D^b(\degradedHalg_\Q^{\zeta, \operatorname{op}}\operatorname{-gmod})\to
				D^b(\Coh_{\mathbb{G}_m}(\fM_\Q))\\
				\RHom(\tilt^\zeta_\Q,-)&\colon D^b(\Coh_{\mathbb{G}_m}(\fM_\Q))\to
				D^b(\degradedHalg_\Q^{\zeta, \operatorname{op}}\operatorname{-gmod})
			\end{align*}
			
			Finally, identical statements hold if we replace $\degradedHalg$ by $\gradedHalg$, and replace $\Coh(\fM_\Q)$ by $\Coh_{G}(\fM_\Q)$ and $\Coh_{\mathbb{G}_m}(\fM_\Q)$ by $\Coh_{\mathbb{G}_m \times G}(\fM_\Q)$.
		\end{corollary}
		Since $\degradedHalg^{\zeta, \operatorname{op}}_\Q$ is defined as a path algebra modulo relations, its graded simple modules are just the 1-dimensional modules $L^{\operatorname{op}}_\Bx := \Hom(\oplus_{\By\in\degradedchambers(\zeta)} L_{\By},L_\Bx)$;  we denote the corresponding complexes of coherent sheaves by \[\mathscr{L}_{\Bx} := L^{\operatorname{op}}_\Bx \Lotimes_{\degradedHalg^{\zeta}_\Q}\tilt^\zeta_\Q.\] 
		
		The induced $t$-structure on $D^b(\mathsf{Coh}_{\mathbb{G}_m}(\fM))$ is what's often called an ``exotic $t$-structure.''
		
		We also have a Koszul dual description of coherent sheaves as dg-modules over the
		quadratic dual $\degradedHalg_{\zeta, \Q}^!$.
		Since $\degradedHalg^{\zeta}_\Q$ is an infinite dimensional algebra, we have to be a
		bit careful about finiteness properties here.  We let $\Coh(\fM_\Q)_o$
		be the category of coherent sheaves set theoretically supported on the
		fiber $\pi^{-1}(o)$, and 
		$\degradedHalg_\Q^{\zeta, \operatorname{op}}\mmod_o$ denote the corresponding
		category of $\degradedHalg_\Q^{\operatorname{op}}$-modules;  one
		characterization of these modules is that for some integer $N$, they
		are killed by all algebra elements of degree $>N$.

		\begin{lemma}
			A complex of coherent sheaves lies in $D^b(\Coh_{\mathbb{G}_m}(\fM_\Q)_{o})$ if and only if it is in the triangulated envelope of the complexes $\mathscr{L}_{\Bx}$.
		\end{lemma}
		\begin{proof}
			The complex $M$ is in the subcategory $D^b(\Coh_{\mathbb{G}_m}(\fM_\Q)_{o})$ if and only if it is sent to a complex of modules over $\degradedHalg_\Q^{\zeta, \operatorname{op}}$ killed up to homotopy by a sufficiently high power of the 2-sided ideal generated by the elements of positive degree in $H^0(\fM_\Q, \structuresheaf_{\fM_\Q})$. This ideal contains all elements of sufficiently large degree (since the quotient by it is finite dimensional and graded), so each cohomology module of the image is a finite extension of the graded simples.  Thus the complex itself is an iterated extension of shifts of these modules.
		\end{proof}
		
		Let
		$\degradedHalg_{\zeta, \Q}^!\perf$ be the category of perfect dg-modules over
		$\degradedHalg_{\zeta, \Q}^!$. 
		As usual, we will abuse notation and let $D^b(\Coh(\fM_\Q))$
		to denote the usual dg-enhancement of this category, and similarly with $D^b(\Coh_{\mathbb{G}_m}(\fM_\Q))$.  Combining  the
		equivalence of Corollary \ref{cor:H-equivalence} with Koszul duality:
		\begin{proposition}\label{prop:Hbang-coh}\hfill
			\begin{enumerate}
				\item   We have an equivalence of dg-categories $\degradedHalg_{\zeta, \Q}^!\perf\cong D^b(\Coh(\fM_\Q)_{o})$, induced by $\oplus_{\By\in\degradedchambers(\zeta)} \Ext(\mathscr{L}_{\By},-)$.
				\item   We have an equivalence of dg-categories $D^b_{\operatorname{perf}}(\degradedHalg_{\zeta, \Q}^!\operatorname{-gmod})\cong D^b(\Coh_{\mathbb{G}_m}(\fM_\Q)_{o})$, induced by $\oplus_{\By\in\degradedchambers(\zeta)} \Ext(\mathscr{L}_{\By},-)$.
			\end{enumerate}
		\end{proposition}
		\begin{proof}
			Since elements of $D^b(\Coh(\fM_\Q)_{o})$ are finite extensions of $\mathscr{L}_{\By}$ for different $\By$, they are sent by $\oplus_{\By\in\degradedchambers(\zeta)} \Ext(\mathscr{L}_{\By},-)$ to perfect complexes and vice versa. This proves item (1).
			
			Item (2) is just the graded version of this statement, which corresponds to Corollary \ref{cor:H-equivalence} via the usual Koszul duality (\cite[Thm. 2.12.1]{BGS96}).  
		\end{proof}

				This shows that smooth parameters also have an interpretation in
				terms of $\A^\la_{\K}$; this is effectively a restatement of Proposition \ref{prop:smooth-tilting}, so we will not include a proof.
				\begin{proposition}
					The functor $\mathbb{R}\Gamma : D^b(\A^\la_{\K}\mmod_0) \mapsto D^b(A^\la_{\K}\mmod_0)$ is an equivalence of categories if
					and only if the parameter $\la$ is smooth.
				\end{proposition}
				\excise{\begin{proof}
						First note that since both are generic properties, for some $p\gg 0$, there is a parameter $\la$ that is smooth such that $\mathbb{R}\Gamma$ is an equivalence.  This shows that the dimension of the $K$-theory of $\fM$ is the number of bases (which is also easily shown using Chern character, and the Betti numbers of $\fM$).
						
						Assume that the functor $\mathbb{R}\Gamma$ is an equivalence.
						The Grothendieck group of  $D^b(\A^\la_{\K}\mmod_0)$ is the number of bases, so the same is true of $A^{\la}_{\K}\mmod_0$.  This shows that $\la$ is smooth.
						
						On the other hand, assume that $\la$ is smooth.
				\end{proof}}

				\section{Mirror symmetry via microlocal sheaves}

				In the previous sections, the conical $\mathbb{G}_m$-action on hypertoric varieties played a key role in our study of coherent sheaves. This is what allowed us to construct a tilting bundle based on a quantization in characteristic $p$.  This conic action also plays a crucial role in the study of enumerative invariants of these varieties \cite{BMO, MaOk, McS}. The quantum connection and quantum cohomology which appear in those papers lose almost all of their interesting features if one does not work equivariantly with respect to the conic action. We are thus interested in a version of mirror symmetry which remembers this conic action. 
				
				We expect the relevant A-model category to be a subcategory of a Fukaya category of the {\bf Dolbeault hypertoric manifold} $\Dol$, built from Lagrangian branes endowed with an extra structure corresponding to the conical $\mathbb{G}_m$-action on $\mathfrak{M}$. However, rather than working directly with the Fukaya category, we will replace it below by a category of DQ-modules on $\Dol$ . The calculations presented there should also be valid in the Fukaya category. The reader is referred to the sequel \cite{GMW} to this paper for more discussion of this point. 
				
				After defining the relevant spaces and categories of DQ modules, we state our main equivalence in Theorem \ref{mainthm-Hodge} and Corollary \ref{centralcorollary}. 
				
				There are a few obvious related questions. What corresponds to the category of all (not necessarily equivariant) coherent sheaves on $\mathfrak{M}$? What corresponds to the full category of $DQ$-modules of $\Dol$? We plan to address these questions in a future publication.
				
				\subsection{Dolbeault hypertoric manifolds} 
				In this section, we introduce Dolbeault hypertoric manifolds, whose definition we learned from unpublished work of Hausel and Proudfoot.
				
				Dolbeault hypertoric manifolds are complex manifolds attached to the
				data of a toric hyperplane arrangement (i.e. a collection of
				codimension one affine subtori), in much the same way that an additive hypertoric variety is attached to an affine hyperplane arrangement, and a toric variety is attached to a polytope. They carry a complex symplectic form, and a proper fibration whose generic fibers are complex lagrangian abelian varieties. 
				
				Our construction of Dolbeault manifolds parallels the construction of toric varieties as Hamiltonian reductions of powers of a basic building block. 
				
				For toric varieties, this building block is $\C$ with the usual Hamiltonian action of $\uu$. Its polytope is a ray in $\R$. Other toric varieties are constructed by taking the Hamiltonian reduction of $\C^n$ by a subtorus of $\uu^n$. Additive hypertoric varieties are similarly constructed from the basic building block $T^*\C$ with its hyperhamiltonian action of $\uu$. The affine hyperplane arrangement associated to this building block is a single point in $\R$. For Dolbeault manifolds, our basic building block will be the Tate curve $\mathfrak{Z}$ with a (quasi)-hyperhamiltonian action of $\uu$. Its toric hyperplane arrangement is a single point in $\uu$.  
				
				We give a construction of $\mathfrak{Z}$ suited to our purposes below, culminating in Definition \ref{deftatecurve}.
				
				Let $\C^* = \operatorname{Spec}\C[q,q^{-1}]$, and let $\DD^*$ be the
				punctured disk defined by $0 < q < 1$. Let $\mathfrak{Z}^*$ be the family of
				elliptic curves over $\DD^*$ defined by
				$( \C^* \times \DD^* ) / \mathbb{Z}$, where $1 \in \mathbb{Z}$ acts by
				multiplication by $q \times 1$. 
				
				We will define an extension of $\mathfrak{Z}^*$ to a family $\ZDol$ over $\DD$ with central fiber equal to a nodal elliptic curve.

				Let $\basicW_n := \Spec \C[x,y]$ for $n \in \Z$.  Consider the birational map $f: \basicW_n \to \basicW_{n+1}$ defined by $f^*(x) = \frac{1}{y}, f^*(y) = xy^2$. 
				This defines an automorphism of the subspace $\basicW_0\setminus \{xy=0\}$, and identifies the $y$-axis in $\basicW_n$ with the $x$-axis in $\basicW_{n+1}$ birationally, so they glue to a $\mathbb{P}^1$.  If we let $q:=xy$, then we can rewrite this automorphism as $(x,y)\mapsto (q^{-1}x,qy)$.    
				Note that this map preserves the product $xy$ and commutes with the $\C^*$-action on $\basicW_n$ defined by $\tau\cdot x=\tau x, \tau\cdot y=\tau^{-1}y$; we let $\mathbb{T}$ denote this copy of $\C^*$. 
				\begin{definition}
					Let $\basicW$  be the quotient of the union $\bigsqcup_{n \in \Z} \basicW_n$ by the equivalence relations that identify  the points $x \in \basicW_n$ and  $f(x) \in \basicW_{n+1}$.  
				\end{definition}
				The variety $\basicW$ is smooth  of infinite type, with a map $q := xy: \basicW \to \C$ and an action of $\C^*$ preserving the fibers of $q$.  The map $\basicW_0 \setminus \{xy=0\} \to \basicW \setminus q^{-1}(0)$ is easily checked to be an isomorphism. 
				
				$\basicW$ carries a $\Z$-action defined by sending $\basicW_n$ to $\basicW_{n+1}$ via the identity map. The action of $n\in \Z$ is the unique extension of the automorphism of $\basicW_0 \setminus \{xy=0\}$ given by $(x,y)\mapsto (q^{-n}x,q^ny)$.  Thus, $n$ fixes a point $(x,y)$ if and only if $q$ is an $n$th root of unity.  In particular, the action of $\Z$ on 
				\begin{equation} \widetilde{\mathfrak{Z}} := q^{-1}(\mathbb{D}) \end{equation} is free.  Combining this with the paragraph above, we see that $q^{-1}(\DD^*)=\{(x,y)\in \basicW_0\mid xy\in \DD^*\}$; since we can choose $x\in \C^*$ and $q\in \DD^*$, with $y=q/x$ uniquely determined, we have an isomorphism $q^{-1}(\DD^*)\cong \C^*\times \DD^*$.  Note that transported by this isomorphism, the $\C^*$-action we have defined acts by scalar multiplication on the first factor, and trivially on the second.

				Thus, we obtain the following commutative diagram of spaces: 
				\begin{equation}
					\tikz[->,very thick,baseline]{ \matrix[row sep=10mm,column sep=20mm,ampersand
						replacement=\&]{ \node (a) {$\C^* \times \DD^*\cong q^{-1}(\DD^*)$}; \& \node (c)
							{$\ZDolcov$}; \\
							\node (b) {$\DD^*$}; \& \node (d) {$\DD$};\\
						}; \draw (a) -- (c) node[above,midway]{ $j$}; \draw (b)
						--(d); \draw (a) --(b) 
						node[left,midway]{$q$}; \draw (c)--(d)
						node[right,midway]{$q$};} 
				\end{equation}
				The fiber $\widetilde{\mathfrak{Z}}_0 := q^{-1}(0)$ is an infinite chain of $\C\mathbb{P}^1$'s with each link connected to the next by a single node. 
				
				The action of $\C^*$ on $\widetilde{\mathfrak{Z}}_0$ scales each component, matching the usual action of scalars on $\mathbb{CP}^1$, thought of as the Riemann sphere.  The action of the generator of $\Z$ translates the chain by one link.
				
				\begin{definition} \label{deftatecurve}
					Let $\ZDol := \ZDolcov / \Z$. 
				\end{definition}
				The manifold $\ZDol$ will be our basic building block. We now study various group actions and moment maps for $\ZDol$, in order to eventually define a symplectic reduction of $\ZDol^n$. 
				
				The action of $\C^*$ on $\ZDolcov$ descends to an action on $\ZDol$; note that on any nonzero fiber of the map to $\DD$, it factors through a free  action of the quotient group $\C^* / q^{\Z}$, which is transitive unless $q=0$. Thus the generic fiber of $q$ is an elliptic curve. The fiber $\mathfrak{Z}_0 := q^{-1}(0)$ is a nodal elliptic curve. We write $\mathbf{n}$ for the node.  
				
				The action of $\uu \subset \C^*$ on $\ZDolcov$ is Hamiltonian with respect to a hyperkahler symplectic form and metric described in \cite[Prop. 3.2]{GrWi}, where one also finds a description of the $\Z$-equivariant moment map. 
				This moment map descends to
				\[ \mu : \ZDol \to \mathbb{R} / \mathbb{Z} = \mathbf{\uu}. \]
				Hence $\mu$ is the quasi-hamiltonian moment map for the action of
				$\uu$ on $\ZDol$. We may arrange that $\mu(\mathbf{n}) = \mathbf{1}
				\in \mathbf{\uu}$.  The nodal fiber $\mathfrak{Z}_0$ is the image of a $\uu$-equivariant immersion
				$\iota\colon \mathbb{CP}^1\to \ZDol$, which
				is an embedding except that $0$ and $\infty$ are both sent to
				$\mathbf{n}$.  We have a commutative diagram:
				\begin{equation}
					\tikz[->,very thick,baseline]{ \matrix[row sep=10mm,column sep=15mm,ampersand
						replacement=\&]{ \node (a) {$\mathbb{CP}^1$}; \& \node (c)
							{$\ZDol$}; \\
							\node (b) {$[0,1]$}; \& \node (d) {$\R/\Z$};\\
						}; \draw (a) -- (c) node[above,midway]{$\iota$}; \draw (b)
						--(d); \draw (a) --(b) 
						node[left,midway]{$\mu_{\mathbb{CP}^1}$}; \draw (c)--(d)
						node[right,midway]{$\mu$};} \qquad
					\mu_{\mathbb{CP}^1}(z)=\frac{|z|^2}{1+|z|^2}\colon \mathbb{CP}^1 \to
					[0,1].\label{eq:moment-map-1}
				\end{equation}
				
				The action of $\uu$ and map $\mu \times q$ form a kind of
				`multiplicative hyperkahler hamiltonian action' of $\uu$. In
				particular, $(\mu\times q)^{-1}(a,b)$ is a single $\uu$ orbit, which
				is free unless $a=1,b=0$, which case it is just the node
				$\mathbf{n}$.  It's worth comparing this with the hyperk\"ahler moment
				map on $T^*\C$ for the action of $\uu$: this is given by the
				map \[T^*\C\to \R\times \C\qquad 
				(z,w)\mapsto (|z|^2-|w|^2,zw)\]  The fibers over non-zero elements of
				$\R\times \C$ are circles, and the fiber over zero is the origin.  In a
				neighborhood of $\mathbf{n}$, $\mu\times q$ is analytically isomorphic
				to this map.  
				
				Without seeking to formalize the notion, we will simply mimic the notion of hyperkhaler reduction in
				this setting.  Recall that a hypertoric variety $\mathfrak{M}$ is defined using
				an embedding of tori $(\C^*)^k = T \to D = (\C^*)^n$. Let
				$T_{\R},D_{\R}$ be the corresponding compact tori in these groups, and
				$T^{\vee}{\R}\cong \ft_{\R}^*/\ft_{\Z}^*$ the Langlands dual torus; the
				usual inner product induces an isomorphism $D_{\R}\cong D_{\R}^{\vee}$
				which we will leave implicit.
				Thus, we have an
				action of $T_{\R}$ on $\ZDol^n$ and a $T_{\R}$-invariant map \begin{equation} \Phi:
					\ZDol^n \to T_{\R}^\vee\times \ft^*. \end{equation}  
				Given
				$\zeta \in  T_{\R}^\vee$, let $\zeta' = \zeta \times 0 \in
				T_{\R}^\vee\times \ft^*$. For generic $\zeta$ the action of $T_{\R}$ on $
				\Phi^{-1}(\zeta')$ is locally free.
				
				For the rest of this paper, we make the additional assumption that the torus embedding $T \to D$ is {\em unimodular}, meaning that if $e_k$ are the coordinate basis of $\mathfrak{d}_{\Z}$, then any collection of $e_k$ whose image spans $\mathfrak{d}_{\Q}/ \mathfrak{t}_{\Q}$ also spans  $\mathfrak{d}_{\Z}/ \mathfrak{t}_{\Z}$. As with toric varieties, this guarantees that for generic $\zeta$ the action of $T$ on $\Phi^{-1}(\zeta')$ is actually free. We expect that this assumption can be lifted without significant difficulties, but it will help alleviate notation in what follows.
				
				The following definition is due to Hausel and Proudfoot.
				\begin{definition} \label{def:dolbhyper}
					Let $\Dol := \Phi^{-1}(\zeta')/T_{\R}.$ 
				\end{definition}
				\begin{proposition}
					$\Dol$ is a $2d = 2n-2k$ dimensional holomorphic symplectic
					manifold.
					
				\end{proposition}

				We will also need to consider the universal cover $\Dolcov$; this can also be constructed as a reduction. We have a hyperk\"ahler moment map $\tilde{\Phi}\colon \ZDolcov^n\to \ft_{\R}^*\oplus \ft^*$. Let $\tilde{\zeta}'$ be a preimage of $\zeta'$. 
				\begin{definition} \label{def:dolcoverbhyper}
					Let $\Dolcov :=\tilde{\Phi}^{-1}(\tilde{\zeta}')/T_{\R}$. 
				\end{definition}
				$\Dolcov$ carries a natural action of $\mathfrak{g}_\Z^*$, the subgroup of $\Z^n$ which preserves the level $\tilde{\Phi}^{-1}(\tilde{\zeta}')$. The quotient by this map is $\Dol$, and the quotient map $\nu \colon \Dolcov \to \Dol$ is a universal cover. Note that $\Dolcov$ is a (non-multiplicative) hyperk\"ahler reduction, and the action of $\mathfrak{g}_\Z^*$ preserves the resulting complex symplectic form. This gives one way of defining the complex symplectic form on $\Dol$. 
				
				The $T$ action on $\ZDolcov^n$ and the holomorphic part of the hyperk\"ahler moment map $\tilde{\Phi}_\C$ both extend to the infinite type algebraic variety $\basicW^n$. 
				\begin{definition}
					Let $\Dolcov^{\operatorname{alg}}$ be the holomorphic symplectic reduction $\tilde{\Phi}_\C^{-1}(0) /\!/\!_{\tilde{\zeta}'} T$, where we take the GIT quotient by $T$ with linearization determined by $\tilde{\zeta}'$. 
				\end{definition} 
				As opposed to $\Dolcov$, the space $\Dolcov^{\operatorname{alg}}$ is naturally an infinite type but finite dimensional algebraic variety. Its construction and properties are described in detail in \cite{groechenighypertoric}. It contains the complex manifold $\Dolcov$ as an (analytic) open subset.
				
				Let $q_{\Dol}: \Dol \to \ft_\C^\perp \cong  \mathfrak{g}_\C^*$ be the map induced by $q^n :
				\ZDol^n \to \C^n \cong \fd^*_\C$. Its fibers are complex Lagrangians. The action of $\C^*$ on $\ZDol$ defines an action of $G = D/T$ on
				$\Dol$, which preserves the complex symplectic form and the fibers of the map $q_{\Dol}$, and acts transitively on fibers over values $(q_1,\dots, q_n)$ with $q_i\neq 0$ for all $i$.  Such fibers are $d$-dimensional abelian varieties.    
				
				\begin{definition}
					We define the {\em core} of $\Dol$ to be $\mathfrak{C} := q_{\Dol}^{-1}(0)$, and denote by $\widetilde{\mathfrak{C}}$ its preimage in $\Dolcov$.
				\end{definition}
				
			We thus have inclusions $\widetilde{\mathfrak{C}} \xrightarrow{closed} \Dolcov \xrightarrow{open} \Dolcov^{\operatorname{alg}}$. The lattice $\mathfrak{g}^*_{\mathbb{Z}}$ acts compatibly on all three spaces, but the quotient only makes sense for the first two, where it gives the inclusion $\mathfrak{C} \to \Dol$.
    
				
				Whereas $\Dol$ is merely a complex manifold, we will see that $\mathfrak{C}$ is naturally an algebraic variety. It is a free quotient of $\Dolcov$, whose components, as we shall see, are smooth complex Lagrangians. We can give an explicit description of $\mathfrak{C}$ as follows, in the spirit of the combinatorial description of toric varieties in terms of their moment polytopes. In our setting, polytopes are replaced by toroidal arrangements.
				
				We have the map $D^{\duall}_{\R} \to T^{\duall}_{\R}$; let $G^{\duall, \zeta}_\R$ be the preimage of $\zeta$. It is a torsor over $G_{\R}^{\duall}$.  The preimage of $G^{\duall, \zeta}_\R$ under the quotient $\mathfrak{d}^*_\R \to D^{\duall}_\R$ is given by $\mathfrak{g}_\R^{\duall, \zeta} := \widetilde{\zeta}' + \mathfrak{g}^*_{\R}$. 
				\begin{definition}
					Let $\textgoth{B}^{\operatorname{per}}_\zeta \subset \mathfrak{g}^{*,\zeta}_\R$ be the periodic hyperplane arrangement defined by the preimage of $\textgoth{B}^{\operatorname{tor}}_\zeta$ in $\widetilde{\zeta}' + \mathfrak{g}^*_{\R}$. Let $\gradedchambers^\R(\zeta)$ be the set of chambers of $\textgoth{B}^{\operatorname{per}}_\zeta$. We write $\Delta^\R_\Bx \subset \mathfrak{g}^{*,\zeta}_\R$ for the (closed) chamber indexed by $\Bx \in \gradedchambers^\R(\zeta)$.
				\end{definition}
				
				As in section \ref{interpretation}, let $\mathfrak{X}_\Bx$ be the toric variety obtained from the polytope $\Delta^\R_\Bx$ by the Delzant construction. 
				
				\begin{proposition} \label{prop:corecharactper}\hfill 
					\begin{enumerate}
						\item \label{strone} The irreducible components of $\widetilde{\mathfrak{C}}$ are smooth toric varieties $\mathfrak{X}_\Bx$ indexed by $\Bx \in \gradedchambers^\R(\zeta)$. 
						\item \label{strtwo} The intersection $\mathfrak{X}_\Bx \cap \mathfrak{X}_\By$ is the toric subvariety of either component indexed by $\Delta^\R_\Bx \cap \Delta^\R_\By$.
						\item \label{strthree} The image under the $G_\R$-moment map of $\mathfrak{X}_\Bx$ is precisely the polytope $\Delta^\R_{\Bx}$.
						\item  \label{strfour}  All components meet with normal crossings. 
					\end{enumerate}
				\end{proposition}
				
				\begin{proof}
					We begin by noting that $\widetilde{\mathfrak{C}}$ is the image in $\Dolcov$ of $\Phi^{-1}(\zeta')\cap \widetilde{\mathfrak{Z}}_0^n$. The irreducible components of $\widetilde{\mathfrak{Z}}_0^n$ are copies of $(\mathbb{CP}^1)^n$ indexed by $\Bx \in \Z^n$. The moment map $\mu^n : \widetilde{\mathfrak{Z}}_0^n \to \R^n$, restricted to the component $(\mathbb{CP}^1)^n_\Bx$, has image the translation $[0,1]^n_\Bx$ of the unit cube by $\Bx$. We write $\tilde{\Phi}_{\Bx} \colon (\mathbb{C}\mathbb{P}^1)^n\to \ft_{\R}^\vee$ for the restriction the $T_\R$ moment map. It is given by be the composition of $\mu_{\C\mathbb{P}^1}^n \colon
					(\mathbb{C}\mathbb{P}^1)^n\to [0,1]^n_\Bx$ with the projection $p: [0,1]^n_\Bx \subset \fd^*_\R \to
					\ft_{\R}^\vee$.

					The preimage $p^{-1}(\zeta) \subset [0,1]_\Bx^n$ is a polytope, given by $\mathfrak{g}_\R^{\duall, \zeta} \cap [0,1]_\Bx^n$. It is non-empty precisely when $\Bx \in \gradedchambers^\R(\zeta)$, in which case it is the chamber $ \Delta^{\R}_\Bx $.
					
					The irreducible components of $\widetilde{\mathfrak{C}}$ are thus the quotients $\tilde{\Phi}_{\Bx}^{-1}(\zeta) / T_\R$ for $\Bx \in \gradedchambers^\R(\zeta)$. The claims \eqref{strone}, \eqref{strtwo} and \eqref{strthree} now follow from standard toric geometry.
					
					Claim \eqref{strfour} follows from the corresponding property for $\widetilde{\mathfrak{Z}}^n_0$. In fact, the
					singular points of $\mathfrak{C}$ are analytically locally a product of $m$ nodes,
					and a $d-m$-dimensional affine space. 
					
				\end{proof}

				\begin{definition} \label{def:torBarrange}
					Let $\textgoth{B}^{\operatorname{tor}}_\zeta \subset G^{\duall, \zeta}_\R$ be the toric hyperplane arrangement defined by the coordinate subtori of $D^{\duall}_\R$. Let $\degradedchambers^\R(\zeta)$ be the set of chambers of $\textgoth{B}^{\operatorname{tor}}_\zeta$. Given $\Bx \in \degradedchambers^\R(\zeta)$, we write $\Delta^\R_\Bx \subset G^{\duall, \zeta}_\R$ for the corresponding chamber. 
				\end{definition}
				
				The toric arrangement $\textgoth{B}^{\operatorname{tor}}_\zeta$ is simply the quotient of the periodic arrangement  $\textgoth{B}^{\operatorname{per}}_\zeta$ by the action of the lattice $\mathfrak{g}^*_\Z$. The restriction of the quotient map to a fixed chamber $\Delta_\Bx \subset \mathfrak{g}^{*, \zeta}_\R$ is one-to-one on the interior, but may identify certain smaller strata. Correspondingly, the composition $\mathfrak{X}_\Bx \to \Dolcov \to \Dol$ is in general only an immersion. The following is easily deduced from \ref{prop:corecharactper}.
				
				\begin{proposition} \label{prop:corecharact}\hfill 
					\begin{enumerate}
						\item \label{degradedstrone} The irreducible components of $\mathfrak{C}$ are immersed toric varieties $\bar{\mathfrak{X}}_{\Bx}$ indexed by $\Bx \in \degradedchambers^\R(\zeta)$. Any lift of $\Bx \in \degradedchambers^\R(\zeta)$ to $\gradedchambers^\R(\zeta)$ determines a birational map $\mathfrak{X}_\Bx \to \bar{\mathfrak{X}}_{\Bx}$ with finite fibers.
						\item \label{degradedstrtwo} The intersection $\bar{\mathfrak{X}}_\Bx \cap \bar{\mathfrak{X}}_\By$ is the (immersed) toric subvariety of either component indexed by $\Delta^\R_\Bx \cap \Delta^\R_\By$.
						\item \label{degradedstrthree}The image under the $G_\R$-moment map of $\bar{\mathfrak{X}}_\Bx$ is precisely the toric chamber $\Delta^\R_{\Bx}$.
						\item  \label{degradedstrfour}  All components meet with normal crossings. 
					\end{enumerate}
				\end{proposition}

				\begin{example}
					We continue with Example \ref{classicexample}.  In this case, for generic $\tilde{\zeta}'$, we arrive at a picture like in \eqref{eq:P2-chambers}.  The 3 chambers shown in total there correspond to the 3 core components of $C$: two of these are isomorphic to $\mathbb{CP}^2$, and one to $\mathbb{CP}^1\times \mathbb{CP}^1$ blown up at $(0,0)$ and $(\infty,\infty)$.  We join these by joining the lines at $\infty$ in the first $\mathbb{CP}^2$ to the exceptional locus of the blow up at $(0,0)$, and its coordinate lines to the unique lifts of  $\mathbb{CP}^1\times \{\infty\}$ and $\{\infty\}\times \mathbb{CP}^1 $ to lines in the blowup (note that in the blowup, these lines don't intersect).  With the second $\mathbb{CP}^2$ we do the same gluing with $0$ and $\infty$ reversed.  
					
					Note that in $\mathfrak{C}$, the two $\mathbb{CP}^2$'s are embedded, but the third component is only immersed: it intersects itself transversely at each torus fixed point.  
				\end{example}
				
				\subsection{Weinstein neighborhoods and scaling actions}
				
				Let $\Bx$ be a chamber of the periodic arrangement $\textgoth{B}^{\operatorname{per}}_\zeta$, and let $\mathfrak{X}_\Bx$ be a component of the periodic core. We will construct an open neighborhood $\algcov_\Bx \cong T^*\mathfrak{X}_\Bx$ of $\mathfrak{X}_\Bx$ in $\algcov$. Its intersection 
				
				$$ \Dolcov_\Bx := \algcov_\Bx \cap \Dolcov$$ is an open neighborhood of $\mathfrak{X}_\Bx$ in $\Dolcov$, which maps by an immersion to an open neighborhood of $\bar{\mathfrak{X}}_\Bx$ in $\Dol$.  
				
				Consider the union $\basicW_0\cup \basicW_1\subset \basicW$. This is a Zariski open subset of $\basicW$ isomorphic to $T^*\C\mathbb{P}^1$. Let $\ZDolcov_0$ be its intersection with $\ZDolcov$. This is an open submanifold, isomorphic to a tubular neighborhood of $\C\mathbb{P}^1$ in its cotangent bundle. 
				
				These identifications map the function $q$ to the function induced by the vector field $z\frac{d}{dz}$ for $z$ the usual coordinate on $\mathbb{CP}^1$. The induced map $\tilde{U} \to \ZDol$ is an immersion.

				Applying the action of $\Z$ gives neighborhoods $\basicW_k$ of each component of $q^{-1}(0) \subset \basicW$. Repeating the same construction for the product $\basicW^n$, we obtain for each $\Bx \in \Z^n$ an open neighborhood $\tilde{\basicW}_{\Bx}$ of $(\mathbb{CP}^1)^n_{\Bx}$ in $\basicW^n$, isomorphic to $T^*(\mathbb{CP}^1)^n$. This neighborhood is preserved by the (complex hamiltonian) action of $T_\R$. Consider its complex symplectic reduction $$\algcov_{\Bx} := \basicW_{\Bx} /\!\!/_{\tilde{\zeta}'} T $$
				It is an open neighborhood of $\mathfrak{X}_{\Bx}$ in $\algcov$, naturally symplectomorphic to $T^*\mathfrak{X}_{\Bx}$. Intersecting with $\Dolcov \subset \algcov$, we obtain an open neighborhood $\Dolcov_\Bx$ of the zero section in $T^*\mathfrak{X}_\Bx$ mapping by a symplectic immersion
				\begin{equation} \label{eq:iotaimmersion} \iota_\Bx\colon \Dolcov_\Bx \to \Dol \end{equation} to an open subset of $\Dol$ extending the immersion $\mathfrak{X}_\Bx\to \Dol$ and a corresponding lift $\tilde\iota_\Bx\colon  \Dolcov_\Bx \to \Dolcov$, which is a symplectomorphism onto an open subset of $\Dolcov$. The set of such lifts is a torsor over $\mathfrak{g}^*_\Z$.

				\subsection{Scaling actions}
				The scaling $\C^*$-action on $T^*\mathfrak{X}_{\Bx}$ extends to an action of $\C^*$ on $\algcov$, which does not preserve $\Dolcov$. We first describe this action in the basic case of $\basicW$. Fix $p\in \Z$ and let $\mathbb{S}_p$ be the copy of $\C^*$ which acts on $\basicW_k$  giving $x$ degree $1-k+p$ and $y$ degree $k-p$. One can easily check on that this action descends to an action on $\basicW$ and gives the Poisson bracket degree one. On $\basicW_p\cup \basicW_{p+1}\cong T^*\mathbb{CP}^1$, it acts by the scaling action on the fibers. Note that $\mathbb{S}_p$ does not preserve the open subset $\ZDolcov \subset \basicW$. 
				
				The action of $\mathbb{S}_p\times \mathbb{T}$ does not commute with the translation action of $\Z$. Instead, the $\Z$-action intertwines the actions of $\mathbb{S}_p\times \mathbb{T}$ for different $p$. In particular, all such actions are given by precomposing an isomorphism $\mathbb{S}_p\times \mathbb{T} \to \mathbb{S}_0 \times \mathbb{T}$ with the action of the latter torus on $\basicW$.

				We can upgrade all these structures to the general case:  for each $\Bx$, we have a copy $\mathbb{S}_{\Bx}$ of $\C^*$ which acts on $\algcov$ such that on $\algcov_{\Bx} \subset T^*\mathfrak{X}_{\Bx}$ it matches the scaling action. As before, these actions do not commute with the $\mathfrak{g}^*_\Z$-action. Instead, they are intertwined by this action. In particular, all such actions factor through an isomorphism $\mathbb{S}_\Bx \times G \to \mathbb{S}_0 \times G$ with the action of the latter torus on $\algcov$. We make the following (purely notational) definition, to emphasise this independence of choices.
				
				\begin{definition}
					Let $\mathbb{S}G := \mathbb{S}_0 \times G$, with its action on $\algcov$.
				\end{definition}

				\subsection{Other flavors of multiplicative hypertoric manifold}
				
				In this paper, starting from the data of an embedding of tori $T \to \mathbb{G}_m^n$, we have constructed both an additive hypertoric variety $\mathfrak{M}$ and a Dolbeault hypertoric manifold $\Dol$. We view the latter as a multiplicative analogue of $\mathfrak{M}$. One can attach to the same data another, better known multiplicative analogue $\Bet$, which however plays only a motivational role in this paper.  For a definition, see \cite{Ganev18}. $\Bet$ is often simply known as a multiplicative hypertoric variety. For generic parameters, it is a smooth affine variety, of the same dimension as $\mathfrak{M}$ and $\Dol$. In fact, work of Zsuzsanna Dancso, Vivek Shende and the first author \cite{dancso2019deletioncontraction} constructs a smooth open embedding $\Dol \to \Bet$, such that $\Bet$ retracts smoothly onto the image. The embedding does not, however, respect complex structures; for instance, the complex Lagrangians considered here map to real submanifolds of the multiplicative hypertoric variety. Instead, $\Bet$ and $\Dol$ play roles analogous to the Betti and Dolbeault moduli of a curve.
				
				In the sequel \cite{GMW} to this paper, joint with Ben Gammage, we show that the core $\mathfrak{C} \subset \Dol$ becomes the Liouville skeleton of $\Bet$, thought of as a Liouville manifold with respect to the affine Liouville structure. Microlocal sheaves on this skeleton compute the wrapped Fukaya category of $\Bet$. In the next section, we will introduce a category of deformation quantization modules on $\Dol$, which roughly corresponds to microlocal sheaves on $\Bet$ with an extra $\mathbb{G}_m$ equivariant structure. This helps place our main results in the usual context of homological mirror symmetry. The relationship between the two papers is explained in more detail in \cite{GMW}.

				\subsection{Deformation quantization of \texorpdfstring{$\Dol$}{Y}} \label{sec:defquant}
				In the next few sections, we define a deformation quantization of $\Dol$ over $\C((\hbar^{\nicefrac{1}{2}}))$, and compare modules over this quantization with the category $A^{\la}_{\K}\mmod_o$ from the first half of the paper.  We'll also discuss how the structure of $\mathbb{G}_m$-equivariance of coherent sheaves can be recaptured by considering a category $\mhm$ of deformation quantization modules equipped with the additional structure of a `microlocal mixed Hodge module.'  
				
				Consider the sheaf of analytic functions $\EuScript{O}_{\basicW_n}$ on $\basicW_n$.  We'll endow the sheaf $\EuScript{O}_{\basicW_n}^\hdef:=\EuScript{O}_{\basicW_n}((\hdef^{\nicefrac{1}{2}}))$ with the Moyal product multiplication \[f\star g :=fg+\sum_{n=1}^{\infty}\frac{\hdef^n}{2^nn!}\big(\frac{\partial^n f}{dx^n}\frac{\partial^n g}{dy^n}-\frac{\partial^n g}{dx^n}\frac{\partial^n f}{dy^n}\big).\]  Note that if $f$ or $g$ is a polynomial this formula only has finitely many terms, but for a more general meromorphic function, we will have infinitely many.  Following the conventions of \cite{BLPWquant}, we let $\EuScript{O}_{\basicW_n}^\hdef(0)=\EuScript{O}_{\basicW_n}[[\hdef^{\nicefrac{1}{2}}]]$, which is clearly a subalgebra.  We'll clarify later why we have adjoined a square root of $\hdef$.
				
				Sending $x\mapsto 1/y,y\mapsto xy^2$ induces an algebra automorphism of this sheaf on the subset $\basicW_n\setminus \{xy=0\}$, since 
				\[\frac 1y\star xy^2= xy+\frac{\hdef}2\qquad xy^2\star \frac1y=xy-\frac{\hdef}{2}. \] 
				This shows that we have an induced star product on the sheaf $\EuScript{O}^\hdef_{\basicW}$, and thus on $\EuScript{O}^\hdef_{\ZDol}$.  We now use non-commutative Hamiltonian reduction to define a star product on $\EuScript{O}^\hdef_{\Dol}$. This depends on a choice of non-commutative moment map $\kappa_\hbar : \ft \to \EuScript{O}^\hdef_{\ZDol^n}$. We fix $\phi\in \fd^*$. Given $(a_1, \ldots, a_n) \in \fd$, define 
				$$\kappa_\hbar(a_1,\dots, a_n) := \sum a_ix_iy_i+\hdef\phi(\Ba).$$ Our quantum moment map is the restriction of $\kappa_\hbar$ to $\ft \subset \fd$. Note that this agrees mod $\hdef$ with the pullback of functions from $\ft^*$ under $\Phi$. 
				
				Let $\EuScript{C}_\phi=\EuScript{O}^\hdef_{\ZDol^n}/\EuScript{O}^\hdef_{\ZDol^n}\cdot \kappa_{\hdef}(\ft)$ be the quotient of $\EuScript{O}^\hdef_{\ZDol^n}$ by the left ideal generated by these functions.  Note that this is supported on the subset $\Phi^{-1}(T^\vee_{\R}\times \{0\})$. We have an endomorphism sheaf $\sEnd(\EuScript{C}_\phi)$ of this sheaf of modules over $\EuScript{O}^\hdef_{\ZDol^n}$.  
				
				\begin{definition} Let $\EuScript{O}^\hdef_\phi$ be the sheaf of algebras on $\Dol$ defined by restricting $\sEnd(\EuScript{C}_\phi)$ to $\Phi^{-1}(\zeta')$ and pushing the result forward to $\Dol$.
				\end{definition}

				One can easily check, as in \cite{KR07}, that $\EuScript{O}^\hdef_\phi$ defines a deformation quantization of $\Dol$, that is, this sheaf is free and complete over $\C[[\hdef]]$, we have an isomorphism of algebra sheaves $\EuScript{O}^\hdef_\phi(0)/\hdef \EuScript{O}^\hdef_\phi(0)\cong \structuresheafcal_{\Dol}$, and given two meromorphic sections $f,g$, we have
				\[ f\star g-g\star f \equiv \hdef\{f,g\} \pmod {\hdef}.\]
				
				\subsection{\texorpdfstring{$G$}{G}-equivariant modules}
				By a $\EuScript{O}^\hdef_\phi$-module, we will always mean a sheaf $\sM$ of $\EuScript{O}^\hdef_\phi$-modules which admits a good lattice $\sM(0) \subset \sM$. 
				
				By construction, the map $\kappa_\hbar : \fd \to \EuScript{O}^\hdef_{\ZDol^n}$ descends to map $\mathfrak{g} \to \EuScript{O}^\hdef_\phi$, which quantizes the moment map for the action of $G$ on $\Dol$.
				
				\begin{definition} \label{def:preGeq}
					We call a $\EuScript{O}_{\phi}^\hdef$-module {\bf pre-weakly $G$-equivariant} if the action of $\mathfrak{g}$ via left multiplication by $\hbar^{-1}\kappa_{\hbar}$ on the sections on any $G$-invariant open set is locally finite, i.e. it is spanned by its generalized weight spaces for this torus.  
				\end{definition}
				A pre-weak equivariant structure can be upgraded to a weak equivariant structure as follows: we can assume that $\sM$ is indecomposable, so all weights appearing are in a single coset of the character lattice of $G$.  We can take the semi-simple part of the action of each element of $\mathfrak{g}$, and globally shift by a character of the Lie algebra to make all weights appearing integral.  The resulting action integrates to a weak $G$-equivariant structure (but we do not want to fix a specific one); we call such an action compatible with the $\EuScript{O}_{\phi}^\hdef$-module structure.  Note that pre-weakly $G$-equivariant modules are a Serre subcategory. 
				\begin{lemma} \label{lem:supportofGeq}
					Any pre-weakly $G$-equivariant module $\sM$ is supported on $q_{\Dol}^{-1}(0)$.
				\end{lemma}
				\begin{proof}
					Given any non-zero $X\in \mathfrak{g}$, consider the action of $\hbar^{-1}k:=\hbar^{-1}\kappa_{\hbar}(X)$ on $\sM(U)$ for $U$ a $G$-equivariant open subset.  By the assumption of local finiteness, for each $m\in \sM(U)$, there is a monic  polynomial $p(u)=u^d+p_{d-1}u^{d-1}+\cdots + p_0\in \C[u]$ such that $p(\hbar^{-1}k)m=0$.  
					
					If $U\cap q_{\Dol}^{-1}(0)=\emptyset$, then $k$ is invertible in $\EuScript{O}_{\phi}^\hdef(0)$, and so we have 
					\[m=\hbar(-p_{d-1} k^{-1}-\cdots - p_0p_{d-1}\hbar^{d-1} k^{-d})m.\] Thus, for any choice of good lattice $\sM(0)\subset \sM$, we have $\sM(0) (U)\subset \hbar \sM(0)(U)$.  Nakayama's lemma then implies that $\sM(0)(U)=0$, so $\sM(U)=0$.
				\end{proof}
				
				Unfortunately, the action of $\mathbb{S}G$ on $\algcov$ does not preserve $\Dolcov$. We can nevertheless speak of $\mathbb{S}G$-equivariance on $\Dolcov$ and $\Dol$, as follows.
				
				Let $\sM$ be a pre-weakly $G$-equivariant $\EuScript{O}^\hdef_\phi$-module. Let $\nu^* \sM$ be the pullback of this module to $\Dolcov$.  We write $(\nu^*\sM)^{\operatorname{alg}}$ for the pushforward of $\nu^*\sM$ along the inclusion $\Dolcov \to \algcov$. Note that by Lemma  \ref{lem:supportofGeq}, $\nu^* \sM$ is supported on $\widetilde{\mathfrak{C}} \subset \Dolcov$, and this subset remains closed in $\algcov$. Thus the support is not enlarged.
				
				Fix $\Bx_0\in \gradedchambers$. 
				
				\begin{definition}
					A {\bf pre-weakly $\mathbb{S}G$-equivariant} structure on a pre-weakly $G$-equivariant $\EuScript{O}^\hdef_\phi$-module $\sM$ is an  action of the Lie algebra $\lieS_{\Bx_0})$ commuting with $\mathfrak{g}$ which integrates to an equivariant structure for $\mathbb{S}_{\Bx_0}$ on $(\nu^*\sM)^{\operatorname{alg}}$.  
					
					We write  $\SGmod$ for the category of such modules. Since a homomorphism between pre-weakly $\mathbb{S}G$-equivariant modules is $\lieS_{\Bx_0})$-equivariant, multiplication by $\hbar$ is not a morphism in this category, so this category is $\C$-linear, not $\C((\hbar))$-linear.  
				\end{definition}

				As with pre-weakly $G$-equivariant modules, after making some auxiliary choices, we can endow a pre-weakly $\mathbb{S}G$-equivariant-module with a `compatible' action of the torus $\mathbb{S}G$, which integrates the semisimple part of (a shift of) the infinitesimal action. 
				\begin{lemma} \label{lem:independenceofBx}
					Let $\sM \in \SGmod$. Fix a compatible action of $\mathbb{S}G$. The action of $\C^*$ on $\nu^*\sM^{\operatorname{alg}}$ induced by the composition $\C^* \cong \mathbb{S}_\By \to \mathbb{S}G$ does not depend on the $G$-equivariant structure, up to isomorphism.
				\end{lemma}
				\begin{proof}
					Again, we can reduce to the case where $\sM$ is indecomposable.  By construction, any two compatible $G$-equivariant structures on $\sM$ differ by tensor product with a character of the group $G$, so the induced $\mathbb{S}_{\By}$ structures differ by tensor product with a character of $\mathbb{S}_{\By}$, which we can think of as the integer weight $w$.  Since $\hbar$ has weight 1 under $\mathbb{S}_{\By}$, multiplication by $\hbar^w$ intertwines these two actions, and gives an isomorphism between the two $\mathbb{S}_{\By}$-equivariant structures.  
				\end{proof}

				\subsection{The deformation quantization near a component of \texorpdfstring{$\mathfrak{C}$}{C}}
				
				Given $\phi\in \ft_\Q^*$, we can define a fractional line bundle $\lin_\phi$ on any quotient by a free $T$-action. The component $\mathfrak{X}_\Bx$ was defined by a free $T_\R$-action; by standard toric geometry, it also carries a canonical presentation as a free $T$-quotient. Applying this construction to $\mathfrak{X}_\Bx$ thus yields a bundle $\lin_{\phi, \Bx}$. If $\phi\in \ft^*_\Z$, the set of honest characters, then this is an honest line bundle; otherwise, it gives a line bundle over a gerbe, but we can still define an associated Picard groupoid, and thus a sheaf of twisted differential operators (TDO) on $\mathfrak{X}_\Bx$.  Let $\Omega_\Bx$ be the canonical line bundle on $\mathfrak{X}_\Bx$, and $\Omega_\Bx^{1/2}$ the half-density fractional line bundle. It is a classical fact that $\Omega_\Bx = l_{-\phi_0, \Bx}$ where $\phi_0$ is the sum of all $T$-characters of $\C^n$ induced by the map $T \to D$.
				
				We let $D_{\phi, \Bx}$ denote the TDO associated to the fractional
				line bundle $\lin_{\phi, \Bx}\otimes \Omega_\Bx^{1/2}$, and let
				$\micro_{\phi, \Bx}$ be its microlocalization on $T^*\mathfrak{X}_\Bx$.
				That is, $\micro_{\phi, \Bx}$ is a sheaf in the classical topology on 
				$T^*\mathfrak{X}_\Bx$ whose sections on $T^*U$ for $U\subset \mathfrak{X}_\Bx$
				is the Rees algebra for the order filtration on $D_{\phi, \Bx}(U)$;
				for an open subset $V\subset T^*U$ (where we can assume WLOG that $U$
				is affine), we further invert any element of the Rees algebra
				whose image under the map $\micro_{\phi, \Bx}(U)/\hbar \micro_{\phi,
					\Bx}(U)\cong \EuScript{O}_{T^*U}(T^*U)$ is invertible on $V$. The construction of this algebra is discussed in more detail in \cite[\S 4.1]{BLPWquant}.  We'll be more interested in its localisation:
				\begin{definition}
					$$\microh_{\phi,\Bx}:=\micro_{\phi, \Bx}[\hbar^{-\nicefrac{1}{2}}].$$
				\end{definition}
				
				If we equip a module
				$\mathcal{M}$ over
				the TDO $D_{\phi, \Bx}$ with a good filtration, which for technical reasons we'll index with $\frac{1}{2}\Z$, its Rees
				module $\mathcal{M}(0)$ generated by $\hbar^{-k}\mathcal{M}_{\leq k}$ for $k\in \frac{1}{2}\Z$ is a coherent module over the Rees algebra (we can use this as a definition of good filtration).  That is, it is a coherent sheaf of
				$\micro_{\phi, \Bx}$-modules, equipped with a $\C^*$-equivariant structure for
				the squared scaling $\C^*$-action (or equivalently, a grading of its sections
				on $T^*U$).  Inverting $\hbar$, we obtain a $\microh_{\phi,\Bx}$-module
				$\sM=\mathcal{M}(0)[\hbar^{-\nicefrac{1}{2}}]$ which is independent of the choice of good
				filtration, which is {\bf good} in the sense of \cite[\S
				4]{BLPWquant}, that is, it admits a coherent,  $\C^*$-equivariant $\micro_{\phi, \Bx}$-lattice.
				By \cite[Prop. 4.5]{BLPWquant}, this is an
				equivalence between coherent $D_{\phi, \Bx}$-modules and good
				$\microh_{\phi, \Bx}$-modules.  
				
				\begin{theorem}\label{thm:micro-iso}
					We have an isomorphism of algebra sheaves $\iota_\Bx^*\EuScript{O}^\hdef_\phi\cong \microh_{\phi, \Bx}|_{\Dolcov_\Bx}$.
				\end{theorem}
				\begin{proof}
					First, we check that this holds in the base case, i.e. when $\Dol = \ZDol$. It is convenient to check this on the universal cover $\ZDolcov$. By the $\Z$-symmetry of the latter, it is enough to check for a single component of the core. Hence, consider the copy of $\mathbb{CP}^1$ in the union of $\basicW_0\cup \basicW_1$.  Using superscripts to indicate which $\basicW_*$ we work on, we have birational coordinates $y^{(0)}=1/x^{(1)}$ and $x^{(0)}=y^{(1)}(x^{(1)})^2$.  We thus have an isomorphism of $\basicW_0\cup \basicW_1$ to $T^*\mathbb{CP}^1$ with coordinate $z$ and dual coordinate $\xi$ sending 
					\[x^{(1)}\mapsto z \qquad y^{(1)} \mapsto \xi \qquad x^{(0)}\mapsto z^2 \xi \qquad y^{(0)} \mapsto \frac{1}{z}\]
					We can quantize this to a map from $\EuScript{O}^\hdef_{\basicW}$ to $\micro_{\phi}$ by the corresponding formulas:
					\[x^{(1)}\mapsto z \qquad y^{(1)} \mapsto \hbar\frac{d}{dz}  \qquad x^{(0)}\mapsto \hbar  z^2\frac{d}{dz} \qquad y^{(0)} \mapsto \frac{1}{z}.\]
					
					This induces an isomorphism of sheaves, which in turn restricts to an isomorphism $\iota_\Bx^* \EuScript{O}^\hdef_{\mathfrak{Z}} \to \microh_{\phi, \Bx}|_{\ZDol_0}$.  Note that under this isomorphism, $q=x^{(1)}y^{(1)}\mapsto  \hbar z\frac{d}{dz}-\hbar/2$.

					To proceed to the general case, we consider $\ZDolcov^n$ and its quantized $T$-moment map $\kappa_{\hbar}$. Fix as above an open subset of isomorphic to $(T^*\mathbb{P}^1)^n$. Applying the above morphism to the image of $\kappa_{\hbar}$, we obtain the following.
					\[\kappa_{\hbar}(a_1,\dots, a_n)\mapsto \sum_i a_iz_i\frac{d}{dz_i}-\frac{\hbar}{2}+\hbar \phi(\Ba)= \sum_i \frac{a_i}{2}(z_i\frac{d}{dz_i}+\frac{d}{dz_i}z_i)+\hbar \phi(\Ba).\]
					The result then follows from the compatibility of twisted microlocal differential operators with symplectic reduction as in \cite[Prop. 3.16]{BLPWquant}.  We can identify the twist of a TDO from its period by \cite[Prop. 4.4]{BLPWquant}.
				\end{proof}
				
				Thus, given a $\EuScript{\tilde{O}}^\hdef_\phi$-module $\ssM$, we can pull it
				back to an $\microh_{\phi,\Bx}|_{\Dolcov_\Bx}$-module $\ssM|_{\Dolcov_{\Bx}}$ on
				$\Dolcov_\Bx \subset T^*\mathfrak{X}_{\Bx}$. 
				
				If we additionally choose a
				$\mathbb{S}_{\Bx}$-equivariant structure which makes $\ssM|_{\Dolcov_{\Bx}}$
				into a good module, then the equivalence of \cite[Prop. 4.5]{BLPWquant} will
				give a corresponding module over the TDO $D_{\phi, \Bx}$, with a
				choice of good filtration.
				\begin{definition}
					Given $\sM \in \SGmod$, let $\sigma_{\Bx}(\ssM) \in D_{\phi, \Bx} \mmod$ be the module defined as above for some choice of $\mathbb{S}_\Bx$-equivariant structure.
				\end{definition}
				The resulting $D$-module does not depend on the choice of compatible $\mathbb{S}G$-equivariant structure, by Lemma \ref{lem:independenceofBx}.  It does carry a good filtration which depends on this choice, but only up to a shift on each indecomposable summand of $\sigma_\Bx(\sM)$.

				The modules $\sigma_\Bx(\ssM)$ for different $\Bx$ are compatible in the following sense: 
				As discussed previously, the intersection $\mathfrak{X}_\Bx\cap \algcov_{\By}$ is precisely the conormal bundle $N_{\Bx,\By}=N^*_{\mathfrak{X}_{\By}}(\mathfrak{X}_\Bx\cap \mathfrak{X}_{\By})$ to $\mathfrak{X}_\Bx\cap \mathfrak{X}_{\By}$ in $T^*\mathfrak{X}_{\By}$.  Thus   the intersection $ \algcov_{\Bx,\By}=\algcov_\Bx\cap \algcov_\By$ can be identified with $T^*(N_{\Bx,\By})$ or swapping the roles of $\Bx,\By$ with $T^*(N_{\By,\Bx})$.
				
				Since the vector bundles $N^*_{\mathfrak{X}_{\By}}(\mathfrak{X}_\Bx\cap \mathfrak{X}_{\By})$ and $N^*_{\mathfrak{X}_{\Bx}}(\mathfrak{X}_\Bx\cap \mathfrak{X}_{\By})$ are dual, so Fourier transform $\mathcal{F}_{\By,\Bx}$ gives an equivalence between the categories of pre-weakly $G$-equivariant D-modules on these spaces, and between constructible sheaves with $\mathbb{R}$-coefficients, which are compatible with respect to the solution functor.  By construction, we thus have 
				\begin{equation}\label{eq:Fourier1}
					\mathcal{F}_{\By,\Bx}\sigma_\Bx(\ssM)|_{N_{\Bx,\By}}\cong \sigma_\By(\ssM)|_{N_{\By,\Bx}}.
				\end{equation} 
				\subsection{Preliminaries on the Ext-algebra of the simples}
				Assume that $\phi$ is chosen so that $\lin_\phi\otimes \Omega^{1/2}$ is an honest line bundle for all $\Bx$.  From now on, we use the abbreviations $\microh_\Bx := \microh_{\phi, \Bx}$, $D_\Bx := D_{\phi, \Bx}$ and $\lin_\Bx := \lin_{\phi, \Bx}$, since the dependence on $\phi$ will not play any further role in this paper.
				
				\begin{remark} Recall that $\Omega_\Bx^{1/2}$ equals $\lin_{-\phi_0/2, \Bx}$ where $\phi_0$ is the sum of all $T$-characters induced by the embedding $T \to D$. Thus our assumption will be satisfied whenever $\phi(\Ba)\in \Z+\frac{1}{2} \sum a_i$ for all $\Ba\in \ft_\Z$. For example, we can let $\phi$ be the restriction of the element $(\frac{1}{2},\cdots, \frac{1}{2})\in \mathfrak{d}^*$.   Nothing we do will depend on this choice; in fact, the categories of $\EuScript{O}^\hdef_{\phi}$-modules for $\phi$ in a fixed coset of $\ft_\Z^*$ are all equivalent via tensor product with quantizations of line bundles on $\Dol$ (as in \cite[\S 5.1]{BLPWquant}), so our calculations will be independent of this choice.
				\end{remark}
				
				In this case, the sheaf $\microh_\Bx$  naturally acts on  $\sL_\Bx' := \lin_{\Bx}\otimes \Omega_\Bx^{1/2}((\hdef))$ as a sheaf on $\mathfrak{X}_\Bx$  pushed forward into $\Dolcov_\Bx$; under the equivalence of \cite[Prop. 4.5]{BLPWquant} mentioned above, this corresponds to the twisted D-module $\lin_{\Bx}\otimes \Omega_\Bx^{1/2}$.    
				Of course, this sheaf is equivariant for the  action of $\mathbb{S}_{\Bx}$, and pre-weakly $G$-equivariant.

				Via the maps \[\iota_\Bx : \Dolcov_\Bx \to \Dol\qquad \tilde{\iota}_\Bx : \Dolcov_\Bx \to \Dolcov\] we can define modules over $\EuScript{O}^\hdef_\phi$ and $\EuScript{\tilde{O}}^\hdef_\phi$:
				\begin{definition}
					Let $\sL_\Bx=\iota_*\sL_\Bx'$ and $\tsL_\Bx=\tilde{\iota}_*\sL_\Bx'$.
				\end{definition} 
				Using $\mathbb{S}_{\Bx}$-equivariance, and the pre-weak $G$-equivariant of this module, we obtain a twisted D-module $\sigma_{\By}\tsL_\Bx$.
				Recall that we have a universal cover map $\nu : \Dolcov \to \Dol$.
				\begin{proposition}
					We have isomorphisms $\nu_*\tsL_\Bx\cong \sL_\Bx$, and $\nu^*\sL_\Bx\cong \oplus_{\Bz\in \mathfrak{g}_\Z^*} \tsL_{\Bx+\Bz}$.
				\end{proposition}
				
				The first category we will consider on the A-side of our correspondence is $\Lcat$, the dg-subcategory of $\SGmod$ generated by $\sL_\Bx$ for all $\Bx$.  As observed before, since weakly $G$-equivariant modules form a Serre subcategory, any finite length object in this category is pre-weakly $G$-equivariant, and so we can define the D-modules $\sigma_{\By}(\ssM)$ for any module $M$ in this category.
				
				This has a natural $t$-structure, whose heart is an abelian category $\lcat$.   We similarly let $\widetilde{\Lcat}$ be the dg-subcategory generated by $\tsL_\Bx$, and $\widetilde{\lcat}$ the heart of the natural $t$-structure. This definition might seem slightly ad hoc, but we will later see that it is motivated by our notion of microlocal mixed Hodge modules.
				
				Since the Ext sheaf between $\tsL_\Bx$ and $\tsL_{\By}$ is supported on the intersection between $\mathfrak{X}_{\Bx}$ and $\mathfrak{X}_{\By}$, we have
				\[\Ext_{\widetilde{\Lcat}}(\tsL_\By,\tsL_\Bx)\cong \Ext_{\microh_\Bx|_{\Dolcov_\Bx}\mmod ^{\mathbb{S}}}(\tilde{\iota_\Bx}^* \tsL_\By, \sL_\Bx')\]
				
				In fact, if we replace $\tsL_\By$ by an injective resolution, we see that this induces a homotopy equivalence between the corresponding Ext complexes.  Since $\cL'_\Bx$ is supported on the zero section, Ext to it is unchanged by passing to an open subset containing this support and 
				\begin{equation}\label{eq:micro-ext}
					\Ext_{\widetilde{\Lcat}}(\tsL_\By,\tsL_\Bx)\cong \Ext_{D(\mathfrak{X}_{\Bx})}(\sigma_{\Bx}\tsL_\By,\lin_\Bx\otimes \Omega_\Bx^{1/2})
				\end{equation}
				where the latter Ext is computed in the category of D-modules on the toric variety $\mathfrak{X}_{\Bx}$.  
				In the toric variety $\mathfrak{X}_\Bx$, the preimage of the intersection with the image of $\mathfrak{X}_\By$ is a toric subvariety corresponding to the intersection of the corresponding chambers in $\textgoth{B}^{\operatorname{per}}_\zeta$.  
				\begin{lemma} \label{lem:functionmodule}
					The microlocalization $\sigma_{\Bx}\tsL_\By$ is the line bundle $\lin_\By\otimes \Omega_\By^{1/2}$ pulled back to $\mathfrak{X}_\Bx\cap \mathfrak{X}_\By$ and pushed forward to $\mathfrak{X}_\Bx$ as a $D_\Bx$-module.
				\end{lemma}
				\begin{proof}
					Consider the intersection of $\mathfrak{X}_{\By}$ with $\algcov_{\Bx}$.  This is a closed $\mathbb{S}_{\Bx}$-invariant Lagrangian closed subset, so it is the conormal to its intersection with the zero section $\mathfrak{X}_{\Bx}\cap \mathfrak{X}_{\By}$.  The D-module $\sigma_{\Bx}\tsL_\By$ has singular support on this subvariety, and thus must be a local system on $\mathfrak{X}_{\Bx}\cap \mathfrak{X}_{\By}$, which is necessarily $\lin_\By\otimes \Omega_\By^{1/2}$.
				\end{proof}
				Since $\mathfrak{X}_{\Bx}\cap \mathfrak{X}_{\By}$ is a smooth toric subvariety, the sheaf $\Ext$ between $\sigma_{\Bx}\tsL_\By$ and $\lin_\Bx\otimes \Omega_\Bx^{1/2}$ is $\mathbb{C}_{\mathfrak{X}_{\Bx}\cap \mathfrak{X}_{\By}}[-k]$ 
				where $k$ is the codimension of $\mathfrak{X}_{\Bx}\cap \mathfrak{X}_{\By}$ in $\mathfrak{X}_{\Bx}$. 
				This shows that we have an isomorphism 
				\begin{equation}
					\Ext_{\widetilde{\Lcat}}^m(\tsL_\By,\tsL_\Bx)\cong 
					\HH^{m-k}(\mathfrak{X}_\By\cap \mathfrak{X}_\Bx;\C).\label{eq:ext-coh}
				\end{equation}

				We will be interested in the class $\DDD_{\By,\Bx}$ in the left-hand space corresponding to the identity in $ \HH^*\big(\mathfrak{X}_\By\cap \mathfrak{X}_\Bx;\C\big)$.  Unfortunately, this is only well-defined up to scalar. We will only need the case where $|\Bx-\By|_1=1$. In this case, we can define $\DDD_{\By,\Bx}$ (without scalar ambiguity) as follows.
				
				Consider the inclusions $\mathfrak{X}_{\Bx}\setminus (\mathfrak{X}_\By\cap \mathfrak{X}_\Bx)\overset{j}{\hookrightarrow}\mathfrak{X}_{\Bx}\overset{i}{\hookleftarrow}\mathfrak{X}_\By\cap \mathfrak{X}_\Bx$ and the corresponding sequence of D-modules 
				\begin{equation} \label{eq:ext1}
					0\longrightarrow \structuresheaf_{\mathfrak{X}_{\Bx}} \longrightarrow 
					j_*\structuresheaf_{\mathfrak{X}_{\Bx}\setminus (\mathfrak{X}_\By\cap \mathfrak{X}_\Bx)}\longrightarrow 
					j_*\structuresheaf_{\mathfrak{X}_{\Bx}\setminus (\mathfrak{X}_\By\cap \mathfrak{X}_\Bx)} / \structuresheaf_{\mathfrak{X}_{\Bx}} 
					\longrightarrow 0.
				\end{equation}
				Any identification of the right-hand D-module with $i_!\structuresheaf_{\mathfrak{X}_\By\cap \mathfrak{X}_\Bx}$ defines a class $\DDD_{\By,\Bx} \in \Ext^1(\structuresheaf_{\mathfrak{X}_{\Bx}}, i_!\structuresheaf_{\mathfrak{X}_\By\cap \mathfrak{X}_\Bx}) = \Ext^1_{\widetilde{\Lcat}}(\tsL_\By,\tsL_\Bx) $.

				Such an identification is obtained by picking the germ of a function $g$ on $\mathfrak{X}_{\Bx}$ in the formal neighborhood of $\mathfrak{X}_\By\cap \mathfrak{X}_\Bx$ that vanishes on this divisor with order 1. Given such a function, the map $f\mapsto \tilde{f}/g$ where $\tilde{f}$ is an extension of a meromorphic function on $\mathfrak{X}_\By\cap \mathfrak{X}_\Bx$ to the formal neighborhood defines an isomorphism of D-modules $ i_!\structuresheaf_{\mathfrak{X}_\By\cap \mathfrak{X}_\Bx} \to j_*\structuresheaf_{\mathfrak{X}_{\Bx}\setminus (\mathfrak{X}_\By\cap \mathfrak{X}_\Bx)} / \structuresheaf_{\mathfrak{X}_{\Bx}} $.  
				
				We can arrange our choice of chart in $\ZDolcov^n$ so that $\mathfrak{X}_\By\cap \mathfrak{X}_\Bx$ is defined by the vanishing of one of the coordinate functions; note that in this case, $\mathfrak{X}_\By\cap \mathfrak{X}_\Bx$ is defined inside $\mathfrak{X}_{\By}$ by the vanishing of the symplectically dual coordinate function (eg, if the first is defined by the vanishing of $x_i$, then the latter will be defined by $y_i$).  We choose this as the function to define $\DDD_{\Bx,\By}$. 
				\begin{definition}
					For any $\Bx, \By$ such that $|\Bx-\By|_1=1$, let $\DDD_{\By,\Bx} \in \Ext_{\widetilde{\Lcat}}^1(\tsL_\By,\tsL_\Bx)$ be the class defined by the above prescription.
				\end{definition}

				\subsection{Mirror symmetry}
				We are almost ready to compare the first and second halves of this paper. First, we need to match the parameters entering into our constructions. Recall that $\Dol$ depends on a choice of generic stability parameter $\zeta \in T^{\duall}_{\R}$, 
				Likewise, the hypertoric enveloping algebra in characteristic $p$ depends on a central character $\la \in \mathfrak{t}^*_{\mathbb{F}_p}$. The algebra $\gradedHalg^{\la, !}$ which describes the $\Ext$ groups of its simple modules (Definition \ref{defHdual}) thereby also depends on $\la$.  In order to match $\zeta$ and $\la$, we identify $\mathfrak{t}^*_{\mathbb{F}_p}$ with $\mathfrak{t}^*_\Z / p \mathfrak{t}^*_\Z$ and thereby embed it in $T^{\duall}_{\R} = \mathfrak{t}^*_\R / \mathfrak{t}^*_\Z$ via $\la \mapsto \frac{1}{p} \la$. From now on we suppose that $\lambda$ is smooth, and that $\zeta$ is its image in $T^{\duall}_\R$. 
				
				It follows that $\textgoth{B}^{\operatorname{per}}_\zeta$ is the real form of $\textgoth{A}^{\operatorname{per}}_\la$, and their sets of chambers are naturally in bijection. Since the chambers of $\textgoth{A}^{\operatorname{per}}_\la$ index the simples of $A^\la_{\K}\mmod_o$ and the chambers of $\textgoth{B}^{\operatorname{per}}_\zeta$ index the simples of $\Lcat$, we obtain a bijection of simple objects. 
				
				We now show that the $\Ext$-algebras of these simples share an integral form. 
				
				\begin{theorem} \label{mainthm}
					We have isomorphisms of algebras \[\gradedHalg_{\zeta, \C}^!\cong \bigoplus_{\Bx, \By \in \gradedchambers(\zeta)}\Ext_{\widetilde{\Lcat}}\Big(\tsL_\Bx, \tsL_\By \Big)\qquad \degradedHalg_{\zeta, \C}^!\cong \bigoplus_{\Bx, \By \in \degradedchambers(\zeta)}\Ext_{{\Lcat}}\Big(\sL_\Bx, \sL_\By \Big)\] sending $d_{\Bx,\By}\mapsto \DDD_{\Bx,\By}$ when $\By\in \al_i(\Bx)$.
				\end{theorem}
				\begin{proof}
					
					We need to check that the rule $d_{\Bx,\By}\mapsto \DDD_{\Bx,\By}$ defines a homomorphism, i.e. that the relations (\ref{wall-crossbang}--\ref{doublestep}) hold in $\bigoplus_{\Bx, \By \in \gradedchambers(\zeta)} \Ext_{\widetilde{\Lcat}}\Big(\tsL_\Bx, \tsL_\By \Big)$.
					\begin{enumerate}
						\item The relation (\ref{wall-crossbang}) follows from the fact
						that when $|\Bx-\By|_1=1$, the element $\DDD_{\Bx,\By}\DDD_{\By,\Bx}$ is the class in
						$\HH^2\big(\mathfrak{X}_\Bx;\Q\big)$ dual to the divisor $\mathfrak{X}_{\Bx}\cap
						\mathfrak{X}_{\By}$, while the class $\mathsf{t}_i$ is defined by the Chern class of the corresponding line bundle, for which a natural section vanishes with order one on $\mathfrak{X}_{\Bx}\cap
						\mathfrak{X}_{\By}$ for $\By\in \alpha(i)$ and nowhere else.
						\item Note that the relations (\ref{codim1bang}) and (\ref{codim2bang}) equate two elements of the 1-dimensional space $\Ext^2(\sL_\Bx, \sL_\Bw) \cong \HH^0(\mathfrak{X}_x \cap \mathfrak{X}_w ; \C)$. Thus, we only need check that we have the scalars right, and this can be done after restricting to any small neighborhood where all the classes under consideration have non-zero image.
						
						Thus ultimately we can reduce to assuming $\mathfrak{X}_{\Bx}=\C^2$, and $\mathfrak{X}_{\By}$ and $\mathfrak{X}_{\Bw}$ are the conormals to the coordinate lines, and $\mathfrak{X}_{\Bz}$ the cotangent fiber over $0$.
						Let $r_1,r_2$ be the usual coordinates on $\C^2$, and $\partial_1,\partial_2$ be the directional derivatives for these coordinates.  Thus, we are interested in comparing the $\Ext^2$'s given by the sequences in the first and third row of the diagram below.  
						Both sequences are quotients of the free Koszul resolution in the second row: 
						\[\tikz[->,thick]{
							\matrix[row sep=14mm, column sep=30mm,ampersand replacement=\&]{
								\node[scale=1.1] (a1) {$\frac{D}{D\cdot (r_1,r_2)}$}; \& \node[scale=1.1] (b1) {$\frac{D_ {\partial_2}}{D_{ {\partial_2}}\cdot (r_1,r_2)}$}; \& \node[scale=1.1] (c1) {$\frac{D_{{\partial_1}}}{D_{ {\partial_1}}\cdot (r_1,\partial_2)}$}; \& \node[scale=1.1] (d1) {$\frac{D}{D\cdot (\partial_1,\partial_2)}$};\\
								\node[scale=1.1] (a2) {$D$}; \& \node[scale=1.1] (b2) {$D\oplus D$}; \& \node[scale=1.1] (c2) {$D$}; \& \node[scale=1.1] (d2) {$\frac{D}{D\cdot (\partial_1,\partial_2)}$};\\
								\node[scale=1.1] (a3) {$\frac{D}{D\cdot (r_1,r_2)}$}; \& \node[scale=1.1] (b3) {$\frac{D_{\partial_1}}{D_{ {\partial_1}}\cdot (r_1,r_2)}$}; \& \node[scale=1.1] (c3) {$\frac{D_{{\partial_2}}}{D_{ {\partial_2}}\cdot (r_2,\partial_1)}$}; \& \node[scale=1.1] (d3) {$\frac{D}{D\cdot (\partial_1,\partial_2)}$};\\
							};
							\draw (a1) -- node[scale=.75,midway, above]{$1\mapsto 1$}(b1) ; 
							\draw (b1) -- node[scale=.75,midway, above]{$\frac{1}{\partial_2}\mapsto 1$}(c1) ;
							\draw (c1) -- node[scale=.75,midway, above]{$\frac{1}{\partial_1}\mapsto 1$}(d1) ;
							\draw (a2) -- node[scale=.75,midway, above]{$1\mapsto \begin{bmatrix} \partial_2 \\ -\partial_1 \end{bmatrix}$}(b2) ; 
							\draw (b2) -- node[scale=.75,midway, above]{$\begin{bmatrix} a\\ b \end{bmatrix}\mapsto a\partial_1+b\partial_2$}(c2) ;
							\draw (c2) -- node[scale=.75,midway, above]{$1\mapsto 1$}(d2) ;
							\draw (a3) -- node[scale=.75,midway, above]{$1\mapsto 1$}(b3) ; 
							\draw (b3) -- node[scale=.75,midway, above]{$\frac{1}{\partial_1}\mapsto 1$}(c3) ;
							\draw (c3) -- node[scale=.75,midway, above]{$\frac{1}{\partial_2}\mapsto 1$}(d3) ;
							\draw (a2) -- node[scale=.75,midway, right]{$1\mapsto 1$}(a1) ; 
							\draw (a2) -- node[scale=.75,midway, right]{$1\mapsto -1$}(a3) ;
							\draw (b2) -- node[scale=.75,right,midway]{$\begin{bmatrix} a\\ b \end{bmatrix}\mapsto\frac{a}{\partial_2}$}(b1) ; 
							\draw (b2) -- node[scale=.75,right,midway]{$\begin{bmatrix} a\\ b \end{bmatrix}\mapsto\frac{b}{\partial_1}$}(b3) ;
							\draw (c2) -- node[scale=.75,right,midway]{$1\mapsto\frac{1}{\partial_1}$}(c1) ; 
							\draw (c2) -- node[scale=.75,right,midway]{$1\mapsto\frac{1}{\partial_2}$}(c3) ;
							\draw (d2) -- node[scale=.75,right,midway]{$\cong$}(d1) ; 
							\draw (d2) -- node[scale=.75,right,midway]{$\cong$}(d3) ;
						}\]
						The opposite signs in the leftmost column confirm that we have $\DDD_{\Bz,\Bw}\DDD_{\Bw,\Bx}=-\DDD_{\Bz,\By}\DDD_{\By,\Bx}$. Hence the elements $\DDD_{\Bx,\By}$ satisfy the relations  (\ref{codim1bang}--\ref{codim2bang}).
						\item The relations (\ref{doublestep}) follows from the fact that in this case $\mathfrak{X}_{\Bx}\cap \mathfrak{X}_{\Bz}=\emptyset$.
					\end{enumerate}
					Recall that the complex dimension of $e_\Bx \gradedHalg^!_{\zeta, \C} e_\By$ coincides with that of $\HH^*(\mathfrak{X}_{\Bx}\cap \mathfrak{X}_{\By};\C)$, as we discussed in Section \ref{interpretation}. Thus the spaces $e_\Bx \gradedHalg^!_{\zeta, \C} e_\By$ and $\HH^*(\mathfrak{X}_{\Bx}\cap \mathfrak{X}_{\By};\C)$ are vector spaces of the same rank. Thus, in order to show that our map is an isomorphism, it is enough to show that it is surjective. 
					
					By Kirwan surjectivity, the fundamental class generates $\HH^*(\mathfrak{X}_{\Bx}\cap \mathfrak{X}_{\By}; \C)$ as a module over the Chern classes  of line bundles associated to representations of $T$. Since the fundamental classes are images of $\pm \DDD_{\Bx,\By}$ and the Chern classes are images of $\C[\mathsf{t}_1,\dots, \mathsf{t}_n]$, we have a surjective map.  As noted before, comparing dimensions shows that it is also injective, which concludes the proof.
				\end{proof}
				
				Comparing this result with Proposition \ref{prop:Hbang-coh}, we see that the categories $\Lcat$ and $D^b(\mathsf{Coh}(\fM))$ are rather similar.  We would immediately obtain a fully-faithful functor $\Lcat\to D^b(\mathsf{Coh}(\fM))$ if we knew that $ \bigoplus_{\Bx, \By \in \degradedchambers(\lambda)}\Ext_{{\Lcat}}\Big(\sL_\Bx, \sL_\By \Big)$ were formal as a dg-algebra, but it is not clear that this is the case. 
				
				To show this formality, we need to use a different approach to construct this functor, using projective objects in the category $\tdq$.  This approach also naturally leads to a structure on $\Lcat$ that corresponds to the $\mathbb{G}_m$-action on $\fM$ discussed earlier: a new structure on DQ-modules, closely related to Saito's theory of mixed Hodge modules.  This will result in a graded category, which is to $D^b(\mathsf{Coh}_{\mathbb{G}_m}(\fM))$ as $\Lcat$ is to $D^b(\mathsf{Coh}(\fM))$. 
				
				\subsection{Projectives}
				
				As described above, we'll construct projective covers in $\tdq$.  As usual, let us first construct these on $\mathfrak{\tilde{Z}}$.  
				
				Consider $A=\C[x,y,\hbar]$ with the usual Moyal star product defined above.
				There are unique dq-modules $\sP_*^{(k)},\sP_!^{(k)}$ over $\C^2$ whose sections are the quotients \[H^0(\C^2; \sP_*^{(k)} ) =A/A\star (y\star x)^{\star k}\qquad H^0(\C^2; \sP_!^{(k)}) = A/A\star (x\star y)^{\star k}.\]
				\nc{\dP}{\mathcal{P}}
				Identifying $A$ with the Rees algebra of differential operators $D_x$ on $\C[x]$ (sending $y\mapsto \hbar \frac{\partial}{\partial x}$), these modules become the Rees modules of D-modules $\dP_*^{(k)},\dP_!^{(k)}$ on $\mathbb{A}^1$ with coordinate $x$.  
				We can identify these with  the $*$- and $!$-pushforwards of the D-module $L^{(k)}$ on $\C^*=\Spec(\C[x,x^{-1}])$ defined by the connection $\nabla=d-\frac{N}{x}$ on the trivial bundle with fiber $\C^k$, where $N$ is the regular nilpotent matrix
				\[N=\begin{bmatrix}
					0& 1 & 0 & \cdots &0&0\\
					0& 0& 1 & \cdots &0& 0\\
					0 & 0 & 0 & \cdots &0&0\\
					\vdots & \vdots & \vdots & \ddots & \vdots & \vdots\\
					0 & 0 & 0 & \cdots &0& 1\\
					0 & 0 & 0 & \cdots &0& 0
				\end{bmatrix} \]
				Both $\dP_*^{(k)}$ and $\dP_!^{(k)}$ are projective in the category of D-modules on $\mathbb{A}^1$ which are smooth away from the origin, whose monodromy around the origin has nilpotent part of length $\leq k$.  The D-module $\dP_!^{(k)}$ is the projective cover of the D-module of polynomials on $\mathbb{A}^1$, and $\dP_*^{(k)}$ is the projective cover of the delta functions at the origin.  
				
				Our presentation of these D-modules induces a good filtration on them; in DQ-module terms, this is an equivariant structure for the cotangent scaling $\mathbb{S}$ which has weight $0$ on $x$ and weight $1$ on $y$.  In fact, we will want to use shifts of this filtration, corresponding to  $\hdef \sP_*^{(k)}$ and $\hdef^{\nicefrac{1}{2}}\sP_!^{(k)}$ (note that the latter is only equivariant under the squared scaling).  In D-module terms, this means that we endow $\dP_*^{(k)}$ with the good filtration such that the image of $\C[x]\subset D_x$ spans $F_{-1} \dP_*^{(k)}$ and $F_{p}\dP_*^{(k)}=F_{p+1}D_x\cdot F_{-1} \dP_*^{(k)}$, and $\dP_!^{(k)}$ with the good filtration such that the image of $\C[x]\subset D_x$ spans $F_{-\nicefrac{1}{2}} \dP_!^{(k)}$ and $F_{p-\nicefrac{1}{2}}\dP_!^{(k)}=F_{p}D_x\cdot F_{-\nicefrac{1}{2}} \dP_!^{(k)}$.  These might seem like slightly strange choices: they are deliberately chosen so that in both cases, the unique simple quotient carries a pure Hodge structure of weight 0.  
				
				We will need certain morphisms between these DQ-modules: 
				
				\begin{enumerate}
					\item The linear map $N$ on $\C^k$ induces an endomorphism on $\sL^{(k)}$ and hence of $\sP_*^{(k)}$ and $\sP_!^{(k)}$.  This is the same as right multiplication by $y\star x$ or $x\star y$, respectively.  
					\item We have a $c^-\colon \sP_*^{(k)}\to \sP_{!}^{(k)}$, induced by multiplication on the right by $y$. Note that this map becomes an isomorphism if we invert $y$, and consider these as D-modules on $\Spec\C[y,y^{-1}]$.  
					\item  In the opposite direction we have a map $c^+ \colon \sP_{!}^{(k)}\to \sP_{*}^{(k)}$, induced by multiplication on the right by $x$; this is also induced by the identity on the local system $\sL^{(k)}$. Similarly, this map becomes an isomorphism if we invert $x$.  
				\end{enumerate}
				Note that the morphisms $c^-$ and $c^+$ shift the good filtration by $\frac{1}{2}$.  
				By \cite[Th. 2.12]{ArapuraTata}, we can identify these maps with the logarithm of the monodromy around the origin, the canonical map from nearby to vanishing cycles and the modified variation map discussed in \cite[\S 2.7]{ArapuraTata}.

				\begin{lemma}\label{lem:end1}
					The algebra $\End(\sP_*^{(k)}\oplus \sP_{!}^{(k)})$  is generated by $c^{\pm}$ subject to the relations 
					\begin{equation}\label{eq:rels1}
						c^+c^-=N\qquad c^-c^+=N \qquad  N^k=0.
					\end{equation}    
					
				\end{lemma}
				\begin{proof}
					By construction, $\Hom( \sP_!^{(k)}, \sM)$ is the kernel of the $k$th power of the logarithm of the monodromy on the stalk of $\sM$ at a generic point, given by the image of $(0,\dots, 0,1)$ in this stalk.  In particular, for $ \Hom( \sP_!^{(k)}, \sP_!^{(k)})$, this is $\C^k$ itself, and the map sending $(0,\dots, 0,1)$ to $(a_1,\dots, a_k)$ is $a_k+a_{k-1}N +\dots a_1N^{k-1}$.  Similarly for $ \Hom( \sP_!^{(k)}, \sP_*^{(k)})$, this stalk is the same, but now the map sending $(0,\dots, 0,1)$ to $(a_1,\dots, a_k)$ is $(a_k+a_{k-1}N +\cdots+ a_1N^{k-1})c^+$.  A symmetric argument holds with $*$ and $!$ reversed.  
				\end{proof}
				
				As noted before, the map $\tau$ induces an isomorphism $\C^2\setminus \{y=0\}\cong \C^2\setminus \{x=0\}.$ We can construct a DQ-module on $\basicW_i\cup \basicW_{i+1}$ glued using $\tau$, and placing $\sP_*^{(k)}$ or $\sP_!^{(k)}$ on each $\basicW_i$.  
				\begin{itemize}
					\item If the two modules are different, i.e. $\sP_*^{(k)}$ on $\basicW_i$ and  $\sP_!^{(k)}$ on $\basicW_{i+1}$ or {\it vice versa} then we use the natural isomorphism induced by swapping the roles of $x$ and $y$.  
					\item If they are the same, i.e.  $\sP_*^{(k)}$ or $\sP_!^{(k)}$ on both $\basicW_i$ and  $\basicW_{i+1}$, then we use the isomorphisms of multiplication by $y^{\pm 1}$ on $\basicW_i$ or equivalently $x^{\pm 1}$ on $\basicW_{i+1}$.
				\end{itemize}

				Iterating this process, we can construct a DQ-module on $\mathfrak{\tilde{Z}}$ associated to a choice of integer $k$ and a map $\wp\colon \Z\to \{*,!\}$, isomorphic to $\sP_{\wp(i)}^{(k)}$ on $\basicW_i$ . To endow this DQ-module with a global $\mathbb{S}$-action, we will need to shift the natural $\mathbb{S}$-action on the local components $\sP_{\wp(i)}^{(k)}$ by a certain amount, determined as follows. 
				
				We can associate to $\sP_*^{(k)}$ a local system on each of the two components of its singular support, $\{x=0\}$ and $\{y=0\}$, both described in terms of the vector space $\vectorspace^{(k)}\cong \mathbb{C}\oplus \mathbb{C}(1)\oplus\cdots \oplus\mathbb{C}(k-1)$; here $(p)$ represents shifting the good filtration/$\bS$-action, though when we discuss Hodge modules, we will want to use it to represent Tate twist by the same amount. At a generic point of $\{x=0\}$, the fiber is $\vectorspace^{(k)}\Big(\frac{1}{2}\Big)$ (so we obtain local systems of weights $0,2,\dots, 2(k-1)$), and at a generic point of  $\{y=0\}$, the fiber is $\vectorspace^{(k)}(1)$; for $\sP_!^{(k)}$, these swap roles.  Thus, in order to have matching $\mathbb{S}$-actions (or equivalently, good filtrations), we need to choose a function $\varsigma\colon \Z\to \frac{1}{2}\Z$ with the property that:
				\begin{equation}\label{eq:varsigma} \varsigma(m+1)=\begin{cases}
						\varsigma(m) & \wp(m)\neq \wp(m+1)\\
						\varsigma(m)-\frac{1}{2} & \wp(m)=\wp(m+1)=*\\
						\varsigma(m)+\frac{1}{2} & \wp(m)=\wp(m+1)=\:\, !\\
					\end{cases} 
				\end{equation}
				and place $\sP_{\wp(i)}^{(k)}(  \varsigma(i))$ on $\basicW_i$.  
				The most important modules constructed this way, denoted $\sP_i^{(k)}$, are given by  
				the functions 
				\begin{equation} \label{eq:wp} \wp(i) =\begin{cases}
						\:\, ! & m>i\\
						\:\,* & m\leq i
					\end{cases} \qquad \zeta(i) =\begin{cases}
						\frac{1}{2}(m-i-1) & m>i\\
						\frac{1}{2}(i-m) & m\leq i
					\end{cases} 
				\end{equation}
				\begin{lemma}\label{lem:projZ}
					The DQ-module $\sP_i^{(k)}$ is the projective cover of $\sL_i$ in the subcategory of $\tdq$ on $\mathfrak{\tilde{Z}}$ where the nilpotent part of the monodromy has length $\leq k$.   
				\end{lemma}
				\begin{proof}
					We can reduce to the case where $i=0$ using the $\Z$ action.  First, we must prove that $\sL_0$ is the unique simple quotient of $\sP_i^{(k)}$.  On $\basicW_0\cup \basicW_1\cong T^*\mathbb{P}^1$, this module is the pushforward $j_! \mathcal{L}^{(k)}$ where $j\colon \C^*\hookrightarrow\mathbb{P}^1$ is the inclusion of the complement of the north and south poles.  This has unique simple quotient given the intermediate extension of the 1-d local system with the standard connection.  This matches the simple $\sL_0$.  Any other simple quotient must be $\sL_m$ with $m\neq 0$.  If $m<0$, this would induce a map on   $\basicW_m$ of $\sP_!^{(k)}$ to the delta function D-module;  similarly, if  $m>0$, it would induce a map on $\basicW_{m+1}$ of $\sP_*^{(k)}$ to the function D-module.  No such map exists, so indeed $\sL_0$ is the unique quotient.  
					
					Now, we need to show it is projective.  Assume that $\ssM$ is an object in $\tdq$ with nilpotent part of the monodromy of length $\leq k$, and that there is a surjective map $\ssM\to  \sL_0$.  First, we note that we can restrict $\ssM $ to $T^*\mathbb{P}^1$ and obtain a D-module on $\mathbb{P}^1$ smooth on $\C^*$.  Since $\sL_0$ is the only simple in $\tdq$ supported on $\mathbb{P}^1$, the local system we obtain on $\C^*$ is regular, so it is on the trivial vector bundle with fiber $\C^d$ with a connection of the form $d-\frac{N'}{x}$ for $N'\colon \C^d\to \C^d$ a nilpotent map, and the map to $\sL_0$ is induced by a map $\phi\colon \C^d\to \C$ whose kernel contains the image of $N'$.  We can lift this up to a map $\sP_i^{(k)}|_{T^*\mathbb{P}^1}\to \ssM |_{T^*\mathbb{P}^1}$ by defining a map $\C^k\to \C^d$ by sending $(0,\dots, 0,1)$ to any vector $v$ with non-zero image under $\phi$, and then extending by the rule that $N^r(0,\dots, 0,1)\mapsto (N')^rv$.  By assumption, $(N')^k=0$, so this sends the standard basis of $\C^k$ to the vectors $(N')^rv$ for $r=0,\dots, k-1$; there a unique linear map satisfying this property. 
					
					Now, we change focus to $\basicW_{-1}$; by the projective property of $\sP_!^{(k)}$, induced the map of local systems on $\mathfrak{X}_{-1}$ extends to a map of $\sP_i^{(k)}|_{\basicW_{-1}}\to \ssM |_{\basicW_{-1}}$.  Applying the same argument again to $\basicW_{-2}$ gives a compatible map $\sP_i^{(k)}|_{\basicW_{-2}}\to \ssM |_{\basicW_{-2}}$.  By induction, we can extend to all $\basicW_{i}$ with $i<0$.  A symmetric argument shows how to extend to $i>1$.  This establishes the result.  
				\end{proof}

				\begin{lemma}
					The stalk of $\sigma_{i'}(\sP_i^{(k)})$ at a generic point in $\mathbb{P}^1$ is $\vectorspace^{(k)}(\frac{|i-i'|+1}{2})$.  
				\end{lemma}
				\begin{proof}
					By their identification with the $*$- and $!$-pushforwards $\sP_{*}^{(k)}$ and $\sP_{!}^{(k)}$ both have stalk $\vectorspace^{(k)}$ on $\mathbb{A}^1-\{0\}$.    
					We need to understand how these correspond to  the generic fiber on $y=0$. This is the same as the vanishing cycles with respect to $x$ at $x=0$.    
					
					In the case $\sP_{*}^{(k)}$, the canonical map induces an isomorphism of these vanishing cycles to $\vectorspace^{(k)}(-\frac{1}{2})$; in  the case $\sP_{!}^{(k)}$, the variation map induces an 
					isomorphism of these vanishing cycles to $\vectorspace^{(k)}(\frac{1}{2})$.
					This makes it clear that on each component, we have shift of the $\bS$-structure on $\vectorspace^{(k)}$, and that this shift is $|i-i'|/2$.
				\end{proof}
				
				Note, this means that $\Hom(\sP_i^{(k)}, \sP_{i'}^{(k)})=\vectorspace^{(k)}(\frac{|i-i'|}{2})$.
				The morphisms $c^{\pm}$ and $N$ induce morphisms of DQ-modules:
				\[N\colon   \sP_i^{(k)} \to \sP_i^{(k)}(1)\qquad c^-\colon \sP_i^{(k)}\to \sP_{i+1}^{(k)}\Big(\frac{1}{2}\Big)\qquad c^+ \colon \sP_{i}^{(k)}\to \sP_{i-1}^{(k)} \Big(\frac{1}{2}\Big)\]
				
				By Lemma \ref{lem:end1}, we have that these morphisms generate the endomorphism algebra $\oplus_{i,j\in \Z}\Hom(\sP_i^{(k)},\sP_j^{(k)})$ subject to the same relations \eqref{eq:rels1}.  That is:
				\begin{lemma}\label{lem:A1-alg}
					The endomorphism algebra $\displaystyle\bigoplus_{i, j\in \Z} \Hom(\sP_i^{(k)}, \sP_j^{(k)})$ is isomorphic to the quotient of the algebra $\gradedHalg_{\C}$ attached to the usual action of $\mathbb{G}_m$ on $\mathbb{A}^1$, modulo the relations $\mathsf{s_i}^k=0$ for all $i$.
				\end{lemma}
				Now, we extend this to the general case.  Given $\mathbf{x}$, we can define a projective by the exterior tensor product  $\sP_{x_1}^{(k)}\boxtimes\cdots \boxtimes \sP_{x_n}^{(k)}$, and consider the action of the torus $T_{\Q}$. We let $\sQ^{(k)}_{\Bx}$ be the unique largest quotient of this exterior product where the monodromy around $T_{\Q}$-orbits is trivial. Concretely, the exterior tensor product above carries an action of $\C[N_1, ...., N_n] = U_\C(\mathfrak{d})$ which can be interpreted as the logarithms of monodromy along orbits of the larger torus $D_{\Q}$. The monodromy is trivial along $T$-orbits if this action factors through the quotient $U_\C(\mathfrak{d}) \to S$. We therefore have \[ Q^{(k)}_{\Bx} = \left( \sP_{x_1}^{(k)}\boxtimes\cdots \boxtimes \sP_{x_n}^{(k)} \right) \otimes_{U_\C(\mathfrak{d})} S. \]
				This quotient has a natural strong $T$-equivariant structure.
				\begin{definition}\label{def:P_x}
					Let $\sP^{(k)}_{\Bx}$ be the Hamiltonian reduction of $\sQ^{(k)}_{\Bx}$  on $\mathfrak{\tilde Y}$; we consider this as a DQ-module.
				\end{definition}
				\begin{lemma}\label{lem:projective-cover}
					The object $\sP^{(k)}_{\Bx}$ is the projective cover of $\sL_{\Bx}$ in the category of DQ-modules with monodromy around $\mathfrak{X}_{\Bx}\cap \mathfrak{X}_{\By}$ unipotent of length $\leq k$ for all $\By$ with $|\Bx-\By|_1=1$.  
				\end{lemma}
				\begin{proof}
					The desired map from $\sP^{(k)}_{\Bx}\to \sL_{\Bx}$ is induced by the simple quotient of $\sP_{x_i}^{(k)}$ for all $i$.  Thus, we need to show the projective property, and the fact that there are no other simple quotients, which we will do by induction on the distance between $\Bx$ and $\By$ in the taxicab metric.  The restriction of $\sP^{(k)}_{\Bx}$ to $T^*\mathfrak{X}_{\Bx}$ is $j_!\mathcal{L}^{(k)}$ where $\mathcal{L}^{(k)}$ is the induced D-module on the complement of the intersection with all other components in $\mathfrak{X}_{\Bx}$.
					There is only one map to $\sL_{\Bx}$ since $\mathcal{L}^{(k)}$ is indecomposable and has unique simple quotient. Since there are no maps of $j_!\mathcal{L}^{(k)}$ to D-modules supported on intersections with other components, we have no maps to $\sL_{\By}$ for $|\Bx-\By|_1=1$.  As in Lemma \ref{lem:projZ}, we can extend this argument to all other $\By$, since the map can't be non-zero on $\By$ if it is zero on all $\By'$ closer to $\Bx$.  This shows that $\sL_{\Bx}$ is the unique simple quotient of $\sP^{(k)}_{\Bx}$.

					Now, let us prove the projective property for $\sP^{(k)}_{\Bx}$.  That is, let  $\sM$ be an object in $\tdq$ with monodromy around $\mathfrak{X}_{\Bx}\cap \mathfrak{X}_{\Bx'}$ for all $|\Bx - \Bx'|_1=1$ unipotent of length $\leq k$, and with a map $\sM \to \sL_{\Bx}$. We wish to show that we have an induced map $\psi\colon \sP^{(k)}_{\Bx}\to \sM$ making the usual diagram commute.
					
					Now, let $\Dol_{\leq p}$ be the union of the subspaces $T^*\mathfrak{X}_{\By}$ for $|\Bx-\By|_1\leq p$.   We will show that the map $\psi$ exists by constructing it inductively of $\Dol_{\leq p}$.

					On $\Dol_{\leq 0}\cong T^*\mathfrak{X}_{\Bx}$, we have that an induced map from $L^{(k)}$ to the local system given by the restriction of $\sM$ to the open orbit in $\mathfrak{X}_{\Bx'}$ by the universal property of $L^{(k)}$, and thus an induced map $\psi|_{\Dol_{\leq 0}}\colon \sigma_{\Bx}(\sP^{(k)}_{\Bx})\cong j_!L^{(k)}\to \sigma_{\Bx}(\sM)$.

					Now, assume that we have defined the map $
					\psi$ on $\Dol_{< p}$, and that $|\By-\Bx|_{1}=p$.  If $U=\Dol_{< p}\cap \mathfrak{X}_{\By}$, then by assumption, we have defined a map $\sigma_{\By}(\sP^{(k)}_{\Bx})|_U\to \sigma_{\By}(\sM)|_U$.  By construction, $\sigma_{\By}(\sP^{(k)}_{\Bx})=i_!(\sigma_{\By}(\sP^{(k)}_{\Bx})|_U)$ where $i\colon U \hookrightarrow\mathfrak{X}_{\By}$ is the inclusion.  Thus, we have a unique induced map $\sigma_{\By}(\sP^{(k)}_{\Bx})\to \sigma_{\By}(\sM)$; applying this for each $\By$ extends this map to $\Dol_{\leq p}$.  This shows that we have the projective property, and we have already confirmed it is the indecomposable projective cover of $\sL_{\Bx}$.
				\end{proof}
				
				\subsection{An equivalence of categories}
				Let $\gradedHalg^{(k)}_\mathbb{K}$ be the quotient of $\gradedHalg^\la_{\mathbb{K}}$ by the 2-sided ideal generated by $\mathsf{s}_i^k$.
				\begin{lemma}\label{lem:end-iso}
					The endomorphism ring $\displaystyle \bigoplus_{\Bx, \By \in \gradedchambers} \Hom_{\dq}(\sP^{(k)}_{\Bx}, \sP^{(k)}_{\By})$ is isomorphic to  $\gradedHalg^{(k)}_{\C}$.  
				\end{lemma}
				\begin{proof}
					This map is induced by sending the  morphism $c^{\pm i}_{\Bx}$ to the  morphism $c^{\pm}$ in the $i$th factor of the exterior product, and $\mathbf{s}_i$ to the endomorphism $N$ of this tensor factor. 
					
					We check that this is well-defined. The linear relations among the variables $\mathbf{s}_i$ in $S$ correspond to the triviality of monodromy along $T$-orbits in $\sP^{(k)}_{\Bx}$. The relations (\ref{wall-cross}) are a consequence of Lemma \ref{lem:end1}, while the relations (\ref{codim1}, \ref{codim2}) follow from the fact that these are the tensor product of endomorphisms of two different tensor factors.  This shows that we have the desired map. 
					
					Now, consider $\Hom(\sP^{(k)}_{\Bx},\sP^{(k)}_{\By})$; this is a quotient of $\C[\mathbf{s}_1,\dots,\mathbf{s}_n]/ (\mathbf{s}_i^k)$, which is the stalk of $  \boxtimes_{i=1}^n \sP^{(k)}_{y_i}$ on the component $\mathfrak{X}_{y_1}\times \cdots \times \mathfrak{X}_{y_n}$.  Killing the monodromy on $T$ gives us the quotient $ S/(\mathbf{s}_i^k)$.  This is generated by the image of $c_{\Bx,\By}$, so our homomorphism is surjective, and the fact that $\gradedHalg_\C$ is free as an $S$-module shows it is also injective.  
				\end{proof}
				Assume that $M$ is a finite dimensional right $\gradedHalg^{\la}_{\C}$-module.  Assume $k$ is chosen large enough that $\mathbf{s}_i^k$ kills $M$ for all $i$.  We can thus write $M$ as a quotient of $\bigoplus_{p=1}^{q}1_{\Bx_p}\gradedHalg^{(k)}_{\Q}$ for some $\Bx_p$, and in fact as the cokernel of a map
				\[ \mapomodules \colon \bigoplus_{r=1}^{s}1_{\By_r}\gradedHalg^{(k)}_{\C}\to \bigoplus_{p=1}^{q}1_{\Bx_p}\gradedHalg^{(k)}_{\C}\]
				induced by elements $a_{rp}\in 1_{\Bx_p}\gradedHalg^{(k)}_{\Q}1_{\By_r}$ of degree $2(\ell_p-\nu_r)$.
				We can also view $\mapomodules$ as a morphism of  DQ-modules 
				\[\mathsf{\tilde d}(\mapomodules)\colon \bigoplus_{r=1}^{s}\sP_{\By_r}^{(k)}\to \bigoplus_{p=1}^{q}\sP_{\Bx_p}^{(k)}\]
				Let $\mathsf{\tilde d}(M)$ denote the cokernel of the map $\mathsf{\tilde d}(\mapomodules)$.  
				Let  $\gradedHalg^{\la,\operatorname{op}}_{\C}\operatorname{-mod}_{o}$ be the category of finite dimensional right modules of $\gradedHalg^{\la}_{\C}$ on which each $\ps_i$ acts nilpotently.
				\begin{proposition}\label{mainthm-noHodge}
					This defines equivalences of categories  \[\mathsf{\tilde d}\colon \gradedHalg^{\la,\operatorname{op}}_{\C}\operatorname{-mod}_{o}\to \tdq\qquad \mathsf{d}\colon \degradedHalg^{\la,\operatorname{op}}_{\C}\operatorname{-mod}_{o}\to \dq.\]  
				\end{proposition}
				\begin{proof}
					By Lemma \ref{lem:Freyd-morita}, the category of DQ-modules which are quotients of a finite sum of the objects $\sP_{\Bx}^{(k)}$ is equivalent to the category $H^{(k),\operatorname{op}}_{\C}\operatorname{-mod}$ via the functor $\mathsf{\tilde d}$. The dimension of $\mathsf{\tilde d}(M)$ under this equivalence is the same as the composition length of $M$, so $M$ is in $\tdq$ if and only if its image is finite dimensional.  Thus, we have an equivalence $H^{(k),\operatorname{op}}_{\C}\operatorname{-mod}_{o}\to \tdq^{\leq k}$ between finite dimensional $H^{(k),\operatorname{op}}_{\C}$-modules and the subcategory $\tdq^{\leq}$ of $\tdq$ where all monodromy has unipotent length $\leq k$.  Since $\gradedHalg^{\la,\operatorname{op}}_{\C}\operatorname{-mod}_{o}$ is the union of the modules factoring through the quotients $H^{(k),\operatorname{op}}_{\C}$ for all $k$, and similarly $\tdq=\cup_{k}\tdq^{\leq k}$, this induces an equivalence $\mathsf{\tilde d}\colon \gradedHalg^{\zeta,\operatorname{op}}_{\C}\operatorname{-mod}_{o}\to \tdq$ as desired.

						The proof for $\mathsf{d}$ is word-for-word identical to that for $\mathsf{\tilde d}$, so we leave the details to the reader.
					\end{proof}

					Since $\tilde{\mathsf{d}}$ is an exact functor, it extends to a (both left and right) derived functor $\tilde{\mathsf{d}}\colon D^{b}( \gradedHalg^{\la,\operatorname{op}}_{\C}\operatorname{-mod}_{o})\to \widetilde{\DQ}$. 
					Combining with Corollary \ref{cor:H-equivalence}, we see our
					version of homological mirror symmetry in this context, as promised in
					the introduction: 
					\begin{theorem} 
						\label{centralcorollary}
						The functor $\Amodule \mapsto \tilde{\mathsf{d}}(\RHom(\tilt^\la_\C,\Amodule))$ defines an equivalence of dg categories $D^b(\Coh_G(\fM_\C)_o)\to \widetilde{\Lcat}$. Similarly, $\mathsf{d}$ defines an equivalence
						$D^b(\Coh(\fM_\C)_o)\to \Lcat$. \end{theorem}
					
					\begin{proof}
						We give the proof for the first equivalence, leaving the second to the reader. We know from Corollary \ref{cor:H-equivalence} that this reduces to showing the derived functor of $\tilde{\mathsf{d}}$ is an equivalence of derived categories $D^{b}( \gradedHalg^{\la,\operatorname{op}}_{\C}\operatorname{-mod}_{o})\to \widetilde{\DQ}$.  Proposition \ref{mainthm-noHodge} show that this functor is an equivalence of categories on the heart of the usual $t$-structure.  It's enough to additionally check that for a set of generating objects, such as the simples $L_\Bx$, the induced map $\Ext^k(L_{\Bx},L_{\By})\to \Ext^k(\tilde{\mathsf{d}}(L_{\Bx}),\tilde{\mathsf{d}}(L_{\By}))$ is an isomorphism for all $k,\Bx,\By$; this isomorphism follows for all other objects by a standard long exact sequence argument.  Thus, to complete the proof, it is enough to show that $\mathsf{\tilde d}$ induces an isomorphism $e_{\By}\gradedHalg_{\la, \C}^!e_{\Bx}\cong \Ext_{\widetilde{\Lcat}}(\tsL_\Bx,\tsL_\By)$.  
						
						Of course, Theorem \ref{mainthm} implies that {\it an} isomorphism between the  corresponding $\Ext$-algebras exists.  We could carefully confirm that this is (up to sign conventions) the same as that of Theorem \ref{mainthm}, but this is not strictly necessary.  The  equivalence of abelian categories of Proposition \ref{mainthm-noHodge} implies this functor induces an isomorphism $\Ext^1(L_{\Bx},L_{\By})\to \Ext^1(\tilde{\mathsf{d}}(L_{\Bx}),\tilde{\mathsf{d}}(L_{\By}))$ for all $\Bx,\By$.  Since $\gradedHalg_{\la, \C}^!$ is generated by elements of degree 1, this implies that the map induced by $\tilde{\mathsf{d}}$ is surjective, and thus an isomorphism since the  dimensions of $e_{\By}\gradedHalg_{\la, \C}^!e_{\Bx}$ and $\Ext_{\widetilde{\Lcat}}(\tsL_\Bx,\tsL_\By)$ in each degree are the same.  In fact, since $\Ext^1_{\widetilde{\Lcat}}(\tsL_\Bx,\tsL_\By)$ is at most 1-dimensional, we must recover the isomorphism of Theorem \ref{mainthm} up to rescaling the image of $d_{\Bx,\By}$ to be a non-zero scalar multiple of $\mathbf{d}_{\Bx,\By}$.  
					\end{proof}
					Note that this also resolves the concern about formality raised below Theorem \ref{mainthm}: since $\gradedHalg^{\la,\operatorname{op}}_{\C}$ is Koszul, the induced dg-algebra structure on the Ext of simples is formal, and this shows that the same holds in $\tdq$.  
					
					\section{Hodge structures}
					\label{sec:Hodge}
					\subsection{Microlocal mixed Hodge modules}

					We will need the notion of a unipotent mixed $\Q$-Hodge structure on $\sigma_{\Bx}(\ssM)$; see \cite{saito2016young} for a reference. 
					``Unipotent'' simply means that the monodromy on every piece of a stratification on which the D-module
					is smooth is unipotent.  Mixed Hodge modules are a very deep subject, but one which we can use in a mostly black-box manner.  The important thing for us is that given a holonomic regular D-module $\mathcal{M}$, a mixed Hodge structure can be encoded as real form and a pair of filtrations, a good filtration (often called the Hodge filtration) and the weight filtration (by submodules) on $\mathcal{M}$.  As discussed previously, we are allowing good filtrations indexed by $\frac{1}{2}\Z$.  
					
					Note that while most references on mixed Hodge modules only consider untwisted D-modules, since a Hodge structure is given by local data, the definition extends to twisted D-modules in an obvious way. We will only be using twists by honest line bundles (as opposed fractional powers), so we have an even easier definition available to us: a mixed/pure Hodge structure on a module $\mathcal{M}$ over differential operators twisted by a line bundle $L$ is the same structure on the untwisted D-module $L^*\otimes \mathcal{M}$. Since we will be working with fixed twists in what follows, we will conceal this choice and simply speak of mixed Hodge modules on $\mathfrak{X}_\Bx$ rather than twisted mixed Hodge modules.
					
					Given an $\mathbb{S}G$-equivariant $\EuScript{\tilde{O}}^\hdef_\phi$-module $\ssM$ in $\dq$, a $\EuScript{\tilde{O}}^\hdef_\phi(0)$-lattice $\ssM(0)$ induces a good filtration on $\sigma_{\Bx}(\ssM)$ for each $\Bx$.  
					
					A $\Q$-form of $\sigma_{\Bx}(\ssM)$ is a perverse sheaf $L$ on $\mathfrak{X}_{\Bx}$ with coefficients in $\Q$ with a fixed isomorphism $L\otimes_{\Q}\C\cong \operatorname{Sol}(\sigma_{\Bx}\ssM)$.  We wish to define a $\Q$-form of $\ssM$ analogously, but we need to think carefully about compatibility between different $\Bx$.
					
					\begin{definition}
						An $\Q$-form for $\ssM \in \SGmod$ is a perverse sheaf $L_{\Bx}$ on $\mathfrak{X}_{\Bx}$ for each $\Bx$ with a fixed isomorphism $L_{\Bx}\otimes_{\Q}\C\cong \operatorname{Sol}(\sigma_{\Bx}\ssM)$ such that the isomorphism \eqref{eq:Fourier1} induces an isomorphism $\mathcal{F}_{\By,\Bx}(L_{\Bx}|_{N_{\Bx,\By}})\cong L_{\By}|_{N_{\By,\Bx}},$ that is, it is compatible with the induced conjugation maps on the solution sheaves $ \operatorname{Sol}(\sigma_{\Bx}\ssM)$.  
						
						A {\bf mixed Hodge structure} on $\ssM$ consists of a lattice $\ssM(0)$, $\Q$-form $L_{\Bx}$ for all $\Bx$ and an increasing weight filtration $W_\bullet$ of $\ssM$ by submodules such that the induced good filtration,  $\Q$-form and weight filtration on $\sigma_{\Bx}(\ssM)$ is a unipotent mixed $\Q$-Hodge structure on this D-module. The real forms are required to be compatible under the isomorphism \eqref{eq:Fourier1}. 
					\end{definition}  
					
					\begin{remark}
						We should note that this definition does not provide any hope of giving a general definition of ``mixed Hodge DQ-modules.'' The space $\algcov$ is a union of cotangent bundles of smooth varieties, with the scaling action  on the cotangent bundle of each component extending to a global action on $\algcov$. We don't know of any similar situation outside the hypertoric case. Generalizing this definition to other cases is, of course, a quite interesting question but not one on which we can provide much insight at the moment. 
					\end{remark}

					\subsection{Hodge structures on projectives}
					
					One natural operation on mixed Hodge DQ-modules is that of {\bf Tate twist}, which shifts the filtrations by $F_i\ssM(k)=F_{i+k}\ssM$ and $W_{i}\ssM(k)=W_{i+2k}\ssM$ for $k\in \frac 12\Z$. Note that defining Tate twists for half-integers requires using good filtrations which are indexed by $k\in \frac 12\Z$, this explains our cryptic introduction of half-integers in earlier sections.  We're only interesting in understanding simple modules up to this operation.  
					We can easily check that:
					\begin{lemma}
						If $\ssM$ is supported on the core $\mathfrak{C}$, then 
						the $D$-module $\sigma_{\Bx}(\ssM)$ is smooth along the orbit stratification of $\mathfrak{X}_{\Bx}$ as a toric variety.
					\end{lemma}
					
					\begin{lemma} \label{MHMgeneration}
						The sheaf $\sL_\Bx$ has a unique mixed Hodge structure whose associated mixed Hodge modules are pure of weight 0. 
					\end{lemma}
					\begin{proof}
						The trivial local system on $\mathfrak{X}_{\Bx}$ has the structure of a variation of Hodge structure which is pure of weight 0.  This is unique by    \cite[Prop. 1.13]{DeligneMonodromie}.
						Of course, any mixed Hodge structure of weight 0 on $\sL_\Bx$ must be induced by this VMHS, which shows uniqueness.  
						
						Thus, we only need to show that the induced lattice $\sL_{\Bx}(0)$, real form, and (trivial) weight filtration induce 
						mixed Hodge structures on the microlocalizations $\sigma_\By(\sL_\Bx)$ for each $\By$. Recall that $\sigma_{\By}(\sL_{\Bx})$ is the pushforward of the trivial line bundle on $\mathfrak{X}_{\Bx}\cap \mathfrak{X}_{\By}$, so the result follows from the compatibility of mixed Hodge structure with pushforward.
					\end{proof}

					Unfortunately, while the Hodge structure on a simple module is unique up to Tate twist, there are ``too many'' different Hodge structures on other objects in $\dq$.  For example, $\sL_\Bx\oplus \sL_\Bx(k)$ has a non-trivial moduli of Hodge structures, induced by the same phenomenon on $\Q\oplus \Q(k)$.  
					
					Thus, we need to find a way of avoiding these sort of deformations of Hodge structure.  We do this by constructing a natural Hodge structure on the modules $\sP^{(k)}_{\Bx}$.  
					
					Recall that we started the contruction of these projectives by considering modules $\sP_*^{(k)},\sP_!^{(k)}$ over $A=\C[x,y,\hbar]$ with the usual Moyal star product.  We make these into mixed Hodge modules on $\mathbb{A}^1$ by  endowing $\dP_*^{(k)}$ with the good filtration such that the image of $\C[x]\subset D_x$ spans $F_{-1} \dP_*^{(k)}$ and $F_{p}\dP_*^{(k)}=F_{p+1}D_x\cdot F_{-1} \dP_*^{(k)}$, and $\dP_!^{(k)}$ with the good filtration such that the image of $\C[x]\subset D_x$ spans $F_{-\nicefrac{1}{2}} \dP_!^{(k)}$ and $F_{p+\nicefrac{1}{2}}\dP_!^{(k)}=F_{p}D_x\cdot F_{-\nicefrac{1}{2}} \dP_!^{(k)}$.  These might seem like slightly strange choices: they are deliberately chosen so that in both cases, the unique simple quotient carries a pure Hodge structure of weight 0.  
					
					Now, we consider Hodge structures on these DQ-modules extending the good filtrations defined above on $\dP_{*}^{(k)}$ and $\dP_{!}^{(k)}$. Their real form is the obvious one where $x$ and $y$ are conjugation invariant; this corresponds to the obvious real form of $L^{(k)}$.  We define the weight filtration on $\dP_{*}^{(k)}$ by 
					\[W_p\dP_{*}^{(k)}=\begin{cases}
						\displaystyle	0 & p<-2k+1. \\
						\displaystyle	D_x (\partial_x x)^{ -p/2}/D_x(\partial_x x)^{ k}& 0> p\geq 2k+1, p\text{ even} \\
						\displaystyle	D_x x (\partial_x x)^{ -(p+1)/2}/D_x(\partial_x x)^{ k}& 0> p\geq 2k+1, p\text{ odd} \\
						\displaystyle	\dP_{*}^{(k)} & p>0
					\end{cases} \]
					and the weight filtration on $\dP_{!}^{(k)}$ analogously, swapping $x$ and $y$.
					
					\begin{lemma}
						These data define mixed Hodge structures on $\dP_{*}^{(k)},\dP_{!}^{(k)}$.
					\end{lemma}
					\begin{proof}
						First, let's consider $\dP_{*}^{(k)}$. By the definition above, $W_p\dP_{*}^{(k)}/W_{p-1}\dP_{*}^{(k)}\cong D_x/D_x x$ if $p$ is even and $0\geq p\geq 2k+1$;  this is equipped the good filtration where the image of $\partial_x^r$ for $r< s$ span $F_{s+\nicefrac{p}{2}}$.  On the other hand, the $V$-filtration of this D-module for the function $x$ has $V^{\ell}$ spanned by $y^r$ for $r\geq -\ell$. Thus, the vanishing cycles $\Phi=\phi(W_p\dP_{*}^{(k)}/W_{p-1}\dP_{*}^{(k)})$ are spanned by the image of 1, i.e. they are 1-dimensional.  Accounting for the shift of good filtration (as in \cite[(2.1.7)]{saito2016young}) they are equipped with the good filtration 
						\[F_{s+\nicefrac{p}{2}}(\Phi)=\begin{cases} \Phi & s\geq 0\\
							0 & s<0\end{cases}\]
						This means that $W_p\dP_{*}^{(k)}/W_{p-1}\dP_{*}^{(k)}$ is isomorphic to
						the usual Tate pure Hodge structure of weight $p$ on $\Q$, pushed forward at the origin $x=0$.  If $p$ is odd, then we have $W_p\dP_{*}^{(k)}/W_{p-1}\dP_{*}^{(k)}\cong D_x/D_x\partial_x$; exactly as above, the generic fiber of this local system has the Tate Hodge structure of weight $p-1$, and so gives a pure Hodge module of weight $p$.  
						
						For $\dP_{!}^{(k)}$, the calculations are the same, but odd and even cases swap roles.  In particular, we see that half-integral filtrations are needed so that we can endow $\Q$ with a Tate Hodge structure of odd weight (i.e. a half integral Tate twist).  
					\end{proof}

					We defined above morphisms $N,c^{\pm}$ between these DQ-modules. These morphisms preserve the mixed Hodge structure up to Tate twist and become morphisms of mixed Hodge modules 
					\[N\colon 	 \sP_*^{(k)} \to \sP_*^{(k)}(1)\qquad c^-\colon \sP_*^{(k)}\to \sP_{!}^{(k)}\Big(\frac{1}{2}\Big)\qquad c^+ \colon \sP_{!}^{(k)}\to \sP_{*}^{(k)} \Big(\frac{1}{2}\Big)\]
					This means that they define Tate elements of the endomorphism algebra $\End(\sP_*^{(k)}\oplus \sP_{!}^{(k)})$, and since they generate, they show that the induced Hodge structure on this algebra is of Tate type agreeing with the grading $\deg(c^{\pm})=1,\deg(N)=2$.
					
					As noted before, the map $\tau$ induces an isomorphism $\C^2\setminus \{y=0\}\cong \C^2\setminus \{x=0\}.$ We can construct a DQ-module on $\basicW_i\cup \basicW_{i+1}$ glued using $\tau$, and placing $\sP_*^{(k)}$ or $\sP_!^{(k)}$ on each $\basicW_i$.  
					\begin{itemize}
						\item If the two modules are different, i.e. $\sP_*^{(k)}$ on $\basicW_i$ and  $\sP_!^{(k)}$ on $\basicW_{i+1}$ or {\it vice versa} then we use the natural isomorphism induced by swapping the roles of $x$ and $y$.  
						\item If they are the same, i.e.  $\sP_*^{(k)}$ or $\sP_!^{(k)}$ on both $\basicW_i$ and  $\basicW_{i+1}$, then we use the isomorphisms of multiplication by $y^{\pm 1}$ on $\basicW_i$ or equivalently $x^{\pm 1}$ on $\basicW_{i+1}$.
					\end{itemize}

					Of course, if we don't include shifts, this gluing will not respect the Hodge structure, so we need to glue these DQ-modules with Tate twists in them.  The functions we Tate twist by have already been constructed in \eqref{eq:varsigma}, based on a choice of which version of the module we will take on each component, expressed by a function $\wp$.  This makes the modules $\sP_i^{(k)}$ into mixed Hodge modules.  
					
					\begin{remark}
						These modules are not projective in the category of mixed Hodge modules (even with appropriate monodromy restrictions) since they don't account for non-Tate extensions.  
					\end{remark}
					
					This induces a Hodge structure on the module  $\sP^{(k)}_{\Bx}$ defined in Definition \ref{def:P_x}, and thus on the endomorphism ring $\bigoplus_{\Bx, \By \in \gradedchambers} \Hom_{\dq}(\sP^{(k)}_{\Bx}, \sP^{(k)}_{\By})$.  
					
					\begin{lemma}\label{lem:hodgeend-iso}
						The Hodge structure on the endomorphism ring $\bigoplus_{\Bx, \By \in \gradedchambers} \Hom_{\dq}(\sP^{(k)}_{\Bx}, \sP^{(k)}_{\By})$ is Tate and matches that constructed from the grading on $H^{(k)}_{\Q}$.  
					\end{lemma}
					
					\subsection{The category of mixed Hodge modules } 
					Now, we wish to establish a graded version of the equivalence of Theorem \ref{centralcorollary}.  As discussed above, looking at all mixed Hodge structures on DQ-modules results in ``too many'' objects; in particular, the graded lift of a projective object will not be projective in the category of all mixed Hodge structures on objects in $\tdq$, which is not the behavior we expect from adding a grading to a ring.  In more categorical terms, the functor of forgetting Hodge structure is not a ``degrading'' functor.  
					
					Thus, we will consider objects in $\tdq$ with a more restricted set of Hodge structures, only those which arise as a quotient of the objects $\sP_{\Bx}^{(k)}$; it's worth noting that while these objects have a projective property in $\tdq$ (subject to a restriction on monodromy), they are not projective amongst mixed Hodge DQ-modules with this monodromy.  The important effect this has is that it forces the local systems on the open part of $\mathfrak{X}_{\Bx}$ to be Tate as mixed Hodge structures;  typically, the structures we wish to avoid will not have this property.
					
					\begin{definition}\label{def:mhm}
						We let $\mhm$ and $\tmhm$ be the categories of mixed Hodge DQ-modules in $\dq$ and $\tdq$ which are quotients of a sum of the form $\bigoplus_{p=1}^{q} \sP_{\Bx_p}^{(k)}(\ell_p)$ for some $k\geq 0,\ell_p\in \frac{1}{2}\Z$ and $\{\Bx_1,\dots, \Bx_{q}\}\subset \gradedchambers$. 
					\end{definition}

					We let $\MHM$ and $\tMHM$ be the standard dg-enhancements of the derived categories $D^b(\mhm)$ and $D^b(\tmhm)$ (the quotient of the dg-category of all complexes modulo that of acyclic complexes).
					
					Now, assume that $M$ is a finite dimensional graded right $\gradedHalg^{\zeta}_{\Q}$-module.  Recall that $M(\ell)$ denotes $M$ with the grading shifted down by $\ell$.  Assume $k$ is chosen large enough that $\mathbf{s}_i^k$ kills $M$ for all $i$.  We can thus write $M$ as a quotient of $\bigoplus_{p=1}^{q}1_{\Bx_p}H^{(k)}_{\Q}(2\ell_p)$ with $\Bx_p$ and $\ell_p$ as above, and in fact as the cokernel of a map
					\[\mapomodules \colon \bigoplus_{r=1}^{s}1_{\By_r}H^{(k)}_{\Q}(2\nu_r)\to \bigoplus_{p=1}^{q}1_{\Bx_p}H^{(k)}_{\Q}(2\ell_p)\]
					induced by elements $a_{rp}\in 1_{\Bx_p}H^{(k)}_{\Q}1_{\By_r}$ of degree $2(\ell_p-\nu_r)$.
					We can also view $A$ as a morphism of Hodge DQ-modules 
					\[\tilde{\mathsf{m}}(\mapomodules)\colon \bigoplus_{r=1}^{s}\sP_{\By_r}^{(k)}(\nu_r)\to \bigoplus_{p=1}^{q}\sP_{\Bx_p}^{(k)}(\ell_p)\]
					Let $\tilde{\mathsf{m}}(M)$ denote the cokernel of the map $\tilde{\mathsf{m}}(\mapomodules)$.  
					Let  $\gradedHalg^{\zeta,\operatorname{op}}_{\Q}\operatorname{-gmod}_{o}$ be the category of finite dimensional graded right modules of $\gradedHalg^{\zeta}_{\Q}$.  
					\begin{theorem}\label{mainthm-Hodge}
						This defines equivalences of categories  $$ \tilde{\mathsf{m}} \colon \gradedHalg^{\zeta,\operatorname{op}}_{\Q}\operatorname{-gmod}_{o}\to \tmhm \qquad \mathsf{m} \colon \degradedHalg^{\zeta,\operatorname{op}}_{\Q}\operatorname{-gmod}_{o}\to \mhm$$ sending grading shift $(\ell)$ to the Tate twist $(\nicefrac{\ell}{2})$. 
					\end{theorem}
					
					\begin{proof}
						If $f\colon M\to M'$ is a homogeneous map of modules, the construction of $\tilde{\mathsf{m}}(f)$ by presenting $M$ and $M'$ as cokernels proceeds exactly as in the proof of Proposition \ref{mainthm-noHodge}, as does the proof that this functor is fully faithful. 
						
						The only point where we need a bit more care is in the proof of essential surjectivity.  By definition, any module $\sM$ in $\tmhm$ is a quotient of $\tilde{\mathsf{m}}(\sP_0)$ for some $\sP_0$.  Thus, we need to show that the kernel $\sK$ is also an object in $\tmhm$.  The object $\sK$ has a largest semi-simple quotient, i.e. its cosocle.  This is a finite sum of objects of the form $\sL_{\By_r}(\nu_r)$.  This shows that $\sK$ is generated by the images of maps (of DQ-modules, ignoring Hodge structure) from $\sP^{(k)}_{\By_r}$ for $r=1,\dots, s$.  Note that $\Hom(\sP^{(k)}_{\By_r},\sK)$ carries a mixed Hodge structure which is a subobject of $\Hom(\sP^{(k)}_{\By_r},\tilde{\mathsf{m}}(\sP_0))$, the former has Tate type since the latter does as well.  Thus, there is a module $M_1$ such that 
						\[\tilde{\mathsf{m}}(\sP_1)=\bigoplus_{r=1}^s \Hom(\sP^{(k)}_{\By_r},\sK)\otimes_{\Q} \sP^{(k)}_{\By_r}\]
						as mixed Hodge DQ-modules; of course, the image of the induced map $\tilde{\mathsf{m}}(\sP_1)\to \tilde{\mathsf{m}}(\sP_0)$ is exactly $\sK$, and so $\sM=\tilde{\mathsf{m}}(M)$ where $M$ is the cokernel of the map $\sP_1\to \sP_0$. This completes the proof that $\tilde{\mathsf{m}}$ is an equivalence. The second equivalence is proven the same way. 
					\end{proof}

					Analogous to the proof of Theorem \ref{centralcorollary}, we have the following: 
					\begin{corollary} \label{cor:maincorollary}
						There are equivalences of categories \[ D^b(\Coh_{\mathbb{G}_m\times G}(\fM_\Q)_o)\to \tMHM.\] and \[ D^b(\Coh_{\mathbb{G}_m}(\fM_\Q)_o)\to \MHM \] 
      sending grading shift $(\ell)$ to the Tate twist $(\nicefrac{\ell}{2})$. 			\end{corollary}
					
					We conclude with a few questions raised by this result. Under our equivalence, the $\mathbb{G}_m$-action on $\fM_\C$ corresponds to the weight grading on $\MHM$. This action, which dilates the symplectic form, is key to the enumerative geometry of hypertoric varieties. Indeed, the symplectic structure on $\fM_\C$ implies that the non-equivariant quantum connection of $\fM_\C$ is essentially trivial. Its $\mathbb{G}_m$-equivariant version, on the other hand, is the hypergeometric system studied in \cite{McS}. The same is true for more general symplectic resolutions : for instance, the $\mathbb{G}_m$-equivariant quantum connection of the Springer resolution is the decidedly non-trivial affine KZ connection \cite{BMO}. Our result thus suggests that the mirror description of these connections can be approached via microlocal Hodge structures.
					
					We also note that whereas the left-hand side of both of our equivalences is a geometrically defined category, the right-hand sides are defined by picking certain generators inside the ambient category of deformation-quantization modules. This is in contrast to the equivalence proven in the sequel to this paper \cite{GMW}, which equates coherent sheaves on $\fM_\C$ with the wrapped Fukaya category of its mirror. A more direct geometric definition of $\MHM$ and its grading, in particular, would be of great interest. 
					
					\bibliography{gen}
					\bibliographystyle{amsalpha}
				\end{document}